\documentclass[10pt, leqno]{amsart}
\usepackage{latexsym, amsfonts,mathrsfs, amsmath, amssymb, amscd, epsfig}
\usepackage{mathrsfs}
\usepackage[usenames,dvipsnames]{xcolor}
\usepackage{amsthm}
\usepackage{array}
\usepackage{psfrag}
\usepackage{bm}
\usepackage{setspace}
\usepackage{soul}
\usepackage{graphicx}
\usepackage{pdfpages}

\numberwithin{equation}{section}

\newtheorem{lemma}{Lemma}[section]
\newtheorem{theorem}[lemma]{Theorem}
\newtheorem{proposition}[lemma]{Proposition}
\newtheorem{definition}[lemma]{Definition}
\newtheorem{remark}[lemma]{Remark}
\newtheorem{problem}[lemma]{{Problem}}

\newtheorem*{proclaim}{Proclamation}

\newtheorem*{question}{Question}
\theoremstyle{definition}

\def\beq#1\eeq{\begin{equation}#1\end{equation}}
\def\balign #1 #2 \ealign{\begin{aligned} #1 #2  \end{aligned} }

\DeclareMathOperator*{\esssup}{ess\,sup}

\newcommand \bvphi{\bar{\varphi}}
\newcommand \bPhi{\bar{\Phi}}

\begin{document}

\title[Supersonic flows of the Euler-Poisson system]{Supersonic flows of the Euler-Poisson system in three-dimensional cylinders}

\author{Myoungjean Bae}
\address{Department of Mathematical Sciences, KAIST, 291 Daehak-Ro, Yuseong-Gu, Daejeon 34141, Republic of Korea}
\email{mjbae@kaist.ac.kr}

\author{Hyangdong Park}
\address{School of Mathematics, Korea Institute for Advanced Study (KIAS), 85 Hoegiro, Dongdaemun-gu, Seoul, 02455, Republic of Korea}
\email{hyangdong@kias.re.kr}

\begin{abstract}{We prove the unique existence of three-dimensional supersonic solutions to the steady Euler-Poisson system in cylindrical nozzles when prescribing the velocity, entropy, and the strength of electric field at the entrance. We first establish the unique existence of irrotational solutions in a cylindrical nozzle with an arbitrary cross section by extending the results of \cite{bae2021three} with an aid of weighted Sobolev norms. Then, we establish the unique existence of axisymmetric solutions with nonzero vorticity in a circular cylinder. In particular, we construct a solution with a nonzero angular momentum density. Therefore this is truly a three-dimensional flow in the sense that the system cannot be reduced to a two-dimensional system via a stream function formulation. The main idea is to reformulate the system via the Helmholtz decomposition method, and to employ the iteration method. Several technical issues, including the issue of corner singularities in a Lipschitz domain, are carefully addressed.}
\end{abstract}




\maketitle

\tableofcontents
\newcommand {\R}{\mathbb{R}}
\newcommand {\gam}{\gamma}
\newcommand \Gam{\Gamma}
\newcommand \trho{\tilde{\rho}}
\newcommand \vphi{\varphi}
\newcommand \ol{\overline}
\newcommand \Om{\Omega}
\newcommand \om{\omega}
\newcommand \der{\partial}
\newcommand \tx{\text}
\newcommand \mcl{\mathcal}
\newcommand \eps{\varepsilon}
\newcommand \mfrak{\mathfrak}
\newcommand \Gamw{\Gamma_{{\rm w}}}
\newcommand \rx{{\rm x}}

\newcommand \rhos{\rho_s}
\newcommand \alp{\alpha}


\section{Introduction}\label{section:introduction}
Given a positive function $b:\R^3\rightarrow \R$, the steady Euler-Poisson system
\begin{equation}\label{E-S}
\left\{
\begin{split}
&\mbox{div}(\rho{\bf u})=0,\\
&\mbox{div}(\rho{\bf u}\otimes{\bf u})+\nabla p=\rho\nabla\Phi,\\
&\mbox{div}(\rho\mathcal{E}{\bf u}+p{\bf u})=\rho{\bf u}\cdot\nabla\Phi,\\
&\Delta\Phi=\rho-b,
\end{split}
\right.
\end{equation}
describes a hydrodynamic model of semiconductor devices or plasmas. The function $b$ represents the density of fixed, positively charged background ions.
And, the functions $\rho$, ${\bf u}$, $p$, and $\mathcal{E}$ represent the macroscopic particle electron density, velocity, pressure, and the total energy, respectively.
The function $\Phi$ represents the electric potential generated by the Coulomb force of particles.

In this paper, we assume that $p$ and $\mathcal{E}$ are given by
\begin{equation}
\label{expression of p in rho}
p(\rho, S)=S\rho^{\gamma}\quad\mbox{and}\quad
\mathcal{E}(\rho,\lvert{\bf u}\rvert,S)=\frac{\lvert{\bf u}\rvert^2}{2}+\frac{S\rho^{\gamma-1}}{\gamma-1}
\end{equation}
for a positive function $S$ called the {\emph {entropy}}, and a constant $\gamma>1$ called the {\emph {adiabatic exponent}}.
Then the system \eqref{E-S} can be rewritten as follows:
\begin{equation}\label{E-B}
\left\{
\begin{split}
&\mbox{div}(\rho{\bf u})=0,\\
&\mbox{div}(\rho{\bf u}\otimes{\bf u})+\nabla p=\rho\nabla\Phi,\\
&\mbox{div}(\rho{\bf u}B)=\rho{\bf u}\cdot\nabla\Phi,\\
&\Delta\Phi=\rho-b
\end{split}
\right.
\end{equation}
for the {\emph{Bernoulli function}} $B$ defined by
\begin{equation*}\label{Ber-f}
B(\rho,\lvert{\bf u}\rvert,S):=\frac{\lvert{\bf u}\rvert^2}{2}+\frac{\gamma }{\gamma-1}S\rho^{\gamma-1}.
\end{equation*}
If $(\rho,{\bf u}, p)\in C^1$ and $\Phi\in C^2$ solve \eqref{E-B}, then
\begin{equation*}
\rho{\bf u}\cdot\nabla S=\rho{\bf u}\cdot\nabla \mathcal{K}=0
\end{equation*}
for
\begin{equation}
\label{equation for K}
\mathcal{K}:=B-\Phi.
\end{equation}
We call the function $\mathcal{K}$ the {\emph{pseudo-Bernoulli invariant}}.
For simplicity, we assume that
\begin{equation*}
\mathcal{K}=0.
\end{equation*}

For an irrotational flow ($\nabla\times {\bf u}={\bf 0}$),
the system \eqref{E-B} can be simplified as
\begin{equation}
\label{potential flow model}
\left\{\begin{split}
&{\rm div}\left(\rho(\Phi, \nabla \vphi)\nabla\vphi\right)=0,\\
&\Delta \Phi=\rho(\Phi, \nabla\vphi)-b
\end{split}\right.
\end{equation}
with $\rho(\Phi, \nabla\vphi):=\left[\frac{\gam-1}{\gam S_0}(\Phi-\frac 12\lvert\nabla\vphi\rvert^2)\right]^{\frac{1}{\gam-1}}$ for a constant $S_0>0$.
In  \cite{bae2016subsonic} and \cite{bae2021three}, this potential flow system in three dimensional domains has been studied.  Especially, the very first result on the existence of a three-dimensional supersonic ($\lvert\nabla\vphi\rvert>\sqrt{\gamma S_0\rho^{\gam-1}(\Phi, \nabla\vphi)}$) solution to the system in a rectangular nozzle is proved in \cite{bae2021three}.
\smallskip

The main purpose of this paper is to generalize the result of \cite{bae2021three}. More precisely, the first part of this paper is devoted to prove the existence of a supersonic solution to the potential flow system in a three dimensional cylinder with an arbitrary cross-section.  For that purpose, we shall newly define weighted Sobolev spaces.

  The second part of this paper is devoted to extend the above result to the full Euler-Poisson system \eqref{E-B}.
  Namely, we construct three-dimensional axisymmetric supersonic ($\lvert{\bf u}\rvert>\sqrt{\gamma S\rho^{\gam-1}}$) solutions to \eqref{E-B} with $\nabla \times {\bf u}\neq {\bf 0}$. In particular, we shall consider a flow with nonzero angular momentum density(or equivalently, nonzero swirl). Clearly, this feature indicates that the flow is truly three dimensional in the sense that the system \eqref{E-B} cannot be simplified as a two dimensional system via a stream function formulation. In the spirit of a Helmholtz decomposition, we express the velocity vector field as ${\bf u}=\nabla\vphi+\nabla\times {\bf V}$ for a scalar function $\vphi$ and a vector field ${\bf V}$. And, we attempt to construct a three-dimensional supersonic solution to the system \eqref{E-B} as a small perturbation from the case of ${\bf V}={\bf 0}$, which is studied in the first part of the paper. The main idea is to reformulate the system into a second order hyperbolic-elliptic coupled system and two transport equations via the method of Helmholtz decomposition, and to employ the results obtained in the study of the potential flows.

 In \cite{bae2021structural, bae2014subsonic, bae20183}, the method of Helmholtz decomposition is used to construct various types of multi-dimensional solutions to the system \eqref{E-B}. In particular, an explicit representation for a velocity field ${\bf u}$ of an axisymmetric flow with nonzero angular momentum density is introduced in \cite{bae20183} for the first time to prove the existence of a three-dimensional subsonic ($\lvert{\bf u}\rvert<\sqrt{\gamma S\rho^{\gam-1}}$) solution to \eqref{E-B}. The representation is as follows:
 \begin{equation}
 \label{velocity in HD}
  {\bf u}=\nabla\vphi+\nabla\times (h{\bf e}_r+\phi {\bf e}_{\theta})
 \end{equation}
 for axisymmetric functions $\vphi$, $h$ and $\phi$.

The main difference of our work from \cite{bae20183} is that we seek for a solution that satisfies the inequality $\lvert{\bf u}\rvert>\sqrt{\gamma S\rho^{\gam-1}}$. In this case, the main challenge is to solve a quasi-linear system consisting of a second order hyperbolic differential equation and a second order elliptic differential equations weakly coupled in a three dimensional domain. To the best of our knowledge, this is the first result on the three-dimensional supersonic flows for the steady Euler-Poisson system with nonzero vorticity.

The rest of the paper is organized as follows.
In Section \ref{Sec-Main}, the main problems and the main theorems are addressed.
In Section \ref{Sec-Irr}, we prove the existence of a supersonic solution to the system \eqref{potential flow model} in a three dimensional cylindrical domain with an arbitrary smooth cross-section.
In Section \ref{Sec-HD}, we prove the
existence of a supersonic solution to the system \eqref{E-B} in a three dimensional cylindrical domain with a circular cross-section by using the Helmholtz decomposition and applying the result obtained from Section \ref{Sec-Irr}.

\section{Main Theorems}\label{Sec-Main}
Fix a constant $\gam>1$.
Suppose that $(\bar{\rho}, \bar u, \bar p,\bar \Phi)(x_1)$ solves the Euler-Poisson system \eqref{E-B} with the function $b$ being given by
\begin{equation*}
  b=b_0
\end{equation*}
for some constant $b_0>0$.
And, let us set $\bar E(x_1):=\bar{\Phi}'(x_1)$. If the inequalities $\bar\rho>0$, $\bar u>0$ and $\bar p>0$ hold, then we can rewrite the system \eqref{E-B}  for $(\bar{\rho}, \bar u, \bar p,\bar E)$ as
\begin{equation}
\label{one-re}
\bar\rho'=\frac{\bar E \bar \rho}{\gamma S_0\bar \rho^{\gamma-1}-\frac{J_0^2}{\bar \rho^2}},\quad \bar E'=\bar \rho-b_0,\quad
\bar u=\frac{J_0}{\bar \rho},\quad \bar p=\bar \rho^{\gamma}S_0
\end{equation}
for some constants $J_0>0$ and $S_0>0$. Let us set
\begin{equation*}
\rho_s:=\left(\frac{J_0^2}{\gamma S_0}\right)^{\frac{1}{\gamma+1}}.
\end{equation*}
Then it can be directly checked that the flow governed by the solution  $(\bar{\rho}, \bar u, \bar p,\bar \Phi)$ is {\emph{supersonic}}($\bar u>\sqrt{\gam S_0\bar{\rho}^{\gam-1}}$) if and only if $\bar{\rho}<\rho_s$, and {\emph{subsonic}}($\bar u<\sqrt{\gam S_0\bar{\rho}^{\gam-1}}$) if and only if $\bar{\rho}>\rho_s$. In this paper, we assume that
\begin{equation}
\label{condition of b0}
0<b_0<\rho_s.
\end{equation}
See Figure \ref{figurec} for a phase plane of solution curves under the assumption of \eqref{condition of b0} in Appendix \ref{appendix-C}.

\medskip
For the rest of the paper, we fix the constants $\gam>1$, $J_0>0$ and $S_0>0$.
Under the condition of \eqref{condition of b0}, let us consider the initial value problem
\begin{equation}\label{one-re-uE}
\left\{\begin{split}
&\bar{\rho}'=\frac{\bar E\bar{\rho}}{\gamma S_0\bar{\rho}^{\gamma-1}-\frac{J_0^2}{\bar{\rho}^2}}\\
& \bar E'=\bar{\rho}-b_0
\end{split}\right.\quad\mbox{with}\quad (\bar{\rho}, \bar E)(0)=(\rho_0, E_0).
\end{equation}
Since we are interested in supersonic flows, we shall fix the initial data $(\rho_0, E_0)$ in the set $(0, \rhos)\times \R$. The unique solvability of the initial value problem \eqref{one-re-uE} is already well known so we state the following lemma without a proof.

\begin{lemma}[One-dimensional supersonic solutions ({\cite[Lemma 1.1]{bae2021structural}, \cite[Section 3.1]{luo2012transonic}})] \label{Lem1}
For any given constant $\bar{\delta}>0$ sufficiently small, there exists a constant $\bar L>0$ depending on $(\gamma, J_0, S_0, b_0, \rho_0, E_0,\bar{\delta})$ so that the initial value problem \eqref{one-re-uE} has a unique smooth solution $(\bar \rho,\bar E)(x_1)$ on $[0,\bar L]$ with satisfying that
\begin{equation*}
\bar{\delta}\le \bar \rho(x_1)\le \rho_s-\bar{\delta}  \quad\tx{for $0\le x_1\le \bar{L}$}.
\end{equation*}
Note that the above inequality is equivalent to
\begin{equation*}
  1+\hat{\delta}\le \frac{\bar u(x_1)}{\sqrt{\gam S_0\bar{\rho}^{\gam-1}(x_1)}}<\infty\quad\tx{for $0\le x_1\le \bar L$}
\end{equation*}
for some constant $\hat{\delta}>0$.
\end{lemma}

Hereafter, we shall fix constants $\gam>1$, $J_0>0$, $S_0>0$, $\rho_0\in(0,\rhos)$ and $E_0\in \R$. Let $(\bar{\rho}, \bar{E})$ be the solution to the initial value problem \eqref{one-re-uE}, and let $\bar u$ be given by
\begin{equation*}
  \bar u=\frac{J_0}{\bar{\rho}}.
\end{equation*}
And, let us define a function $\bar{\Phi}$ by
\begin{equation}
\label{def-Phi0}
  \bar{\Phi}(\rx)=\int_0^{x_1} \bar{E}(t)\,dt+B_0
  \quad\text{with $B_0= \frac{J_0^2}{2\rho_0^2}+\frac{\gam S_0\rho_0^{\gam-1}}{\gam-1}$}
\end{equation}
for $\rx=(x_1,x_2, x_3)\in \R^3$.
\begin{definition}[Background solutions]
\label{definition:background sol}
Given constants $\gam>1$, $J_0>0$ and $S_0>0$, we shall call $(\bar{\rho}, \bar u, S_0, \bar{\Phi})$ the background solution to the steady Euler-Poisson system \eqref{E-B} associated with $(\rho_0, E_0)\in(0, \rhos)\times \R$.
\end{definition}
Here, the pressure $p$ is given by $p=S_0\bar{\rho}^{\gam}$. Since we reformulate the system \eqref{E-B} later so that the new system contains a transport equation for the entropy function $S=\frac{p}{\rho^{\gam}}$, we include $S_0$ in Definition \ref{definition:background sol} as a part of a background solution.
\medskip

The main purpose of this work is to investigate the structural stability of the background solution in three dimensional nozzles by the following regimes:
\begin{itemize}
\item[(i)] First of all, we shall consider the structural stability of the background solution in a three dimensional cylindrical nozzle with an arbitrary cross-section for the potential flow model.

\item[(ii)] Secondly, we shall consider the structural stability in a three dimensional circular cylinder with respect to axisymmetric perturbations for the full Euler-Poisson system.
\end{itemize}

\subsection{Potential flows}
\label{subsection:intro-part 1}
In $\R^2$, fix an open, connected and bounded domain $\mcl{D}$ with a smooth boundary $\der \mcl{D}$.
For a constant $L\in(0,\bar{L}]$, define a three dimensional cylinder $\Omega_L$ by
\begin{equation}\label{def-Omega}
\Omega_L:=\left\{\rx=(x_1,\rx')\in\mathbb{R}^3:\mbox{ }  0<x_1<L,\, \rx'=(x_2,x_3)\in\mathcal{D}\right\}.
\end{equation}
We shall denote the entrance, wall, and the exit of $\Omega_L$ by
\begin{equation*}
	\Gamma_0:=\partial\Omega_L\cap\{x_1=0\},\quad \Gamma_{\rm w}:=(0,L)\times\partial\mathcal{D},\quad	\Gamma_L:=\partial\Omega_L\cap\{x_1=L\}.
	\end{equation*}
Finally, let ${\bf n}_w$ represent the inward unit normal vector field on $\Gamw$.
\medskip

If the velocity field ${\bf u}$ is given as
\begin{equation*}
{\bf u}=\nabla\varphi\quad\mbox{in}\quad \Omega_L
\end{equation*}
for a scalar function $\varphi=\varphi(\rx)$, and if $(\rho, {\bf u}, p, \Phi)$ is a classical solution to the system \eqref{E-B} in $\Om_L$, then it is well known that the entropy $S=\frac{p}{\rho^{\gam}}$ is globally a constant so that we can set as
\begin{equation*}
  S\equiv S_0\quad\tx{in $\Om_L$}
\end{equation*}
for some constant $S_0>0$. Moreover, the solution can be given by solving the following system, which is called {\emph{the potential flow model}} of the Euler-Poisson system:
\begin{equation}\label{back-HD}
\left\{\begin{split}
&\mbox{div}(\rho(\Phi,\nabla\varphi)\nabla\varphi)=0,\\
&\Delta\Phi=\rho(\Phi,\nabla\varphi)-b
\end{split}\right.
\end{equation}
for $\rho$ defined by
\begin{equation}\label{def-H}
\rho(z,{\bf q}):
=\left[\frac{\gamma-1}{\gamma S_0}\left(z-\frac{1}{2}\lvert{\bf q}\rvert^2\right)\right]^{\frac{1}{\gamma-1}}\quad\mbox{for}\quad z\in\mathbb{R},\, {\bf q}\in\mathbb{R}^3.
\end{equation}

Given constants $\gam>1$, $J_0>0$, $S_0>0$ and $(\rho_0, E_0)\in (0, \rhos)\times \R$, let $(\bar{\rho}, \bar u, S_0, \bar{\Phi})$ be the background solution associated with $(\rho_0, E_0)$ in the sense of Definition \ref{definition:background sol}. And, let us define
\begin{equation}\label{def-phi0}
\begin{split}
&\bar{\varphi}(\rx):=\int_0^{x_1}\bar u(t)dt\\
\end{split}
\end{equation}
for $\rx=(x_1, x_2, x_3)\in\Omega_L.$
Then $(\vphi, \Phi)=(\bvphi,\bPhi)$ solves the system \eqref{back-HD}, and satisfies the following boundary conditions:
\begin{equation}\label{back-com}
\begin{split}
\varphi=0,\quad \partial_{x_1}\varphi=\frac{J_0}{\rho_0}(=:u_0),\quad \partial_{x_1}\Phi=E_0\quad\mbox{on}\quad\Gamma_0,\\
\partial_{{\bf n}_w}\varphi=0,\quad\partial_{{\bf n}_w}\Phi=0\quad\mbox{on}\quad\Gamma_{\rm w}.
\end{split}
\end{equation}
In addition, it satisfies the inequality
\begin{equation}\label{back-super}
\lvert\nabla\bvphi\rvert^2> \gam S_0\rho^{\gam-1}(\bPhi, \nabla\bvphi)
\quad\tx{in $\ol{\Om_L}$}.
\end{equation}

First, we set up a problem to find a solution $(\vphi, \Phi)$ to the system \eqref{back-HD} as a small perturbation of $(\bvphi,\bPhi)$ in $\Om_L$.
\begin{problem}\label{Irr-EP-Prob}
Fix functions $b\in C^2(\overline{\Omega_L})$, $u_{\rm en}\in C^3(\overline{\Gamma_0})$, $E_{\rm en}\in C^4(\overline{\Gamma_0})$, and a function $E_{\rm ex}\in  C^4(\overline{\Gamma_L})$.
In addition, given a small constant $\bar{\epsilon}\in(0,\frac{1}{4}]$, let us set
\begin{equation}
\label{definition-bd-port-ent}
      \Gamma_{0}^{\bar{\epsilon}}:=\{(0, \rx')\in\overline{\Gamma_0}:\mbox{dist}(\rx',\der\mcl{D})\le\bar{\epsilon}\}.
    \end{equation}

Let us assume that the functions $(b, E_{\rm en}, E_{\rm ex})$ satisfy the  compatibility conditions:
\begin{equation}\label{Thm-com-Irr}
\begin{split}
&\partial_{x_1}b=0,\quad E_{\rm en}-E_0=0\quad\mbox{on}\quad\Gamma_{0}^{\bar{\epsilon}},\\
&\partial_{{\bf n}_w}(b, u_{\rm en}, E_{\rm en}, E_{\rm ex})={\bf 0}\quad\mbox{on}\quad\Gamma_{\rm w}.
\end{split}
\end{equation}
Under this assumption, find a solution $(\varphi,\Phi)$ to the Euler-Poisson system in $\Omega_L$ with satisfying the following properties:
\begin{itemize}
\item[(i)] (The boundary conditions)
	\begin{equation}\label{Irr-Phy-bd}
	\begin{split}
	\varphi=0,\quad \partial_{x_1}\varphi=u_{\rm en},\quad\partial_{x_1}\Phi=E_{\rm en}\quad&\mbox{on}\quad\Gamma_0,\\
	\partial_{{\bf n}_w}\varphi=0,\quad\partial_{{\bf n}_w}\Phi=0\quad&\mbox{on}\quad\Gamma_{\rm w},\\
	\der_{x_1}\Phi=E_{\rm ex}\quad&\mbox{on}\quad\Gamma_L.
	\end{split}
	\end{equation}	
\item[(ii)] (Positivity of the density and the velocity along $x_1$-direction)
\begin{equation}\label{positive-phi}
\rho(\Phi,\nabla\varphi)>0\quad\mbox{and}\quad \der_{x_1}\varphi>0\quad \mbox{ in}\quad\overline{\Omega_L}.
\end{equation}
\item[(iii)] (Supersonic speed) ${\lvert\nabla\varphi\rvert}>c(\Phi,\nabla\varphi)$ in $\overline{\Omega_L}$ for the sound speed  $c(z,{\bf q})$ defined by
 \begin{equation}\label{Sound1}
 c(z,{\bf q}):=\sqrt{(\gamma-1)\left(z-\frac{1}{2}\lvert{\bf q}\rvert^2\right)}
  \end{equation}
  for $(z,{\bf q})\in\mathbb{R}\times\mathbb{R}^3$ with $z-\frac{1}{2}\lvert{\bf q}\rvert^2>0$.
\end{itemize}
\end{problem}

Before we state our main theorem, we first introduce two {\emph{weighted Sobolev norms}}.
\begin{definition}[A weighted Sobolev norm]
\label{definition-weighted Sobolev norm}
For each $t\in(0,L)$, let us define $\Om_{t}$ by
\begin{equation*}
  \Omega_t:=\left\{\rx=(x_1,\rx')\in\mathbb{R}^3:\mbox{ }  0<x_1<t,\, \rx'\in\mathcal{D}\right\}(=\Om_L\cap\{x_1<t\}).
\end{equation*}
For each $k\in \mathbb{N}$, we define a weighted $H^k$-norm by
\begin{equation*}
  \|\phi\|_{H^k_*(\Om_L)}:=\|\phi\|_{H^{k-1}(\Om_L)}+\sup_{0<d<L}d^{\frac 12}\|D^k\phi\|_{L^2(\Om_{L-d})}.
\end{equation*}
And, we define $H^k_*(\Om_L)$ to be
\begin{equation*}
  H^k_*(\Om_L):=\{\phi\in H^{k-1}(\Om_L)\cap H^k_{\rm loc}(\Om_L): \|\phi\|_{H^k_*(\Om_L)}<\infty\}.
\end{equation*}
\end{definition}


\begin{definition}[A weighted Sobolev norm with involving a time-like variable ]
\label{definition:space involving time}
\quad\\

(1) For a function $\phi:\ol{\Om_L}\longrightarrow \R$, we define two weighted norms involving a time-like variable:
\begin{itemize}
\item[(i)] Regarding the function $\phi$ as a map $x_1\in (0, L)\mapsto \phi(x_1, \cdot)\in L^2(\mcl{D})$, we define
\begin{equation*}
\|\phi\|_{L_*^{\infty}((0,L);L^2(\mcl{D}))}:=\sup_{0<d<L}d^{\frac 12} \esssup_{0<s<L-d}\|\phi(s,\cdot)\|_{L^2(\mcl{D})}.
\end{equation*}

\item[(ii)] For each $k\in \mathbb{N}$, we define
\begin{equation*}
\|\phi\|_{L_*^{\infty}((0,L);H^k(\mcl{D}))}:=
\sum_{j=0}^{k-1} \esssup_{0<s<L}\|D_{\rx'}^j\phi(s,\cdot)\|_{L^2(\mcl{D})}
+\|D_{\rx'}^k\phi\|_{L_*^{\infty}((0,L);L^2(\mcl{D}))}.
\end{equation*}
\end{itemize}
\medskip

(2) For a fixed constant $m\in \mathbb{N}$, we define
    \begin{equation*}
    \begin{split}
    &\|\phi\|_{\mcl{W}^{m,\infty}_{\mcl{D}}(0,L)}
      :=\sum_{j=0}^m \|\der_1^j\phi\|_{L^{\infty}((0,L);H^{m-j}(\mcl{D}))};\\
    &\|\phi\|_{\mcl{W}^{m,\infty}_{*, \mcl{D}}(0,L)}
      :=\sum_{j=0}^m \|\der_1^j\phi\|_{L_*^{\infty}((0,L);H^{m-j}(\mcl{D}))}.
     \end{split}
    \end{equation*}
Finally, we define two vector spaces $\mcl{W}^{m,\infty}_{\mcl{D}}(0,L)$ and $\mcl{W}^{m,\infty}_{*, \mcl{D}}(0,L)$ by
\begin{equation*}
  \begin{split}
  &\mcl{W}^{m,\infty}_{\mcl{D}}(0,L):=\left\{\phi:\Om_L\rightarrow \R\;\middle\vert\; \begin{split}&D^j\phi\in L^{\infty}((0,L);L^2(\mcl{D}))\,\,\tx{for $j=0,1,\cdots, m$},\\
  &\|\phi\|_{\mcl{W}^{m,\infty}_{\mcl{D}}(0,L)}<\infty\end{split}\right\},\\
&\mcl{W}^{m,\infty}_{*,\mcl{D}}(0,L):=\left\{\phi:\Om_L\rightarrow \R\;\middle\vert\; \begin{split}&D^j\phi\in L^{\infty}((0,L);L^2(\mcl{D}))\,\,\tx{for $j=0,1,\cdots, m-1$},\\
&D^m\phi\in L^{\infty}_{\rm loc}((0,L);L^2(\mcl{D})),\\
  &\|\phi\|_{\mcl{W}^{m,\infty}_{*,\mcl{D}}(0,L)}<\infty\end{split}\right\}.
  \end{split}
\end{equation*}
Clearly, $\|\cdot\|_{\mcl{W}^{m,\infty}_{\mcl{D}}(0,L)}$ and $\|\cdot\|_{\mcl{W}^{m,\infty}_{*, \mcl{D}}(0,L)}$ are norms, thus $\mcl{W}^{m,\infty}_{\mcl{D}}(0,L)$ and $\mcl{W}^{m,\infty}_{*,\mcl{D}}(0,L)$ are normed vector spaces.
\end{definition}

\begin{theorem}[Potential flows]
\label{Irr-MainThm}
Fix constants $\gam>1$, $J_0>0$, $S_0>0$. And, fix $E_0$ as
\begin{equation}
\label{special condition}
E_0=0.
\end{equation}
 For fixed constants $b_0$ and $\rho_0$ satisfying the condition \eqref{condition of b0} and the inequality $0<\rho_0<\rho_s$, respectively, let $(\bar\rho, \bar E, \bar u, \bar p)$ be the solution to \eqref{one-re} with the initial condition $(\bar\rho, \bar E)(0)=(\rho_0, E_0)$. And, let $(\bvphi, \bPhi)$ be given by \eqref{def-phi0} and \eqref{def-Phi0}, respectively.
 \medskip

For a fixed small constant $\bar{\epsilon}\in(0, \frac 14]$, let functions $(b, u_{\rm en}, E_{\rm en},E_{\rm ex})$ be given with satisfying the compatibility conditions \eqref{Thm-com-Irr} stated in Problem \ref{Irr-EP-Prob}. And,
let us set
\begin{equation}
\label{definition of sigma}
\begin{split}
\sigma(b,u_{\rm en},E_{\rm en},E_{\rm ex})
&:=\|b-b_0\|_{C^2(\overline{\Omega_L})}+\|u_{\rm en}-u_0\|_{C^3(\overline{\Gamma_0})}\\&\quad+\|E_{\rm en}\|_{C^4(\overline{\Gamma_0})}+\|E_{\rm ex}-\bar E(L)\|_{C^4(\overline{\Gamma_L})}.
\end{split}
\end{equation}
For a fixed constant $\bar{\delta}>0$, let $\bar L$ be from Lemma \ref{Lem1}.
Then, there exists a constant $\hat{L}\in(0,\bar L]$ depending on $(\gamma, J_0, S_0, b_0, \rho_0, \bar{\delta})$  so that the following properties hold:
For any given $L\in(0, \hat{L}]$, one can fix a constant $\sigma_p>0$ sufficiently small depending on $(\gamma, J_0,S_0,b_0,\rho_0, \bar{\delta},L,\bar{\epsilon})$ so that if the inequality
\begin{equation*}
\sigma(b,u_{\rm en},E_{\rm en},E_{\rm ex})\le\sigma_p
\end{equation*}
holds,
then Problem \ref{Irr-EP-Prob} has a unique  solution $U=(\vphi, \Phi)\in [H^4_{*}(\Om_L)\cap \mcl{W}^{4,\infty}_{*,\mcl{D}}(0,L)]\times H^4_{*}(\Om_L)$ that satisfies the estimate
 \begin{equation}\label{up-est}
 \begin{split}
&\|\Phi-{\bPhi}\|_{H^4_*(\Omega_L)}
+\|\varphi-\bvphi\|_{H^4_*(\Omega_L)}
+\|\varphi-\bvphi\|_{\mcl{W}^{4,\infty}_{*,\mcl{D}}(0,L)}\\
&\le C\sigma(b,u_{\rm en},E_{\rm en},E_{\rm ex})
\end{split}
\end{equation}
for a constant $C>0$, fixed depending only on $(\gamma, J_0,S_0,b_0,\rho_0,\bar{\delta},L,\bar{\epsilon})$.
\end{theorem}

\begin{remark}\label{Rem}
The main difference of this work from \cite{bae2021three} is that the domain that we consider is a three dimensional nozzle with an arbitrary cross-section while a rectangular nozzle is considered in \cite{bae2021three}. But instead, we add the condition \eqref{special condition} for the following reason.
As we shall see later, a priori $H^3$ estimates of $(\vphi, \Phi)$ can be given by the arguments given in \cite{bae2021structural, bae2021three}. But we need more careful approach to establish a global $H^4$ estimate of $(\vphi, \Phi)$ in this work.

In \cite{bae2021three}, we establish an a priori $H^4$-estimate of $\Phi$ by applying the method of reflection to the elliptic equation $\Delta \Phi=\rho(\Phi, \nabla\vphi)-b$. To put it simply,
one can achieve a global $H^4$-estimate of $\Phi$ by using its local even reflection about the wall boundary $\Gamw$, which is a rectangular shell in the Euclidean coordinates. Then, a global $L^2$ estimate of $D^4\vphi$ is given by a bootstrap argument and the standard hyperbolic estimate method. This is a direct extension of the result from \cite{bae2021structural}. In this paper, however, we consider the case where the wall boundary $\Gam_w$ is non-flat thus the reflection method used in [5] is no longer applicable. To overcome this difficulty, we add the condition \eqref{special condition} and the compatibility condition for $E_{\rm en}$ on $\Gam_0$ as stated in \eqref{Thm-com-Irr}, which enables us to use a local even reflection of $\Phi$ about $\Gam_0$ so that we can establish an a priori $H^4$-estimate of $\Phi$ up to $\Gam_0$. And, a global $H^4$-estimate of $\Phi$ up to the exit boundary $\Gam_L$ is achieved by using the weighted Sobolev norm $\|\cdot\|_{H^4_*(\Om_L)}$ given by Definition \ref{definition-weighted Sobolev norm}. Using a weighted Sobolev norm is one of the new features in this paper. Then the remaining challenge is to show that the a priori estimate of $\|\Phi\|_{H^4_*(\Om_L)}$ yields an estimate of $\|\vphi\|_{H^4_*(\Om_L)}+\|\vphi\|_{\mcl{W}^{4,\infty}_{*,\mcl{D}}(0,L)}$.
\end{remark}


\subsection{Nonzero vorticity flows}

\begin{definition}
\label{definition-axixymmetry}
Let $(x_1,r,\theta)$ denote the cylindrical coordinates of $\rx=(x_1,x_2,x_3)\in\mathbb{R}^3$, i.e.,
\begin{equation*}
(x_1,x_2,x_3)=(x_1,r\cos\theta,r\sin\theta),\quad r\ge0,\quad\theta\in\mathbb{T}
\end{equation*}
for a one-dimensional torus  $\mathbb{T}$ with period $2\pi$.
\begin{itemize}
\item[(i)] A function $f(\rx)$ is said to be {\emph{axisymmetric}} if its value is independent of $\theta$, i.e., $f(\rx)=g(x_1,r)$ for some function $g:\R^2\rightarrow \R$.
\item[(ii)] A vector-valued function ${\bf F}(\rx)$ is said to be {\emph{axisymmetric}} if ${\bf F}$ can be represented as $${\bf F}(\rx)=F_{x_1}{\bf e}_{x_1}+F_r{\bf e}_r+F_{\theta}{\bf e}_{\theta}$$
for axisymmetric functions $F_{x_1}, F_r$ and $F_{\theta}$. In the above, the vectors ${\bf e}_{x_1}$, ${\bf e}_r$ and ${\bf e}_{\theta}$ are given by
\begin{equation*}
{\bf e}_{x_1}=(1,0,0),\quad {\bf e}_r=(0,\cos\theta,\sin\theta),\quad{\bf e}_{\theta}=(0,-\sin\theta, \cos\theta)
\end{equation*}
in the Euclidean coordinate system.
\end{itemize}
\end{definition}

\begin{problem}\label{EP-Prob}
Let $\Om_L$ be given by Definition \ref{def-Omega} for
\begin{equation*}
\mcl{D}:=\{\rx'=(x_2, x_3)\in \R^2: |\rx'|<1\}.
\end{equation*}
Fix an axisymmetric function $b\in C^2(\overline{\Omega_L})$. And, fix functions $u_{\rm en}\in C^3(\overline{\Gamma_0})$, $v_{\rm en}, w_{\rm en}, S_{\rm en}, E_{\rm en}\in C^4(\overline{\Gamma_0})$, and a function $E_{\rm ex}\in  C^4(\overline{\Gamma_L})$ so that they are all independent of $\theta\in \mathbb{T}$. In other words, they are all radial functions.
\smallskip

For a fixed constant $\bar{\epsilon}\in(0,\frac{1}{4}]$, let $\Gam_0^{\bar{\epsilon}}$ be given by \eqref{definition-bd-port-ent}. In this case, we can simply write as
\begin{equation*}
  \Gam_0^{\bar{\epsilon}}=\{(0, \rx'): 1-\bar{\epsilon}\le |\rx'|\le 1\}.
\end{equation*}
And, let us assume that the following compatibility conditions hold:
\begin{itemize}
\item[-] $\displaystyle{\der_r b=0\quad\tx{on $\ol{\Gam_0}\cap\ol{\Gamw}$},\quad \der_{x_1} b=0\quad\tx{on $\Gam_0^{\bar{\epsilon}}$ }}$;
 \smallskip
\item[-] $\displaystyle{\der_r u_{\rm en}=0\quad\tx{on $\ol{\Gam_0}\cap\ol{\Gamw}$}}$;
\item[-] $\displaystyle{v_{\rm en}=0\quad\tx{on $\Gam_0^{\bar{\epsilon}}\cup \{{\bf 0}\} $}};$
  \smallskip
\item[-] $\displaystyle{\der_r E_{\rm en}=0\quad\tx{on $\ol{\Gam_0}\cup\ol{\Gamw}$},\quad E_{\rm en}=E_0\quad\tx{on $\Gam_0^{\bar{\epsilon}}$}};$
  \smallskip
\item[-] $\displaystyle{\der_r E_{\rm ex}=0\quad\tx{on $\ol{\Gam_0}\cup\ol{\Gamw}$}};$
  \smallskip
\item[-] $\displaystyle{w_{\rm en}=0\quad\tx{on $ \Gam_0^{\bar{\epsilon}} $},\quad \der_r^k w_{\rm en}({\bf 0})=0\quad\tx{for $k=0,1,2,3$}};$
    \smallskip
   \item[-] $\displaystyle{S_{\rm en}=S_0\quad\tx{on $ \Gam_0^{\bar{\epsilon}} $},\quad \der_r^kS_{\rm en}({\bf 0})=0\quad\tx{for $k=1,2,3$}}$.
\end{itemize}

\medskip

Under the conditions stated in the above, find a solution $U=({\bf u},\rho, p,\Phi)$ to the full Euler-Poisson system \eqref{E-B} in $\Omega_L$ with satisfying the following properties:
\begin{itemize}
\item[(i)] (The boundary conditions for $U$)
	\begin{equation}\label{Phy-bd}
	\begin{split}
	{\bf u}=u_{\rm en}{\bf e}_{x_1}+v_{\rm en}{\bf e}_r+w_{\rm en}{\bf e}_{\theta},\quad \frac{p}{\rho^{\gamma}}=S_{\rm en},\quad\partial_{x_1}\Phi=E_{\rm en}\quad&\mbox{on}\quad\Gamma_0,\\
	{\bf u}\cdot{\bf e}_r=0,\quad\partial_{r}\Phi=0\quad&\mbox{on}\quad\Gamma_{\rm w},\\
	\der_{x_1}\Phi=E_{\rm ex}\quad&\mbox{on}\quad\Gamma_L.
	\end{split}
	\end{equation}	
\item[(ii)] 
The inequalities
$\displaystyle{\rho>0\,\,\mbox{and}\,\, {\bf u}\cdot{\bf e}_{x_1}>0\,\,\mbox{hold in}\,\,\overline{\Omega_L}}$.
\item[(iii)] For the sound speed  $c(z,{\bf q})$ defined by \eqref{Sound1}, it holds that
$${\lvert{\bf u}\rvert}>c(\Phi,{\bf u})\quad\tx{in}\quad \overline{\Omega_L}.$$
\end{itemize}
\end{problem}

\begin{remark}
If ${\bf u}$ is $C^1$ and axisymmetric in $\Omega_L$, then it must hold that
\begin{equation*}
{\bf u}\cdot{\bf e}_r={\bf u}\cdot{\bf e}_\theta\equiv0\quad\mbox{on}\quad\overline{\Omega_L}\cap\{r=0\}.
\end{equation*}
Therefore we require for the boundary data $(v_{\rm en},w_{\rm en})$ to satisfy the compatibility conditions $v_{\rm en}({\bf 0})=0$ and $w_{\rm en}({\bf 0})=0$ as stated in Problem \ref{EP-Prob}.
\end{remark}


\begin{theorem}[Nonzero vorticity flows]\label{MainThm}
Fix constants $\gam>1$, $J_0>0$, $S_0>0$, and fix $E_0$ as \begin{equation*}
  E_0=0.
\end{equation*}

For fixed constants $b_0$ and $\rho_0$ satisfying the condition \eqref{condition of b0} and the inequality $0<\rho_0<\rho_s$, respectively, let $(\bar\rho, \bar E, \bar u, \bar p)$ be the solution to \eqref{one-re} with the initial condition $(\bar\rho, \bar E)(0)=(\rho_0, E_0)$. And, let $(\bar\rho, \bar u, S_0, \bPhi)$ be the background solution associated with $(\rho_0, E_0)$ in the sense of Definition \ref{definition:background sol}.
Given functions $(b, u_{\rm en}, v_{\rm en},w_{\rm en},S_{\rm en}, E_{\rm en},E_{\rm ex})$ with satisfying all the compatibility conditions stated in Problem \ref{EP-Prob},
let us set
\begin{equation*}
\label{definition of sigma2}
\begin{split}
&\sigma(b,u_{\rm en},v_{\rm en},w_{\rm en},S_{\rm en},E_{\rm en},E_{\rm ex})\\
&:=\|b-b_0\|_{C^2(\overline{\Omega_L})}+\|u_{\rm en}-u_0\|_{C^3(\overline{\Gamma_0})}\\&\quad+\|(v_{\rm en},w_{\rm en},E_{\rm en}, S_{\rm en}-S_0)\|_{C^4(\overline{\Gamma_0})}+\|E_{\rm ex}-\bar E(L) \|_{C^4(\overline{\Gamma_L})}.
\end{split}
\end{equation*}

Fix a constant $\bar{\delta}>0$, and let $\bar{L}$ be given from Lemma \ref{Lem1}. Then, there exists a constant $L^{\ast}\in(0,\bar{L}]$ depending on $(\gamma, J_0, S_0, b_0, \rho_0, \bar{\delta})$  so that, for any $L\in(0, L^{\ast}]$, Problem \ref{EP-Prob} is well-posed provided that $\displaystyle{\sigma(b,u_{\rm en},v_{\rm en},w_{\rm en},S_{\rm en},E_{\rm en},E_{\rm ex})}$ is sufficiently small. More precisely, for any fixed $L\in (0, L^{\ast}]$, one can fix a constant $\sigma_1>0$ sufficiently small depending on $(\gamma, J_0,S_0,b_0,\rho_0, \bar{\delta},L,\bar{\epsilon})$ so that if the inequality
\begin{equation*}
\sigma(b,u_{\rm en},v_{\rm en},w_{\rm en},S_{\rm en},E_{\rm en},E_{\rm ex})\le\sigma_1
\end{equation*}
holds, then Problem \ref{EP-Prob} has a unique axisymmetric solution $({\bf U}, \Phi)$ with
${\bf U}=({\bf u},\rho,p)$ that satisfies the estimate

 \begin{equation}\label{up-est-vol}
 \begin{split}
&\|{\bf U}-(\bar{u}{\bf e}_1,\bar{\rho},S_0\bar{\rho}^{\gam})\|_{H^3_*(\Omega_L)}
+\|{\bf U}-(\bar{u}{\bf e}_1,\bar{\rho},S_0\bar{\rho}^{\gam})\|_{\mcl{W}^{3,\infty}_{*, \mcl{D}}(0,L)}+\|\Phi-{\bPhi}\|_{H^4_*(\Omega_L)}\\
&\le C\sigma(b,u_{\rm en},v_{\rm en},w_{\rm en},S_{\rm en},E_{\rm en},E_{\rm ex})
\end{split}
\end{equation}
for a constant $C>0$ depending only on $(\gamma, J_0,S_0,b_0,\rho_0, \bar{\delta},L,\bar{\epsilon})$.
\end{theorem}

\begin{remark}
\label{remark-embedding}
By applying the generalized Sobolev inequality, it can be directly checked that
\begin{equation*}
H^{3+k}_*(\Om_L)\subset C^{k}(\ol{\Om_L})\cap C^{k+1}(\ol{\Om_L\cap\{x_1< L-d\}})\quad\tx{for any $d\in(0,L)$ with $k=0,1$}.
\end{equation*}
Therefore it follows from the estimate \eqref{up-est-vol} that the solution $({\bf U}, \Phi)$ given in Theorem \ref{MainThm}  is a classical solution of Problem \ref{EP-Prob}.
\end{remark}

Once we prove Theorem \ref{Irr-MainThm}, we can prove Theorem \ref{MainThm} via the method of Helmholtz decomposition. So we first prove Theorem \ref{Irr-MainThm} in the next section, which is the main part of this paper, then prove Theorem \ref{MainThm} in Section \ref{Sec-HD}.

\section{Potential flows (Proof of Theorem \ref{Irr-MainThm})}\label{Sec-Irr}
Throughout Section \ref{Sec-Irr}, we assume that $\Om_L$ is given by \eqref{def-Omega} with an arbitrary cross-section $\mcl{D}$ with its boundary $\der\mcl{D}$ being smooth.


\subsection{Outline of the proof}
Suppose that $(\vphi, \Phi)$ solves the system \eqref{back-HD}, and set
\begin{equation*}
  (\psi, \Psi):=(\vphi, \Phi)-(\bvphi, \bPhi)\quad\tx{in $\ol{\Om_L}$}.
\end{equation*}
First of all, we rewrite Problem \ref{Irr-EP-Prob} in terms of $(\psi, \Psi)$.

\begin{definition}
\label{definition:iteration}
Let $z\in \R$, ${\bf p}=(p_1, p_2,p_3)\in \R^3$ and ${\bf q}=(q_1, q_2,q_3)\in \R^3$.
\begin{itemize}
\item[(i)] Principal coefficients $a_{ij}$ for $i,j=1,2,3$:
\begin{itemize}
\item[-] Define the sound speed $c(z, {\bf q})$ by
\begin{equation*}
  c( z, {\bf q}):=\sqrt{(\gam-1)\left(\bPhi+z-\frac 12\lvert\nabla\bvphi+{\bf q}\rvert^2\right)}.
\end{equation*}
\item[-] Let $\delta_{ij}$ denote the Kronecker delta, that is, $\delta_{ij}=0$ if $i\neq j$, and $\delta_{ij}=1$ if $i=j$.
\end{itemize}
For each $i,j=1,2,3$, we define $a_{ij}(z, {\bf q})$ by
\begin{equation}
\label{coeff aij for potential}
  a_{ij}( z, {\bf q})=\frac{c^2(z, {\bf q})\delta_{ij}-(\der_i\bvphi+q_i)(\der_j\bvphi+q_j)}{c^2( z, {\bf q})-(\der_1\bvphi+q_1)^2}
\end{equation}
for $\bvphi$ given by \eqref{def-phi0}.\\

Note that the value of each $a_{ij}(z, {\bf q})$ varies depending on $\rx\in\Om_L$ as well as $(z, {\bf q})$. Later, we shall fix $(z, {\bf q})$ as functions of $\rx\in\Om_L$, and regard that $a_{ij}$ is evaluated at $\rx\in\Om_L$. Nonetheless, we write as $a_{ij}(z, {\bf q})$ without indication of its dependence on $\rx\in\Om_L$ for the purpose of simplicity.\\

Note that $a_{11}\equiv1$. Furthermore, due to \eqref{back-super}, there exists a constant $\lambda_0>0$ such that
  \begin{equation}\label{sign-aii}
  -a_{ii}(0, {\bf 0})>\lambda_0 \quad\tx{in $\ol{\Om_L}$ for  $i=2,3$}.
  \end{equation}

\item[(ii)] Smooth coefficients for lower order derivative terms:

Define $B(z, {\bf p}, {\bf q})$ by
\begin{equation*}
B(z,{\bf p},{\bf q})=\frac{(\nabla\bPhi+{\bf p})\cdot(\nabla\bvphi+{\bf q})}{c^2(z, {\bf q})-(\partial_1\bvphi+q_1)^2}.
\end{equation*}
We define $\bar a_1$, $\bar b_1$ and $\bar b_2$ by
\begin{align}
\label{def-a-bar}
\bar a_1&:=\partial_{q_1}B(0,{\bf 0},{\bf 0})=
\frac{(\partial_1\bPhi)\left[\gamma(\partial_1\bvphi)^2+c^2(0,{\bf 0})\right]}{[c^2(0,{\bf 0})-(\partial_1\bvphi)^2]^2},\\
\label{def-bar-b12}
\begin{split}
\bar b_1&:=\partial_{p_1}B(0,{\bf 0},{\bf 0})=\frac{\partial_1\bvphi}{(\partial_1\bvphi)^2-c^2(0,{\bf 0})},\\
\bar b_2&:=\partial_{z}B(0,{\bf 0},{\bf 0})=
 \frac{-(\gamma-1)(\partial_1\bPhi)(\partial_1\bvphi)}{[c^2(0,{\bf 0})-(\partial_1\bvphi)^2]^2}.
  \end{split}
\end{align}

For $\rho$ given by \eqref{def-H}, we define $\bar h_1$ and $\bar h_2$ by
\begin{equation}\label{def-bar-h}
\begin{split}
&\bar h_1:=\partial_z\rho(\bPhi,\nabla\bvphi)=\frac{1}{\gam S_0}\bar{\rho}^{2-\gam},\\
&\bar h_2:=\partial_{q_1}\rho(\bPhi,\nabla\bvphi)
=-\frac{\bar u}{\gam S_0}\bar{\rho}^{2-\gam}.
\end{split}
\end{equation}
Note that $\bar{a}_1$, $\bar{b}_1$, $\bar b_2$, $\bar{h}_1$ and $\bar{h}_2$ are smooth in $\ol{\Om_L}$ due to Lemma \ref{Lem1}.

\item[(iii)] Non-homogeneous terms $(f_1,f_2)$:

 For $i=1$ and $2$, we define $f_1(z, {\bf p} ,{\bf q})$ and $f_2(z, {\bf q})$ by
\begin{equation}\label{fs for potential}
  \begin{split}
  &f_1(z, {\bf p} ,{\bf q})=-\left[B(tz,t{\bf p},t{\bf q})\right]_{t=0}^1+\bar a_1q_1+\bar b_1p_1+\bar b_2 z,\\
  &f_2(z, {\bf q})=[\rho(\bPhi+tz, \nabla\bvphi+t{\bf q})]_{t=0}^1-(b-b_0)-\bar h_1z-\bar h_2q_1
  \end{split}
\end{equation}
with $[g(t)]_{t=0}^1:=g(1)-g(0)$.
\end{itemize}
\end{definition}
A straightforward computation shows that
if the condition
\begin{equation*}
  (\gamma-1)(\Phi-\frac{1}{2}\lvert\nabla\varphi\rvert^2)-\lvert\nabla\varphi\rvert^2\ne 0
\end{equation*}
holds in $\Om_L$, then we can rewrite Problem \ref{Irr-EP-Prob} as a nonlinear boundary value problem for $(\psi, \Psi)$ as follows:
\begin{equation}\label{system for potential perturbation}
  \begin{cases}
  &\displaystyle\sum_{i,j=1}^3 {a_{ij}(\Psi, \nabla\psi)}\der_{ij}\psi+\bar a_1\der_1\psi+{\bar b_1}\der_1\Psi+{\bar b}_2\Psi=f_1(\Psi, \nabla\Psi, \nabla\psi)\\
  &\displaystyle\Delta \Psi-\bar h_1\Psi-\bar h_2\der_1\psi=f_2(\Psi, \nabla\psi)
  \end{cases}\quad\tx{in $\Om_L$},
\end{equation}

\begin{equation}
\label{bcs for potential perturbation}
  \begin{split}
  \psi=0,\quad\partial_{x_1}\psi=u_{\rm en}-u_0,\quad
	\partial_{x_1}\Psi=E_{\rm en}\quad&\mbox{on}\quad\Gamma_0,\\
\partial_{{\bf n}_w}\psi=0,\quad\partial_{{\bf n}_w}\Psi=0\quad&\mbox{on}\quad\Gamma_{{\rm w}},\\
\der_{x_1}\Psi=E_{\rm{ex}}-\bar{E}(L)\quad&\mbox{on}\quad\Gamma_L.
  \end{split}
\end{equation}
Therefore, it suffices to solve the boundary value problem of \eqref{system for potential perturbation} and \eqref{bcs for potential perturbation} to prove Theorem \ref{Irr-MainThm}. To solve this boundary value problem by iterations, we introduce an iteration set by using the weighted norms introduced in Definitions \ref{definition-weighted Sobolev norm} and \ref{definition:space involving time}.
\medskip

For a constant $\delta>0$ to be fixed later, we define two sets $\mcl{H}_{\delta}^P$ and $\mcl{I}_{\delta}^E$ by
\begin{equation}
\label{definition-potentials}
\begin{split}
  \mcl{H}_{\delta}^P&:=\left\{\psi\in H^4_{\ast}(\Om_L)\cap \mcl{W}^{4,\infty}_{*, \mcl{D}}(0,L)\;\middle\vert\;\begin{split}
							&\|\psi\|_{H^4_*(\Om_L)}+\|\psi\|_{\mcl{W}^{4,\infty}_{*, \mcl{D}}(0,L)}\le \delta,\\
  							&\der_{{\bf n}_w}\psi=0\tx{ on $\Gamw$,}\\
							&\der_{x_1}^{k-1}\psi=0\tx{ on $\Gamma_0^{\bar{\epsilon} }$}\\
							&\tx{in the trace sense for $k=1,3$}\end{split}\right\},\\
 \mcl{I}_{\delta}^E&:=\left\{\Psi\in H^4_{\ast}(\Om_L)\;\middle\vert\;\begin{split}
  								&\|\Psi\|_{H^4_*(\Om_L)}\le \delta,\\
  &\der_{{\bf n}_w}\Psi=0\tx{ on $\Gamw$},\\
								&\partial^k_{x_1}\Psi=0\tx{ on $\Gamma_0^{\bar{\epsilon}}$ in the trace sense for $k=1,3$}
								\end{split}\right\}.				
\end{split}				
\end{equation}
And, we define an iteration set $ \mcl{J}_{\delta}$ by
\begin{equation}
\label{definition of iteration set}		
  \mcl{J}_{\delta}:=\mcl{H}^P_{\delta}\times \mcl{I}^E_{\delta}.
\end{equation}
For each $P=(\psi, \Psi)\in \mcl{J}_{\delta}$, let us set
\begin{equation}
\label{definition-P-norms}
\begin{split}
&\|P\|_{H^4_*(\Om_L)}:=\|\psi\|_{H^4_*(\Om_L)}+\|\Psi\|_{H^4_*(\Om_L)},\\
&\|P\|_{*}:=\|\psi\|_{H^4_*(\Om_L)}+\|\psi\|_{\mcl{W}^{4,\infty}_{*, \mcl{D}}(0,L)}+\|\Psi\|_{H^4_*(\Om_L)}.
  \end{split}
\end{equation}

By applying Sobolev inequality, Morrey's inequality, Arzel\`{a}-Ascoli theorem, Rellich's theorem and the weak compactness property of a Hilbert space, we can prove the following lemma, which is used in Section \ref{subsection-proof of theorem-potential}.
\begin{lemma}\label{lemma-general property of iter set-Irr}
For any fixed constant $\delta>0$, the iteration set $\mcl{J}_{\delta}$ satisfies the following properties:
\begin{itemize}
\item[(a)] The set $\mcl{J}_{\delta}$ is convex.
\item[(b)] The set $\mcl{J}_{\delta}$ is compact in $[(H^2(\Om_L)\cap C^{1,\frac 14}(\ol{\Om_L}))]^2$.
\item[(c)] The set $\mcl{J}_{\delta}$ is weakly compact in $[H^4_*(\Om_L)]^2$ in the following sense: For any given sequence $\{(\psi_j, \Psi_j)\}_{j\in \mathbb{N}}$, one can take a subsequence $\{(\psi_{j_k}, \Psi_{j_k})\}$ and an element $(\psi_{\infty}, \Psi_{\infty})\in [H^4_*(\Om_L)]^2$ so that
    \begin{itemize}
    \item[-] For each $m=0, 1, 2, 3$, the sequence $\{D^m(\psi_{j_k}, \Psi_{j_k})\}$ weakly converges to $D^m(\psi_{\infty}, \Psi_{\infty})$ in $L^2(\Om_L)$;
    \item[-] For any $d\in(0,L)$, the sequence $\{D^4(\psi_{j_k}, \Psi_{j_k})\}$ weakly converges to $D^4(\psi_{\infty}, \Psi_{\infty})$ in $L^2(\Om_L\cap\{x_1<L-d\})$;
    \item[-] $\displaystyle{\|(\psi_{\infty}, \Psi_{\infty})\|_{H^4_*(\Om_L)}\le \sup_{j\in \mathbb{N}}\|(\psi_{j}, \Psi_{j})\|_{H^4_*(\Om_L)}}$.
    \end{itemize}
\end{itemize}
\end{lemma}

Since the definition of the iteration set $\mcl{J}_{\delta}$ involves the norm $\|\cdot\|_{\mcl{W}^{4,\infty}_{*, \mcl{D}}(0,L)}$, we need to establish a compactness property of the space $\mcl{W}^{4,\infty}_{*, \mcl{D}}(0, L)$. Accordingly, we need the following lemma.
\begin{lemma}
\label{lemma-weak star limit}
If a sequence of functions $\{\psi_n\}_{n\in \mathbb{N}}$ is bounded in $H^4_*(\Om_L)\cap \mcl{W}^{4,\infty}_{*, \mcl{D}}(0,L)$, then it has a subsequence that converges to a function $\psi_{\infty}\in H^4_*(\Om_L)\cap \mcl{W}^{4,\infty}_{*, \mcl{D}}(0,L)$ in $C^1(\ol{\Om_L})\cap C^{2}(\ol{\Om_L}\setminus \Gam_L)$. Furthermore, the limit function $\psi_{\infty}$ satisfies the estimate
\begin{align}
\label{weak_cpt_H4}
  &\|\psi_{\infty}\|_{H^4_*(\Om_L)}\le \liminf_{n\to \infty} \|\psi_n\|_{H^4_*(\Om_L)},\\
  \label{weak_*-cpt_W4}
  &\|\psi_{\infty}\|_{\mcl{W}^{4, \infty}_{*, \mcl{D}}(\Om_L)}\le \liminf_{n\to \infty} \|\psi_n\|_{\mcl{W}^{4, \infty}_{*, \mcl{D}}(\Om_L)}.
\end{align}
\begin{proof} This lemma is established based on the weak-$*$ compactness of the space $\mcl{W}^{4,\infty}_{*, \mcl{D}}(0, L)$.

Suppose that a sequence of functions $\{\psi_n\}_{n\in \mathbb{N}}$ is bounded in $H^4_*(\Om_L)\cap \mcl{W}^{4,\infty}_{*, \mcl{D}}(0,L)$. Since the sequence is bounded in $H^3(\Om_L)$, it follows from the generalized Sobolev inequality and the Arzel\`{a}-Ascoli theorem that there exists a subsequence of $\{\psi_n\}$ that converges in $C^1(\ol{\Om_L})$. For every small constant $d>0$, the sequence $\{\psi_n\}$ is bounded in $H^4(\Om_{L}\cap\{x_1<L-d\})$. Therefore, we can apply the generalized Sobolev inequality and the Arzel\`{a}-Ascoli theorem again to extract a subsequence of $\{\psi_n\}$ that converges in $C^2(\ol{\Om_{L}\cap\{x_1<L-d\}})$. By a diagonal argument, one can finally extract a subsequence of $\{\psi_n\}$ that converges in $C^2$ away from $\Gam_L$.
\smallskip

Let $\psi_{\infty}$ be the limit function of the convergent subsequence of $\{\psi_n\}$. Then, the estimate \eqref{weak_cpt_H4} is obtained due to the weak compactness property of the space $H^4_*(\Om_L)$.
\smallskip

Given a Hilbert space $H$, note that $L^{\infty}(0,L;H)$ is a Banach space. Furthermore, $L^{\infty}(0, L;H)$ is the dual space of $L^1(0,L;H)$, which is also a Banach space (see \cite{temam2001navier} or Appendix \ref{appendix-B}). Therefore, the Banach-Alaoglu theorem implies that $L^{\infty}(0,L;H)$ is weak-$*$ compact. Then the estimate \eqref{weak_*-cpt_W4} can be directly obtained by using the lower semi-continuity property of a weak-$*$ limit and a diagonal argument.

\end{proof}
\end{lemma}

\newcommand \tpsi{\tilde{\psi}}
\newcommand \tPsi{\tilde{\Psi}}
For a fixed $P=(\tpsi, \tPsi)\in \mcl{J}_{\delta}$, let us set
\begin{equation}
\label{definition of coeff}
\begin{split}
 &a_{ij}^P:=a_{ij}( \tPsi, \nabla\tpsi),\quad \tilde{f}_1^P:=f_1(\tPsi, \nabla\tPsi, \nabla\tpsi),\quad
 \tilde{f}_2^P:=f_2(\tPsi, \nabla\tpsi)
 \end{split}
\end{equation}
for $a_{ij}(z, {\bf q})$, $f_1(z, {\bf p}, {\bf q})$ and $f_2(z, {\bf q})$ given by Definition \ref{definition:iteration}. And, we define a bilinear differential operator $\mcl{L}^P$ by
\begin{equation}
\label{definition of L operator}
  \begin{split}
  \mcl{L}^P(v,w):=\sum_{i,j=1}^3 a_{ij}^P\der_{ij}v+\bar{a}_1\der_1v+\bar b_1\der_1w+\bar b_2w.
  \end{split}
\end{equation}
Next, we set up a linear boundary value problem associated with $P$ as follows:
\begin{equation}
\label{lbvp associated with P}
\left\{  \begin{split}
   \mcl{L}^P(v,w)=\tilde{f}_1^P,\quad
   \Delta w-\bar h_1w-\bar h_2\der_1v=\tilde{f}_2^P\quad&\tx{in\,\,$\Om_L$},\\
   v=0,\quad\partial_{x_1}v=u_{\rm en}-u_0,\quad
	\partial_{x_1}w=E_{\rm en}\quad&\mbox{on}\,\,\Gamma_0,\\
\partial_{{\bf n}_w}v=0,\quad\partial_{{\bf n}_w}w=0\quad&\mbox{on}\,\,\Gamma_{{\rm w}},\\
\der_{x_1}w=E_{\rm{ex}}-\bar{E}(L)\quad&\mbox{on}\,\,\Gamma_L.
  \end{split}\right.
\end{equation}
For $\rx=(x, \rx')\in (0, L)\times \mcl{D}(=\Om_L)$, let us set
\begin{equation}\label{def-Wbd}
\begin{split}
&w_{\rm bd}(\rx):=
\int_0^{x_1} (1-\frac yL)E_{\rm en}(\rx')+\frac yL(E_{\rm ex}(\rx')-\bar E(L))\,dy,\\
&g_1:=u_{\rm en}-u_0,\\
&f_1^P:=\tilde{f}_1^P-\mcl{L}^P(0,w_{\rm bd}),\\ &f_2^P:=\tilde{f}_2^P-\Delta w_{\rm bd}+\bar h_1w_{\rm bd}.
\end{split}
\end{equation}
It is clear that $(v,w)$ solves the problem \eqref{lbvp associated with P} if and only if $(V,W):=(v,w)-(0, w_{\rm bd})$ solves the following problem:
\begin{equation}
\label{simple-lbvp}
\left\{\begin{split}
   \mcl{L}^P(V,W)=f_1^P,\quad
   \Delta W-\bar h_1W-\bar h_2\der_1V=f_2^P\quad&\tx{in\quad$\Om_L$},\\
   V=0,\quad\partial_{x_1}V=g_1,\quad
	\partial_{x_1}{W}=0\quad&\mbox{on}\quad\Gamma_0,\\
\partial_{{\bf n}_w}V=0,\quad\partial_{{\bf n}_w}W=0\quad&\mbox{on}\quad\Gamma_{{\rm w}},\\
\der_{x_1}{W}=0\quad&\mbox{on}\quad\Gamma_L.
  \end{split}\right.
\end{equation}

\begin{proclaim}
For the rest of the paper, we shall state that {\emph{a constant is fixed depending on the date}} if the constant is fixed depending on $(\gam, J_0, S_0, \rho_0, b_0, \bar{\delta})$ unless otherwise specified. Also, any estimate constant $C$ appearing hereafter is presumed to be fixed depending only on the data unless otherwise specified.
\end{proclaim}

\newcommand \til{\tilde}
By using \eqref{back-com}-\eqref{Thm-com-Irr}, \eqref{coeff aij for potential}, \eqref{fs for potential}, and the compatibility conditions prescribed in the definition \eqref{definition of iteration set} of $\mcl{J}_{\delta}$, one can directly prove the following lemma (see \cite[Section 2.3.2]{bae2021three} for a detailed proof):
\begin{lemma}
\label{lemma-properties of coeffs}
There exist a small constant $\eps_0>0$, and a constant $C>0$ depending only on the data so that if $L$ satisfies $L\le \bar{L}$ for $\bar{L}$ from Lemma \ref{Lem1} then the following properties hold for all $P\in \mcl{J}_{2\eps_0}$:
\begin{itemize}
\item[(a)] For $P_0:=(0,0)$, we have
\begin{equation*}
  a_{ii}^{P_0}(\rx)=\begin{cases}
  1\quad&\mbox{for $i=1$}\\
  -\frac{1}{\frac{\bar u^2(x)}{\gam S_0\bar{\rho}^{\gam-1}}(x)-1}\quad&\mbox{for $i\neq 1$}
  \end{cases},\quad a_{ij}^{P_0}=0\quad\tx{for $i\neq j$};
\end{equation*}

\item[(b)] Let us set $P_0:=(0,0)$. For each $i,j=1,2,3$, we have
\begin{equation*}
\|a_{ij}^P-a_{ij}^{P_0}\|_{H^3_*(\Om_L)}+
\|a_{ij}^P-a_{ij}^{P_0}\|_{\mcl{W}^{3, \infty}_{*, \mcl{D}}(0,L)}\le C\|P\|_{*};
\end{equation*}

\item[(c)] The differential operator $\mcl{L}^P$ is hyperbolic with respect to $V$. Furthermore, the matrix $[a_{ij}^P]_{i,j=1}^3$ satisfies the following properties:
    \begin{itemize}
    \item[-] the matrix $[a_{ij}^P]_{i,j=1}^3$ is symmetric;
    \item[-] $\displaystyle{a_{11}^P\equiv 1\quad\tx{in}\quad \Om_L}$;
    \item[-] the sub-matrix $[a_{ij}^P]_{i,j=2}^3$ is negative definite with
        \begin{equation*}
          \frac{1}{2(\frac{\bar u^2(x)}{\gam S_0\bar{\rho}^{\gam-1}}(x)-1)}\le  -[a_{ij}^P]_{i,j=2}^3\le \frac{2}{\frac{\bar u^2(x)}{\gam S_0\bar{\rho}^{\gam-1}}(x)-1}\quad\tx{in $\Om_L$};
        \end{equation*}

    \end{itemize}

\item[(d)]
The functions $f_1^P$ and $f_2^P$ satisfy the estimates
\begin{equation*}
\begin{split}
&\|f_1^P\|_{H^3_*(\Om_L)}\le C\left(\|P\|_{*}^2
+\|E_{\rm en}\|_{C^3(\overline{\Gamma_0})}
+\|E_{\rm ex}-\bar{E}(L)\|_{C^3(\overline{\Gamma_L})}\right),\\
&\|f_2^P\|_{H^2(\Om_L)}\le C\Bigl(\|P\|_{H^4_*(\Om_L)}^2+\|b-b_0\|_{C^2(\overline{\Omega_L})}
+\|E_{\rm en}\|_{C^4(\overline{\Gamma_0})}+\|E_{\rm ex}-\bar{E}(L)\|_{C^4(\overline{\Gamma_L})}\Bigr);
\end{split}
\end{equation*}

\item[(e)] On the wall boundary $\Gamw$, the following compatibility conditions hold
\begin{equation*}
{{\bf n}_w}\cdot \sum_{j=1}^3a_{ij}^P{{\bf e}}_j=0\quad\tx{for $i=1,2,3$},\quad
\der_{{\bf n}_w}f_k^P=0  \quad\tx{for $k=1,2$};
\end{equation*}

\item[(f)] On the portion $\Gamma_0^{\bar{\epsilon}}$ of the entrance boundary $\Gam_0$, the functions $f_1^P$ and $f_2^P$ satisfy the compatibility conditions:
    \begin{equation*}
      f_1^P=0\quad\tx{and}\quad \partial_{x_1}f_2^P=0;
    \end{equation*}

\item[(g)] For any given $P_1, P_2\in \mcl{J}_{\delta}$, we have
\begin{equation*}
  \begin{split}
  &\|a_{ij}^{P_1}-a_{ij}^{P_2}\|_{L^2(\Om_L)}\le C\|P_1-P_2\|_{H^{1}(\Om_L)},\\
  &\|(\til{f}_1^{P_1}, \til{f}_2^{P_1})-(\til{f}_1^{P_2}, \til{f}_2^{P_2})\|_{L^2(\Om_L)}\le C\|P_1-P_2\|_{H^{1}(\Om_L)}\sum_{j=1}^2\|P_j\|_{H^4_*(\Om_L)} ;
  \end{split}
\end{equation*}

\item[(h)] For $j=1$ and $2$, let us write as $P_j=(\tpsi_j, \tPsi_j)$. Then, we have
    \begin{equation*}
      \|a_{ij}^{P_1}-a_{ij}^{P_2}\|_{L^{\infty}((0,L);L^2(\mcl{D}))}\le C\left(\|\tpsi_1-\tpsi_2\|_{\mcl{W}^{1,\infty}_{\mcl{D}}(0,L)}
      +\|\tPsi_1-\tPsi_2\|_{H^1(\Om_L)}\right).
    \end{equation*}
\end{itemize}
\end{lemma}


According to the statement (a) in  Lemma \ref{lemma-properties of coeffs}, the two equations given in \eqref{simple-lbvp} form a mixed-type system consisting of a hyperbolic equation and an elliptic equation. The main feature of this paper is to establish the well-posedness of the boundary value problem \eqref{simple-lbvp} in a three dimensional cylindrical domain with an arbitrary cross-section. The well-posedness of the same problem in a three dimensional rectangular nozzle is proved in \cite{bae2021three} by the method of reflections. In a rectangular nozzle, one can locally extend the boundary value problem by an even reflection about a flat portion on the wall boundary so that any corner point on the wall boundary can be dealt as an interior point. By this approach, one can establish a unique existence of a solution which is globally in $H^4$ up to the boundary. But if the cross-section is not rectangular, then this approach fails. To resolve this issue, we establish the well-posedness of the problem \eqref{simple-lbvp} in the weighted Sobolev space $H^4_*(\Om_L)\times H^4_*(\Om_L)$ introduced in Definition \ref{definition-weighted Sobolev norm}. The following proposition is a generalization of {\cite[Proposition 2.1]{bae2021three}}.

\begin{proposition}[The well-posedness of \eqref{simple-lbvp}]
\label{lemma-wp of lbvp for potential}
Fix a constant $\bar{\delta}\in (0,\rhos)$, and let $\bar L$ be fixed according to Lemma \ref{Lem1}.
And, let the constant $\eps_0>0$ be fixed according to Lemma \ref{lemma-properties of coeffs}.
Under the assumptions same as Theorem \ref{Irr-MainThm}, there exists a constant $L^*\in(0, \bar L]$ and a sufficiently small constant $\eps_1\in(0, \eps_0]$ so that
\begin{itemize}
\item[-] if $L$ satisfies $L\le L^*$,
\item[-] and if $\delta$ from the definition \eqref{definition of iteration set} of $\mcl{J}_{\delta}$ satisfies the inequality $0<\delta\le \eps_1$,
\end{itemize}
then the linear boundary value problem \eqref{simple-lbvp} associated with $P\in \mcl{J}_{\delta}$ has a unique solution $(V, W)\in[C^1(\ol{\Om_L})\cap C^2(\Om_L)]^2$ that satisfies the following estimates:
\begin{equation}
          \label{VW-weight=est}
          \begin{split}
  &\|(V, W)\|_{*}\le \kappa_0\left(\|P\|_*^2+\sigma(b, u_{\rm en}, E_{\rm en}, E_{\rm ex})\right)
  \end{split}
         \end{equation}
for the terms $\|P\|_*$ and $\sigma(b, u_{\rm en}, E_{\rm en}, E_{\rm ex})$ defined by \eqref{definition-P-norms} and \eqref{definition of sigma}, respectively.
Furthermore, the solution satisfies the compatibility conditions
\begin{equation}
\label{Comp-VW}
  \der_{x_1}^{k-1}V=0\quad\tx{and}\quad
  \der_{x_1}^k W=0\quad\tx{on $\Gamma_0^{\bar{\epsilon}}$ for $k=1,3$}.
\end{equation}

Finally, the constants $L^*$ and $\eps_1$ are fixed depending only on the data, and the constant $\kappa_0$ in the estimate \eqref{VW-weight=est} is fixed depending only on the data and $\bar{\epsilon}$.

\end{proposition}

In Section \ref{subsection:proof of main proposition-smooth}, we first establish a priori estimate \eqref{VW-weight=est} for smooth solutions to the boundary value problem \eqref{simple-lbvp}, then we shall prove Proposition \ref{lemma-wp of lbvp for potential} in Section \ref{subsection:proof of main proposition} by the method of Galerkin's approximations and a limiting argument.


\subsection{A priori estimates of a smooth solution} \label{subsection:proof of main proposition-smooth}
For a fixed $P=(\tpsi, \tPsi)\in \mcl{J}_{\delta}$, suppose that $(V,W)$ is a smooth solution to the linear boundary value problem \eqref{simple-lbvp} associated with $P$. We first establish a prior estimate of the solution $(V, W)$ in $[H^4_*(\Om_L)\cap \mcl{W}^{4,\infty}_{*, \mcl{D}}(0, L)]\times H^4_*(\Om_L)$.
\begin{proposition}
\label{proposition-apriori estimate of smooth solution}
Let two constants $\bar{L}$ and $\eps_0$ be from Lemma \ref{Lem1} and Lemma \ref{lemma-properties of coeffs}, respectively.
There exist a constant $L^*\in(0, \bar{L}]$ and a sufficiently small constant $\bar{\eps}\in(0, \eps_0]$ depending only on the data so that
\begin{itemize}
\item[-] if $L$ satisfies $L\le L^*$,
\item[-] and if $\delta$ from the definition \eqref{definition of iteration set} of $\mcl{J}_{\delta}$ satisfies the inequality $0<\delta\le \bar{\eps}$,
\end{itemize}
then we have the estimate
\begin{equation}
          \label{VW-weight=est-smooth}
          \|(V, W)\|_{*}\le C\left(\|f_1^P\|_{H^3_*(\Om_L)}+\|f_2^P\|_{H^2(\Om_L)}
  +\|g_1\|_{C^3(\ol{\Gam_0})}\right).
         \end{equation}

\end{proposition}

\subsubsection{A priori estimates of the first order derivatives} Differently from the works from \cite{bae2021structural, bae2021three}, we shall establish a priori $H^1$-estimate of $(V,W)$ in two steps.
\medskip

{\textbf{1.}} For a smooth function $\mfrak{W}(x_1)$ to be determined, let us consider the following integral expression:
\begin{equation}
\label{integral-expression}
  \begin{split}
  &\int_{\Om_L} \mfrak{W}\der_1 V\mcl{L}^P(V, W)-W(\Delta W-\bar{h}_1W-\bar{h}_2\der_1 V)\,d\rx=\int_{\Om_L} f_1^P \mfrak{W}\der_1 V-f_2^PW\,d\rx.
  \end{split}
\end{equation}
Hereafter, we let $a_{ij}$ and $\bar{a}_{ij}$ denote $a_{ij}^P$ and $a_{ij}^{P_0}$(for $P_0=(0,0)$), respectively.

By repeating the argument in \cite[Section 2.3.2]{bae2021three} with using the compatibility condition of the coefficient matrix $[a_{ij}]_{i,j=1}^3$ stated in Lemma \ref{lemma-properties of coeffs}(e), we get
\begin{equation*}
  \tx{LHS of \eqref{integral-expression}}=\mcl{I}_{\rm main}+\mcl{I}_{\rm wgt}
\end{equation*}
for $\mcl{I}_{\rm main}$ and $\mcl{I}_{\rm wgt}$ given as follows:
\begin{equation}\label{I-main}
  \begin{split}
  \mcl{I}_{\rm main}:=&\int_{\Gam_L}((\der_1V)^2-\sum_{i,j=2}^3 a_{ij}\der_iV\der_jV)\frac{\mfrak M}{2}\,d\rx'-\frac 12\int_{\Gam_0}g_1^2\mfrak M\,d\rx'\\
  &+\int_{\Om_L} (\bar a_1\mfrak M-\frac 12 \mfrak M')(\der_1 V)^2
  +\frac 12 \der_1(\bar a_{ij}\mfrak M)\der_iV\der_j V+ (\bar b_1\der_1 W+\bar b_2 W)\mfrak M \der_1 V\,d\rx\\
  &+\int_{\Om_L} \lvert \nabla W\rvert ^2+\bar h_1 W^2+\bar h_2W\der_1 V\,d\rx,
  \end{split}
\end{equation}
\begin{equation*}
\begin{split}
\mcl{I}_{\rm wgt}:=\int_{\Om_L}\frac 12\sum_{i,j=2}^3 \der_1 \left((a_{ij}-\bar a_{ij})\mfrak M\right)\der_i V\der_j V-\sum_{i=2}^3\sum_{j=1}^3\mfrak M \der_i(a_{ij}-\bar a_{ij})\der_1 V\der_j V\,d\rx.
\end{split}
\end{equation*}
By following the argument of step 3 in the proof of \cite[Proposition 2.4]{bae2021structural}, or referring the argument given in \cite[Section 2.3.2]{bae2021three}, we have the following lemma:
\begin{lemma}
\label{lemma-coercivity}
There exist constants $L^*\in(0, \bar L]$, $\mu_0>0$, $\mu_1>0$ and a small constant $\eps_1\in(0, \eps_0]$ (for the constant $\eps_0>0$ from Lemma \ref{lemma-properties of coeffs}) depending only on the data so that
\begin{itemize}
\item[-] if $L$ satisfies $L\le L^*$,
\item[-] and if $\delta$ from the definition \eqref{definition of iteration set} of $\mcl{J}_{\delta}$ satisfies $\delta\le \eps_1$,
\end{itemize}
then we have the inequality
\begin{equation}
\label{estimate-I main}
  \mcl{I}_{\rm main}\ge \mu_0
  \left(\int_{\Om_L}\lvert\nabla V\rvert^2+\lvert\nabla W\rvert^2+W^2\,d\rx+\int_{\Gam_L} \lvert\nabla V\rvert^2\,d\rx'\right)-\mu_1\int_{\Gam_0} g_1^2\,d\rx'.
\end{equation}
\end{lemma}
\smallskip

{\textbf{2.}}
By applying the Morrey's inequality, Lemma \ref{lemma-properties of coeffs}(b) and Definition \ref{definition-weighted Sobolev norm}, we can check the following estimate
\begin{equation}
\label{estimate-coeff}
  \|D(a_{ij}-\bar a_{ij})\|_{C^0([0, s]\times \ol{\mcl{D}})}\le
  \begin{cases}
 C\delta\quad&\tx{if $s\le \frac 34 L$}\\
  C\delta(L-s)^{-1/2} \quad&\tx{if $s> \frac 34 L$}
\end{cases}.
\end{equation}
Let us define a function $\kappa:(0,L)\rightarrow \R$ by
\begin{equation}
\label{definition of kappa}
  \kappa(s):=\left(\min\left\{\frac L4, L-s\right\}\right)^{-\frac 12}.
\end{equation}
Then it directly follows from \eqref{estimate-coeff} that one can fix a constant $\mu_L>0$ depending only on the data and $L$ so that the following estimate holds:
\begin{equation}
\label{estimate-coeff-C0}
  \|D a_{ij}\|_{C^0([0, s]\times \ol{\mcl{D}})}\le \mu_L\kappa(s)\quad\tx{for $0<s<L$}.
\end{equation}

For a fixed constant $t\in[0,L]$, let us set
\begin{equation*}
  \Om_t:=(0,t)\times \mcl{D},\quad \Gam_t:=\{t\}\times \mcl{D}.
\end{equation*}
Differently from the works in \cite{bae2021structural, bae2021three}, we shall estimate the term $ |\mcl{I}_{\rm wgt}|$ as
\begin{equation}
\label{estimate-I wgt}
\begin{split}
  |\mcl{I}_{\rm wgt}|
  &\le \int_0^L \|D(a_{ij}-D\bar{a}_{ij})\|_{C^{0}(\ol{\Gam_t})}
  \int_{\mcl{D}}|DV(x_1)| ^2\,d\rx'dx_1\\
  &\le \mu_L\int_0^L \kappa(x_1) \int_{\mcl{D}}|DV(x_1)| ^2\,d\rx'dx_1 \\
  &\le \mu_L\|V\|^2_{\mcl{W}^{1,\infty}_{\mcl{D}}(0,L)}\int_0^L \kappa(x_1)\,dx_1.
  \end{split}
\end{equation}

By combining the two estimates \eqref{estimate-I main} and \eqref{estimate-I wgt}, we easily derive the following lemma from \eqref{integral-expression}:
\begin{lemma}[Intermediate $H^1$ estimate]
\label{lemma-intermediate H1 est}
There exists a small constant $\hat{\eps}_1\in(0,\eps_1]$ depending only on the data so that
\begin{itemize}
\item[-] if $L$ satisfies $L\le L^*$,
\item[-] and if $\delta$ from the definition \eqref{definition of iteration set} of $\mcl{J}_{\delta}$ satisfies $\delta\le \hat{\eps}_1$,
\end{itemize}
then
 $(V, W)$ satisfies the estimate
\begin{equation}
\label{estimate-pre-H1}
\begin{split}
  &\|V\|_{H^1(\Om_L)}+\|W\|_{H^1(\Om_L)}\\
           &\le \kappa_0\left(\|f_1^P\|_{L^2(\Om_L)}
           +\|f_2^P\|_{L^2(\Om_L)}
           +\|g_1\|_{C^0(\overline{\Gamma_0})}+\delta \|DV\|_{L^{\infty}((0,L);L^2(\mcl{D}))}\right)
           \end{split}
\end{equation}
for a constant $\kappa_0>0$ fixed depending only on the data.
\end{lemma}
We note that the upper bound $L^*$ of the nozzle length $L$ is given for the sole purpose the $H^1$-estimate given in the above lemma.
\smallskip

{\textbf{3.}} Let us define a linear differential operator $\mfrak{L}^P_{\rm h}$ by 	
	\begin{equation*}
\mfrak{L}^P_{\rm h}(V):=\mcl{L}^P(V,0)(=\sum_{i,j=1}^3a_{ij}^{P}\partial_{ij}V+\bar{a}_1\partial_1V).
	\end{equation*}
It follows from the statement (c) in Lemma \ref{lemma-properties of coeffs} that the operator $\mfrak{L}^P_{\rm h}$ is hyperbolic in $\Om_L$. Next, we define a function $F_1$ by
\begin{equation}
\label{definition of F1}
F_1:=f_1^P-\mcl{L}^P(0, W).
\end{equation}
Then we regard $V$ as a solution to the linear boundary value problem:
\begin{equation*}
\mfrak{L}^P_{\rm h}(V)=F_1\tx{ in $\Om_L$}, \quad
  \begin{cases}V=0\\
  \der_{x_1} V=g_1\end{cases}
\tx{ on $\Gam_0$},\quad
\der_{{\bf n}_w}V=0\tx{ on $\Gam_{\rm w}$}.
\end{equation*}
In order to close the estimate given in Lemma \ref{lemma-intermediate H1 est}, we shall apply the method of an energy estimate for a linear hyperbolic equation for which we employ the Gr\"{o}nwall's inequality.

Fix $t\in(0,L)$. By integrating by parts the expression
\begin{equation}
\label{hyperbolic-integral-H1}
  \int_{\Om_t} \mfrak{L}^P_{\rm h}(V) \der_1 V\,d\rx=\int_{\Om_t} F_1\der_1 V\,d\rx,
\end{equation}
we get
\begin{equation*}
  \begin{split}
  {\tx{LHS of \eqref{hyperbolic-integral-H1}}}=&\frac 12(\int_{\Gam_t}-\int_{\Gam_0})(\der_1 V)^2-\sum_{i,j=2}^3 a_{ij}\der_i V\der_j V\,d\rx'+\mcl{H}(t)
  \end{split}\end{equation*}
for
\begin{equation*}
 \mcl{H}(t)= \int_{\Om_t}\bar a_1(\der_1 V)^2 +\sum_{i,j=2}^3 \frac 12 \der_1(a_{ij})\der_i q\der_j V-\sum_{i=1}^3\sum_{j=2}^3 \der_j a_{ij}\der_1 V\der_i V\,d\rx.
\end{equation*}
Next, let us define a function $\mcl{X}:[0,L]\rightarrow \R$ by
\begin{equation*}
  \mcl{X}(t):=\int_{\mcl{D}}|D V(t, \rx')|^2\, d\rx'.
\end{equation*}
By applying \eqref{estimate-coeff-C0}, we can estimate $\mcl{H}(t)$ as
\begin{equation*}
  |\mcl{H}(t)|\le C\int_0^t \kappa(s) \mcl{X}(s)\,ds\quad\tx{for any $t\in(0, L)$}.
\end{equation*}
By applying the estimate right in the above, we get a differential inequality for $\mcl{X}(t)$ as follows:
\begin{equation*}
  \mcl{X}(t)\le \mcl{X}(0)+C\left(\int_0^t \kappa(s)\mcl{X}(s)\,dx+ \left(\|f_1^P\|_{L^2(\Om_L)}
  +\|W\|_{H_1(\Om_L)}\right)^2\right).
\end{equation*}
Since we have
\begin{equation}
\label{est-kappa}
\int_0^L \kappa(s)\,ds=\frac 52\sqrt{L}<\infty,
\end{equation}
we can apply the Gr\"{o}nwall's inequality to obtain the estimate
\begin{equation*}
\|DV\|_{L^{\infty}((0,L);L^2(\mcl{D}))}\le C(\|g_1\|_{L^2(\mcl{D})}+\|f_1^P\|_{L^2(\Om_L)}
  +\|W\|_{H_1(\Om_L)}).
\end{equation*}
We substitute this estimate into \eqref{estimate-pre-H1} to get
\begin{equation*}
\begin{split}
&\|V\|_{H^1(\Om_L)}+\|W\|_{H^1(\Om_L)}\\
&\le \kappa_0\left(\|f_1^P\|_{L^2(\Om_L)}
           +\|f_2^P\|_{L^2(\Om_L)}
           +\|g_1\|_{C^0(\overline{\Gamma_0})}+\delta \|W\|_{H^1(\Om_L)}\right).
           \end{split}
\end{equation*}
Then, the following essential proposition is obtained:
\begin{proposition}[A priori estimates of the first order derivatives]
\label{proposition-H1-new}
One can reduce the constant $\hat{\eps}_1$ further from the one given in Lemma \ref{lemma-intermediate H1 est} depending only on the data so that if the constant $\delta>0$ from \eqref{definition of iteration set} satisfies the inequality
    \begin{equation*}
      0<\delta\le \hat{\eps}_1,
    \end{equation*}
    then we have the estimate
    \begin{equation}
    \label{estimate-H1-final}
    \begin{split}
&\|(V,W)\|_{H^1(\Om_L)}+\|V\|_{\mcl{W}^{1,\infty}_{\mcl{D}}(0,L)}\\
&\le C\left(\|f_1^P\|_{L^2(\Om_L)}+\|f_2^P\|_{L^2(\Om_L)}
  +\|g_1\|_{C^0(\ol{\Gam_0})}\right).
  \end{split}
\end{equation}
\end{proposition}

\subsubsection{Second order derivative estimates}
\label{subsubsection-H2 estimate}
For higher order derivative estimates of $(V, W)$, we shall use a bootstrap argument.
\medskip

{\textbf{A priori estimate of $D^2W$:}}
First of all, we rewrite \eqref{simple-lbvp} as two separate boundary value problems for $W$ and $V$, respectively. First of all, let us define a function $F_2$ by
\begin{equation}
\label{definition of F2}
  F_2:=\bar h_1W+\bar h_2\der_1V+f_2^P.
\end{equation}
Then we regard $W$ as a solution to the linear boundary value problem:
\begin{equation}
\label{bvp for W}
   \Delta W=F_2\tx{ in\quad$\Om_L$}, \quad
	\partial_{x_1}{W}=0 \tx{ on $\Gamma_0\cup \Gam_L$}, \quad
\partial_{{\bf n}_w}W=0 \tx{ on $\Gamma_{{\rm w}}$}.
\end{equation}
We locally extend the above problem about $\Gam_0$ and $\Gam_L$ by even extensions, respectively, then apply \cite[Theorem 8.12]{GilbargTrudinger} and Proposition \ref{proposition-H1-new} to obtain the estimate
\begin{equation}
\label{estimate-H2-W}
  \|W\|_{H^2(\Om_L)}\le C\left(\|f_1^P\|_{L^2(\Om_L)}
           +\|f_2^P\|_{L^2(\Om_L)}
           +\|g_1\|_{C^0(\overline{\Gamma_0})}\right).
\end{equation}

\medskip

{\textbf{A priori estimate of $D^2V$:}}
{\textbf{1.}}
On the whole, we shall estimate the $H^2$-norm of $V$ by following the idea given in \cite[Appendix A: Step 2 in the proof of Lemma 2.8]{bae2021structural}. In other words, we apply the method of an energy estimate for a linear hyperbolic equation, which employs the Gr\"{o}nwall's inequality. But, differently from the cases considered in \cite{bae2021structural, bae2021three}, we need to be more careful in treating the derivatives of the coefficients $\{a_{ij}\}_{i,j=1}^3$.

We differentiate the equation $\mfrak{L}^P_{\rm h}(V)=F_1$ with respect to $x_1$, and rewrite the result in terms of $q:=\der_1V$ as follows:
\begin{equation*}
  \mfrak{L}^P_{\rm h}(q)=\der_1F_1-(\der_1\mfrak{L}^P_{\rm h})(V)\quad\tx{in $\Om_L$}
\end{equation*}
for $\displaystyle{\der_1\mfrak{L}^P_{\rm h}:=\sum_{(i,j)\neq (1,1)}}\der_1a_{ij}\der_{ij}+\bar a_1'\der_1$. Then we get the following integral:
\begin{equation}\label{integral-H2}
  \int_{\Om_t}\mfrak{L}^P_{\rm h}(q)\der_1q\,d\rx=\int_{\Om_t}(\der_1F_1-(\der_1\mfrak{L}^P_{\rm h})(V))\der_1 q\,d\rx.
\end{equation}
\smallskip

{\textbf{2.}} By integrating by parts with using the compatibility condition for the coefficient matrix $[a_{ij}]_{i,j=1}^3$, given in the statement (e) in Lemma \ref{lemma-properties of coeffs}, we get
\begin{equation*}
  \begin{split}
  {\tx{LHS of \eqref{integral-H2}}}=&\frac 12(\int_{\Gam_t}-\int_{\Gam_0})(\der_1 q)^2-\sum_{i,j=2}^3 a_{ij}\der_i q\der_j q\,d\rx'+ \widetilde{\mcl{H}}(t)
  \end{split}\end{equation*}
for
\begin{equation*}
 \widetilde{\mcl{H}}(t)= \int_{\Om_t}\bar a_1(\der_1 q)^2 + \sum_{i,j=2}^3 \frac 12 \der_1(a_{ij})\der_i q\der_j q\,d\rx-\sum_{i=1}^3\sum_{j=2}^3 \der_j a_{ij}\der_1 q\der_i q\,d\rx.
\end{equation*}
Next, let us define a function $X:[0,L]\rightarrow \R$ by
\begin{equation*}
  X(t):=\int_{\mcl{D}}|\nabla q(t, \rx')|^2\, d\rx'.
\end{equation*}
By applying \eqref{estimate-coeff-C0}, we can estimate the term $\widetilde{\mcl{H}}(t)$ as
\begin{equation*}
  |\widetilde{\mcl{H}}(t)|\le C\int_0^t \kappa(s) X(s)\,ds\quad\tx{for any $t\in[0, L]$}.
\end{equation*}

Next, we rewrite the term $(\der_1\mfrak{L}^P_{\rm h})(V)$  as
\begin{equation*}
(\der_1\mfrak{L}^P_{\rm h})(V)=
2\sum_{j=2}^3\der_1 a_{1j}\der_jq+\sum_{i,j=2}^3 \der_1 a_{ij}\der_{ij}V=:I_1+I_2
\quad\tx{in $\Om_L$}.
\end{equation*}
By applying the estimate \eqref{estimate-coeff-C0}, we have
\begin{equation*}
  \begin{split}
  \int_{\Om_t}|I_1\der_{1}q|\,d\rx
  \le C\int_{0}^t \kappa(s)X(s)\,ds.
  \end{split}
\end{equation*}

In order to estimate the term $\int_{\Om_t} |I_2\der_1 q|\,d\rx$, we rewrite the equation $\mfrak{L}^P_{\rm h}(V)=F_1$ as
\begin{equation}
\label{gronwall-7}
  \sum_{i,j=2}^3 \bar a_{ij}\der_{ij}V+\bar{a}_1\der_1 V=F_1-a_{11}\der_{11}V-2\sum_{j=2}^3a_{1j}\der_j q+\sum_{i,j=2}^3 (\bar a_{ij}-a_{ij})\der_{ij} V\quad\tx{in $\Om_L$}.
\end{equation}
Given a constant $s\in(0, L)$, we regard $V(s,\cdot)$ as a solution to the two dimensional elliptic equation \eqref{gronwall-7} in $\mcl{D}$ with the Neumann boundary condition $\der_{\bf n}V=0$ on $\der \mcl{D}$. Then, it follows from \cite[Theorem 8.12]{GilbargTrudinger} and Lemma \ref{lemma-properties of coeffs}(b) that
\begin{equation*}
\begin{split}
&\|D_{\rx'}^2V(s,\cdot)\|_{L^2(\mcl{D})}\\
& \le C\left(\|V(s,\cdot)\|_{H^1(\mcl{D})}+\|F_1(s,\cdot)\|_{L^2(\mcl{D})}
  +\|\nabla q(s,\cdot)\|_{L^2(\mcl{D})}+\delta\|D_{\rx'}^2V(s,\cdot)\|_{L^2(\mcl{D})}\right).
  \end{split}
\end{equation*}
This directly yields the following lemma:
\begin{lemma}
\label{lemma-intermediate H2 est}
Suppose that the constants $L$ and $\delta$ satisfy all the assumptions stated in Lemma \ref{lemma-coercivity}. Then  one can fix a small constant $\eps_2\in(0,\eps_1]$ depending only on the data so that if $\delta$ satisfies the inequality $\delta\le \eps_2$, then $V$ satisfies the estimate
  \begin{equation*}
\begin{split}
&\|D_{\rx'}^2V(s,\cdot)\|_{L^2(\mcl{D})}\le C\left(\|V(s,\cdot)\|_{H^1(\mcl{D})}+\|F_1(s,\cdot)\|_{L^2(\mcl{D})}
  +\|\nabla q(s,\cdot)\|_{L^2(\mcl{D})}\right)
  \end{split}
\end{equation*}
for all $0<s<L$.
\end{lemma}
By applying this lemma and the estimate \eqref{estimate-coeff-C0}, we get
\begin{equation*}
  \begin{split}
  \int_{\Om_t}|I_2\der_1 q|\,d\rx
  \le \int_{0}^t\kappa(s)(\|V(s,\cdot)\|_{H^1(\mcl{D})}+\|F_1(s,\cdot)\|_{L^2(\mcl{D})}
  +\|\nabla q(s,\cdot)\|_{L^2(\mcl{D})})^2\, ds.
  \end{split}
\end{equation*}
Then it is derived from \eqref{integral-H2} that
\begin{equation}
\label{equation for X}
  X(t)\le X(0)+ C\left(\int_0^t \kappa(s)\left(X(s)+\|V(s,\cdot)\|^2_{H^1(\mcl{D})}\right)\,ds
  +\|F_1\|_{H^1(\Om_L)}^2\right)
\end{equation}
provided that the constant $\delta$ satisfies the condition $\delta\le \eps_2$ for the constant $\eps_2$ from Lemma \ref{lemma-intermediate H2 est}.

In order to get an estimate of the term $\|V(s,\cdot)\|^2_{H^1(\mcl{D})}$, let us set $Y(t):=\int_{\Gam_t}|DV|^2\,d\rx'$. By adjusting the argument given in the above, we can directly derive from the integral expression $\displaystyle{\int_{\Om_t} \mfrak{L}^P_{\rm h}(V)\der_1 V\,d\rx=\int_{\Om_t}F_1\der_1 V\,d\rx}$ that
\begin{equation*}
\begin{split}
  Y(t)&\le Y(0)+ C\left(\int_0^t\kappa(s)Y(s)\,ds+\|F_1\|_{L^2(\Om_L)}^2\right)\\
  &=\int_{\mcl{D}}|g_1|^2\,d\rx'
  +C\left(\int_0^t\kappa(s)Y(s)\,ds+\|F_1\|_{L^2(\Om_L)}^2\right).
  \end{split}
\end{equation*}
By applying the Gr\"{o}nwall's inequality with using \eqref{est-kappa}, we obtain the estimate
\begin{equation*}
  \esssup_{0<x_1<L}\|DV(x_1,\cdot)\|_{L^2(\mcl{D})}\le C(\|g_1\|_{L^2(\mcl{D})}+\|F_1\|_{L^2(\Om_L)}).
\end{equation*}
Next, we apply the trace inequality to get
\begin{equation*}
   \esssup_{0<x_1<L}\|V(x_1, \cdot)\|_{L^2(\mcl{D})} \le C\|V\|_{H^1(\Om_L)}.
\end{equation*}
The previous two estimates combined with \eqref{equation for X} yield that
\begin{equation*}
 X(t)\le X(0)+ C\left(\int_0^t \kappa(s)X(s)\,ds
  +(\|V\|_{H^1(\Om_L)}+\|g_1\|_{L^2(\mcl{D})}+\|F_1\|_{H^1(\Om_L)})^2\right).
\end{equation*}
By rewriting the equation $\mfrak{L}^P_{\rm h}(V)=F_1$ as $\der_{11}V=F_1-\sum_{(i,j)\neq (1,1)} a_{ij}\der_{ij}V-\bar a_1\der_1V$ on $\Gam_0$, we directly estimate the term $X(0)$ as
\begin{equation*}
|X(0)|\le C\left(\|F_1\|_{L^2(\Gam_0)} +\sum_{k=0}^1\|D_{\rx'}^k g_1\|_{L^2(\Gam_0)}+ \|W\|_{H^1(\Om_L)} \right)^2
\end{equation*}
so we finally have
\begin{equation*}
  X(t)\le
  C\left(\int_0^t \kappa(s)X(s)\,ds
  +\left(\|V\|_{H^1(\Om_L)}+\sum_{k=0}^1\|D_{\rx'}^k g_1\|_{L^2(\Gam_0)} +\|F_1\|_{H^1(\Om_L)}\right)^2\right)
\end{equation*}
for all $0<t<L$. Then the Gr\"{o}nwall's inequality yields that
\begin{equation*}
  \esssup_{0<t<L}X(t)\le C \left(\|V\|_{H^1(\Om_L)}+\sum_{k=0}^1\|D_{\rx'}^k g_1\|_{L^2(\Gam_0)} +\|F_1\|_{H^1(\Om_L)}\right)^2.
\end{equation*}
By combining this estimate with \eqref{definition of F1}, \eqref{estimate-H2-W}, Proposition \ref{proposition-H1-new} and Lemma \ref{lemma-intermediate H2 est}, we obtain the following estimate:
\begin{equation}
\begin{split}
\label{gronwall-4}
&\esssup_{0<x_1<L}\|D^2V(x_1,\cdot)\|_{L^2(\mcl{D})}+\|V\|_{H^2(\Om_L)}\le C \left(\|f_1^P\|_{H^1(\Om_L)}+\|f_2^P\|_{L^2(\Om_L)}
+\|g_1\|_{C^1(\ol{\Gam_0})}\right).
\end{split}
\end{equation}

So the following lemma is obtained:
\begin{lemma}
\label{proposition-H2-final}
One can fix a constant $\eps_3\in(0,\eps_2]$ sufficiently small depending only on the data so that if the constant $\delta>0$ from \eqref{definition of iteration set} satisfies the inequality
    \begin{equation*}
      0<\delta\le \eps_3,
    \end{equation*}
    then we have the estimate
    \begin{equation*}
\|V\|_{\mcl{W}^{2,\infty}_{\mcl{D}}(0,L)}
+\|(V, W)\|_{H^2(\Om_L)}\le
 C\left(\|f_1^P\|_{H^1(\Om_L)}+\|f_2^P\|_{L^2(\Om_L)}
  +\|g_1\|_{C^1(\ol{\Gam_0})}\right).
\end{equation*}
\end{lemma}

\subsubsection{A priori estimates of $D^3(V, W)$}
\label{subsubsection-H3 estimate}

Back to the boundary value problem \eqref{bvp for W}, we can easily establish the estimate
\begin{equation}
\label{estimate-H3-W}
  \|W\|_{H^3(\Om_L)}
  \le C\left(\|f_1^P\|_{H^1(\Om_L)}+\|f_2^P\|_{H^1(\Om_L)}
+\|g_1\|_{C^1(\ol{\Gam_0})}\right)
\end{equation}
by using a local extension argument combined with Lemma \ref{lemma-intermediate H1 est} and \eqref{gronwall-4}.

Next, we estimate the term $\displaystyle{ \esssup_{0<x_1<L}\|D^3V(x_1, \cdot)\|_{L^2(\mcl{D})}}$ by adjusting the argument given in \S \ref{subsubsection-H2 estimate}. Since it contains many tedious computations which can be easily derived by minor adjustments in the argument given in \S \ref{subsubsection-H2 estimate}, we shall explain the whole process very briefly, and provide details only for the parts that needs a careful approach.
\smallskip

{\textbf{1.}} Let us set
\begin{equation*}
  \zeta:=\der_{11}V\quad\tx{in $\Om_L$},
\end{equation*}
and define a function $X:[0,L]\rightarrow \R$ by
\begin{equation*}
  X(t):=\int_{\mcl{D}}|\nabla \zeta(t, \rx')|^2\, d\rx'.
\end{equation*}
By differentiating the equation $\mfrak{L}^P_{\rm h}(V)=F_1$ with respect to $x_1$ twice, we get the following equation for $\zeta$:
\begin{equation*}
  \mfrak{L}^P_{\rm h}(\zeta)=\der_{11}F_1-(\der_{11}\mfrak{L}^P_{\rm h})(V)\quad\tx{in $\Om_L$}
\end{equation*}
for $\displaystyle{\der_{11}\mfrak{L}^P_{\rm h}:=\sum_{(i,j)\neq (1,1)} \der_{11}(a_{ij}\der_{ij}+\bar a_1\der_1)}$. Then, we can show that
\begin{equation}
\label{gronwall for H3}
\begin{split}
  X(t)\le X(0)&+C\left(\int_0^t \kappa(s)X(s)\,ds+\int_{\Om_L\cap\{x_1<t\}} |\der_{11}F_1|^2\,d\rx\right)\\
 & +\int_0^t\int_{\mcl{D}}|\der_{11}\mfrak{L}^P_{\rm h}(V)\der_{1}\zeta|\,d\rx'ds.
  \end{split}
\end{equation}
\smallskip

{\textbf{2.}} Let us set $\displaystyle{\mcl{S}:=\sum_{(i,j)\neq (1,1)}\der_1 a_{ij}\der_{1ij} V\der_1\zeta}$. We rewrite $\mcl{S}$ as
    \begin{equation*}
      \mcl{S}=\sum_{{\tx{$i$ or $j$=1}\atop {(i,j)\neq (1,1)}}} \der_1 a_{ij}\der_{ij1}V\der_{1}\zeta+\sum_{i,j=2}^3\der_1 a_{ij}\der_1(\der_{ij1}V)\der_{1}\zeta=:\mcl{S}_1+\mcl{S}_2.
    \end{equation*}
    By \eqref{estimate-coeff}, we have
    \begin{equation*}
    \begin{split}
      &\int_0^t\int_{\mcl{D}}|\mcl{S}_1|\,d\rx'ds\le C\int_0^t \kappa(s)X(s)\,ds.
      \end{split}
    \end{equation*}
    Next, we differentiate the equation \eqref{gronwall-7} with respect to $x_1$, and apply \cite[Theorem 8.12]{GilbargTrudinger}, Lemma \ref{lemma-properties of coeffs}(b), \eqref{gronwall-4}, \eqref{estimate-H3-W} and the trace inequality to get the following result:

    \begin{lemma} \label{lemma-V-H3-intermediate1}
    For the constant $\eps_3>0$ from Lemma \ref{proposition-H2-final}, one can fix a small constant $\eps_4\in(0, \eps_3]$ depending only on the data so that if the constant $\delta>0$ from \eqref{definition of iteration set} satisfies the inequality
    \begin{equation*}
      0<\delta\le \eps_4,
    \end{equation*}
    then we have the estimate
    \begin{equation*}
\begin{split}
  &\|D^2_{\rx'}\der_1V(s,\cdot)\|_{L^2(\mcl{D})}\\
  &\le C\left(\|f_1^P\|_{H^2(\Om_L)}+\|f_2^P\|_{H^1(\Om_L)}
  +\|g_1\|_{C^1(\ol{\Gam_0})}+\delta \|D_{\rx'}^3V(s,\cdot)\|_{L^2(\mcl{D})} \right)
  \end{split}
\end{equation*}
for all $0<s<L$.
\end{lemma}

Then, by a lengthy computation with using \eqref{estimate-coeff}, \eqref{gronwall-4} and Lemma \ref{lemma-V-H3-intermediate1}, we can show that
    \begin{equation}
    \begin{split}
      &\int_0^t\int_{\mcl{D}}|\mcl{S}_2|\,d\rx'ds\\
&      \le C\int_0^t \kappa(s)\left(\|DF_1(s, \cdot)\|_{L^2(\mcl{D})}
  +\|g_1\|_{C^1(\ol{\Gam_0)}}+
  \delta\|D_{\rx'}^3V(s,\cdot)\|_{L^2(\mcl{D})}\right)
  \|\der_{1}\zeta(s,\cdot)\|_{L^2(\mcl{D})}\,ds.
  \end{split}
    \end{equation}
Next, we apply \cite[Theorem 8.13]{GilbargTrudinger} to the equation \eqref{gronwall-7} to estimate the term $\|D_{\rx'}^3V(s,\cdot)\|_{L^2(\mcl{D})}$ as
$$
 \|D_{\rx'}^3V(s,\cdot)\|_{L^2(\mcl{D})}\le C\left(\|V(s,\cdot)\|_{H^2(\mcl{D})}+\|\tx{RHS of \eqref{gronwall-7} at $x_1=s$}\|_{H^1(\mcl{D})}\right).
$$
Then we apply Lemma \ref{lemma-properties of coeffs}(b), \eqref{gronwall-4}, \eqref{estimate-H3-W} and the trace inequality to get
\begin{equation*}
\begin{split}
  &\|D_{\rx'}^3 V\|_{L^2(\mcl{D}_s)}\\
  &\le C
  \left(\|f_1^P\|_{H^2(\Om_L)}+\|f_2^P\|_{H^1(\Om_L)}
  +\|g_1\|_{C^1(\ol{\Gam_0})}+\delta \|D_{\rx'}^3V(s,\cdot)\|_{L^2(\mcl{D})} \right),
  \end{split}
\end{equation*}
which leads to the following lemma:
\begin{lemma} \label{lemma-V-H3-intermediate2}
    One can reduce the constant $\eps_4$ further from the one given in Lemma \ref{lemma-V-H3-intermediate1} with depending only on the data so that if the constant $\delta>0$ from \eqref{definition of iteration set} satisfies the inequality
    \begin{equation*}
      0<\delta\le \eps_4,
    \end{equation*}
    then we have the estimate
    \begin{equation}
\label{gronwall-13}
\|DD^2_{\rx'}V(s,\cdot)\|_{L^2(\mcl{D})} \le C\left(\|f_1^P\|_{H^2(\Om_L)}+\|f_2^P\|_{H^1(\Om_L)}
  +\|g_1\|_{C^1(\ol{\Gam_0})}\right)
\end{equation}
for all $0<s<L$.
\end{lemma}
Then we can finally estimate the term $\displaystyle{\int_0^t\int_{\mcl{D}}|\mcl{S}|\,d\rx'ds}$ as
\begin{equation}
\label{estimate for S-H3}
\begin{split}
  &\int_0^t\int_{\mcl{D}}|\mcl{S}|\,d\rx'ds\\
  &\le C\left(\int_0^t \kappa(s)X(s)\,ds
  +\left(\|f_1^P\|_{H^2(\Om_L)}+\|f_2^P\|_{H^1(\Om_L)}
  +\|g_1\|_{C^1(\ol{\Gam_0})} \right)^2\right).
  \end{split}
\end{equation}
\smallskip

{\textbf{3.}} Let us set
\begin{equation*}
\mcl{R}:=\der_{11}\mfrak{L}^P_{\rm h}(V)\der_{1}q-\sum_{(i,j)\neq (1,1)}\der_1 a_{ij}\der_{1ij}V\der_ 1 q.
\end{equation*}
A direct computation with applying the generalized H\"{o}lder inequality and Sobolev inequality yields the estimate
\begin{equation*}
\begin{split}
&\int_0^t\int_{\mcl{D}}  |\mcl{R}|\,d\rx'ds\\
\le&
C\Bigl(\int_0^t\int_{\mcl{D}}(\sum_{k=0}^2 |D^kV|^2
+|\der_1 q|^2)\,d\rx'ds\\
&\phantom{\le}+\int_0^t\sum_{(i,j)\neq (1,1)} \|\der_{11}a_{ij}(s,\cdot)\|_{L^4(\mcl{D})}
\left(\|D^2V(s,\cdot)\|_{L^2(\mcl{D})}+\|D_{\rx'}\der_{ij}V(s,\cdot)\|_{L^2(\mcl{D})}\right)\|\der_1q(s,\cdot)\|_{L^2(\mcl{D})}\,ds\Bigr).
\end{split}
\end{equation*}
By Lemma \ref{lemma-properties of coeffs}(b), Sobolev inequality and the definition of the function $\kappa(s)$ given by \eqref{definition of kappa}, we can estimate the term $\|\der_{11}a_{ij}(s,\cdot)\|_{L^4(\mcl{D})}$ as
\begin{equation*}
  \|\der_{11}a_{ij}(s,\cdot)\|_{L^4(\mcl{D})}\le C\kappa(s).
\end{equation*}
By combining this estimate with \eqref{gronwall-4} and \eqref{gronwall-13}, we can directly check that
\begin{equation}
\label{estimate for R-H3}
\begin{split}
  &\int_0^t\int_{\mcl{D}}  |\mcl{R}|\,d\rx'ds\\
 & \le C\left(\int_0^t \kappa(s)X(s)\,ds+\left(\|f_1^P\|_{H^2(\Om_L)}+\|f_2^P\|_{H^1(\Om_L)}
  +\|g_1\|_{C^1(\ol{\Gam_0})} \right)^2\right).
  \end{split}
\end{equation}
\smallskip

{\textbf{4.}} Now, we collect all the estimates of \eqref{gronwall for H3}, \eqref{estimate for S-H3} and \eqref{estimate for R-H3} together to show that
\begin{equation*}
  \esssup_{0<x_1<L}\|D\der_{11}V(x_1, \cdot)\|_{L^2(\mcl{D})}\le C\left(\|f_1^P\|_{H^2(\Om_L)}+\|f_2^P\|_{H^1(\Om_L)}
  +\|g_1\|_{C^2(\ol{\Gam_0})} \right).
\end{equation*}
Next, we combine the above estimate with \eqref{gronwall-13} to obtain the estimate
\begin{equation*}
  \esssup_{0<x_1<L}\|D^3V(x_1, \cdot)\|_{L^2(\mcl{D})}\le C\left(\|f_1^P\|_{H^2(\Om_L)}+\|f_2^P\|_{H^1(\Om_L)}
  +\|g_1\|_{C^2(\ol{\Gam_0})} \right).
\end{equation*}
Finally, we can easily prove the following lemma by using \eqref{gronwall-4} and the above estimate:
\begin{lemma}
\label{proposition-V-H3-final}
One can fix a constant $\eps_4>0$ sufficiently small depending only on the data so that if the constant $\delta>0$ from \eqref{definition of iteration set} satisfies the inequality
    \begin{equation*}
      0<\delta\le \eps_4,
    \end{equation*}
    then we have the estimate
    \begin{equation*}
    \begin{split}
&\|V\|_{\mcl{W}^{3,\infty}_{\mcl{D}}(0,L)}
+\|(V, W)\|_{H^3(\Om_L)}\\
&\le
 C\left(\|f_1^P\|_{H^2(\Om_L)}+\|f_2^P\|_{H^1(\Om_L)}
  +\|g_1\|_{C^2(\ol{\Gam_0})}\right).
  \end{split}
\end{equation*}
\end{lemma}

\subsubsection{A priori estimates of the fourth order derivatives}
By extending the argument given in \S \ref{subsubsection-H3 estimate} to the fourth order derivatives of $(V, W)$, we shall prove the following lemma:
\begin{lemma}
\label{lemma-H4 estimate-final}
For the constant $\eps_4>0$ given in Lemma \ref{proposition-V-H3-final}, one can fix a constant $\eps_5\in(0,\eps_4]$ sufficiently small depending only on the data so that if the constant $\delta>0$ from \eqref{definition of iteration set} satisfies the inequality
\begin{equation*}
  0<\delta\le \eps_5,
\end{equation*}
then we have the estimate
\begin{equation*}
\begin{split}
  &\|(V, W)\|_{*}\le C\left(\|f_1^P\|_{H^3_*(\Om_L)}+\|f_2^P\|_{H^2(\Om_L)}
  +\|g_1\|_{C^3(\ol{\Gam_0})}\right).
  \end{split}
\end{equation*}
\end{lemma}
\begin{proof}
{\textbf{1.}} The first part of the proof is devoted to estimate the norm $\|W\|_{H^4_*(\Om_L)}$.
\smallskip

{\textbf{(1-1)}}
Back to the boundary value problem \eqref{bvp for W} for $W$, it follows from the compatibility condition \eqref{Comp-VW} and Lemma \ref{lemma-properties of coeffs}(f) that the function $F_2$, defined by \eqref{definition of F2}, satisfies the compatibility condition
\begin{equation*}
 \der_1 F_2=\bar h_1'W+\bar h_2'\der_1 V\quad\tx{on  $\Gamma_0^{\bar{\epsilon}}$.}
\end{equation*}
Note that \eqref{one-re-uE} and \eqref{def-bar-h} imply that
\begin{equation*}
  (\bar h_1', \bar h_2')
  =\frac{\bar{\rho}^{3-\gam}\bar E}{\gam S_0\left(\gam S_0\bar{\rho}^{\gam-1}-\frac{J_0^2}{\bar{\rho}^2}\right)}
  \left(2-\gam,\,\frac{J_0}{\bar{\rho}^2}-(2-\gam)\frac{1}{\bar{\rho}}\right)\quad\tx{in $\ol{\Om_L}$}.
\end{equation*}
Then the condition \eqref{special condition} yields that
\begin{equation}
\label{comp-cond-F2}
  \der_1 F_2=0\quad\tx{on $\Gamma_0^{\bar{\epsilon}}$.}
\end{equation}
Then we can apply a reflection argument (with using a local even extension of $W$ about $\Gam_0$) and Proposition \ref{proposition-V-H3-final} to obtain the estimate
\begin{equation}
\label{H4-estimate of W pre}
  \|W\|_{H^4(\Om_L\cap\{x<\frac 45 L\})}\le C(\|f_1^P\|_{H^2(\Om_L)}+\|f_2^P\|_{H^2(\Om_L)}+\|g_1\|_{C^2(\ol{\Gam_0})}).
\end{equation}
\smallskip

{\textbf{(1-2)}}
Fix a function $\chi\in C^{\infty}(\R)$ that satisfies the following properties:
\begin{equation*}
  \chi(x_1)=\begin{cases}1\quad\mbox{for $x_1\le 0$}\\
  0\quad\mbox{for $x_1\ge \frac 12$}\end{cases},\quad \chi'(x_1)\le 0\quad\tx{for all $x_1\in \R$}.
\end{equation*}
Next, we fix another function $\eta \in C^{\infty}(\R)$ that satisfies the following properties:
\begin{equation*}
  \eta(x_1)=\begin{cases}0\quad\mbox{for $x_1<\frac L8 $}\\
  1\quad\mbox{for $x_1\ge \frac L4$}\end{cases},\quad \eta'(x_1)\ge 0\quad\tx{for all $x_1\in \R$}.
\end{equation*}
For each constant $d\in(0, \frac L2)$, we define a scaled cut-off function $\chi_d:\R\to\mathbb{R}_+$  by
	\begin{equation*}
	\chi_d(x_1):=\eta(x_1)\chi\left(\frac{x_1-L+d}{d}\right).
	\end{equation*}
For each $d>0$, the scaled function $\chi_d$ satisfies the following properties:
\begin{itemize}
\item[-] $\displaystyle{{\rm spt}\,\chi_d\subset \left[\frac L8, L-\frac d2\right]}$;
\item[-] $\displaystyle{\chi_d(x_1)=1\,\,\tx{for $\frac L4\le x_1\le L-d$}}$;
\item[-] for each $k\in \mathbb{N}$, there exists a constant $C_k>0$ depending only on $k$ so that the estimate $\displaystyle{\left\|\frac{d^k\chi_d}{dx_1^k}\right\|_{C^0(\R)}\le \frac{C_k}{d^k}}$ holds.
\end{itemize}
For each $d\in (0, \frac L2)$, let us set $\xi_d(\rx):=\der_1^3W(\rx)\chi_d(x_1)$. Then, by a straight forward computation with using the estimate \eqref{H4-estimate of W pre} and the integral expression
\begin{equation}
\label{weak formulation for W2}
  \int_{\Om_L} \nabla (\der_1^3W)\cdot \nabla \xi_d\,d\rx=\int_{\Om_L} \der_{11}F_2\der_1\xi_d\,d\rx,
\end{equation}
we can show that
\begin{equation*}
  \|D\der_1^3W\|_{L^2(\Om_L\cap\{\frac L4 <x_1<L-d\})}\le \frac{C}{\sqrt{d}}
  \left(\|f_1^P\|_{H^2(\Om_L)}+\|f_2^P\|_{H^2(\Om_L)}+\|g_1\|_{C^2(\ol{\Gam_0})}\right).
\end{equation*}
More generally, by locally flattening the wall boundary $\Gam_{\rm w}$ and using the partition of unity, one can similarly show that
\begin{equation*}
  \|D^4_{\rx}W\|_{L^2(\Om_L\cap\{\frac L4 <x_1<L-d\})}\le \frac{C}{\sqrt{d}}\left(\|f_1^P\|_{H^2(\Om_L)}+\|f_2^P\|_{H^2(\Om_L)}
  +\|g_1\|_{C^2(\ol{\Gam_0})}\right).
\end{equation*}
By combining this estimate with \eqref{H4-estimate of W pre} and Proposition \ref{proposition-V-H3-final}, we obtain the estimate
\begin{equation}
\label{H4-estimate of W final}
 \|W\|_{H^4_*(\Om_L)}\le C\left(\|f_1^P\|_{H^2(\Om_L)}+\|f_2^P\|_{H^2(\Om_L)}
 +\|g_1\|_{C^2(\ol{\Gam_0})}\right).
\end{equation}

{\textbf{2.}}
To complete the proof, it remains to estimate the term $$\displaystyle{\sup_{d\in(0,L)}d^{1/2}\esssup_{0<x_1<L-d}\|D^4V(x_1, \cdot)\|_{L^2(\mcl{D})}}.$$
Overall, the estimate is given by adjusting the argument given in \S \ref{subsubsection-H3 estimate} so we shall provide an outline for the estimate, and add more details only for the parts that need a careful computation.
\smallskip

{\textbf{(2-1)}}
Let us set
\begin{equation*}
 \xi:= \der_1^3 V\quad\tx{in $\Om_L$},
\end{equation*}
and define a function $Z:[0,L]\rightarrow \R$ by
\begin{equation*}
  Z(t):=\int_{\mcl{D}}|\nabla \xi(t,\rx')|^2\,d\rx.
\end{equation*}
As an analogy of \eqref{gronwall for H3}, we can directly derive from the integral expression
\begin{equation}
\label{integral-expression for H4}
\int_0^t\int_{\mcl{D}}\mfrak{L}^P_{\rm h}(\der_1^3 V)\der_1\xi\,d\rx'ds
=\int_0^t\int_{\mcl{D}}(\der_1^3F_1-(\der_1^3\mfrak{L}^P_{\rm h})(V))\der_1\xi\,d\rx'ds
\end{equation}
that, for any $0<t<L$,
\begin{equation}
\label{gronwall-H4-part1}
\begin{split}
  Z(t)\le Z(0)
 & +C\left(\int_0^t \kappa(s)Z(s)\,ds+\int_{\Om_L\cap\{x_1<t\}}|\der_1^3F_1|^2\,d\rx\right)\\
 & +\left|\int_0^t\int_{\mcl{D}}(\der_1^3\mfrak{L}^P_{\rm h})(V)\der_1\xi\,d\rx'ds\right|
 \end{split}
\end{equation}
for $\displaystyle{\der_1^3\mfrak{L}^P_{\rm h}:=\sum_{(i,j)\neq (1,1)}\der_1^3(a_{ij}\der_{ij}+\bar a_1\der_1)}$.
\medskip

{\textbf{(2-2)}} For each $k=1,2,3$, let us set
\begin{equation*}
  T_k:= \int_0^t\int_{\mcl{D}} \left(\sum_{(i,j)\neq (1,1)}\der_1^k a_{ij}\der_1^{3-k} \der_{ij}V+\der_1^k\bar{a}_1\der_1^{3-k}\der_1 V\right)\der_1\xi\,d\rx'ds
\end{equation*}
so that we have
\begin{equation*}
\int_0^t\int_{\mcl{D}}(\der_1^3\mfrak{L}^P_{\rm h})(V)\der_1\xi\,d\rx'ds=T_1+T_2+T_3.
\end{equation*}
One can adjust the argument in Steps 2--4 in \S \ref{subsubsection-H3 estimate}, and  apply the generalized H\"{o}lder inequality, Sobolev inequality, Lemma \ref{lemma-properties of coeffs}(b) and the estimate \eqref{H4-estimate of W final} to show that
\begin{equation}
\label{app-h4-6}
  \begin{split}
  \sum_{k=1}^3|T_k|\le C  \int_0^t \kappa(s)\|D\xi(s,\cdot)\|_{L^2(\mcl{D})}\Bigl(
  \sum_{n=0}^2\|D_{\rx'}^{4-n}\der_1^n V(s,\cdot)\|_{L^2(\mcl{D})}+\|D^{n+1}V(s,\cdot)\|_{L^2(\mcl{D})}&\\
  +\|D\xi(s,\cdot)\|_{L^2(\mcl{D})}
  &\Bigr)\,ds.
  \end{split}
\end{equation}
Furthermore, by adjusting the argument given in \S \ref{subsubsection-H3 estimate}, we can prove the following lemma:
\begin{lemma}
\label{lemma-intermediate H4 est}
For the constant $\eps_4>0$ given in Lemma \ref{proposition-V-H3-final}, one can fix a constant $\eps_5\in(0,\eps_4]$ sufficiently small depending only on the data so that if the constant $\delta>0$ from \eqref{definition of iteration set} satisfies the inequality
\begin{equation*}
  0<\delta\le \eps_5,
\end{equation*}
then, for each $s\in(0, L)$, we have the estimate
\begin{equation*}
\begin{split}
&\sum_{n=0}^2\|D_{\rx'}^{4-n}\der_1^n V(s,\cdot)\|_{L^2(\mcl{D})}\\
&\le
 C\left(\|F_1(s,\cdot)\|_{H^2(\mcl{D})}
 +\sum_{m=0}^3\|D^{m}V(s,\cdot)\|_{L^2(\mcl{D})}+\|D\xi(s, \cdot)\|_{L^2(\mcl{D})}\right)
 \end{split}
\end{equation*}
for all $0<s<L$.
\end{lemma}
We combine this lemma with \eqref{app-h4-6} and \eqref{definition of F1} to obtain the estimate
\begin{equation}
\label{gronwall-H4-part2}
\begin{split}
&\sum_{k=1}^3|T_k|\\
&\le C\int_0^t \kappa(s) \left(Z(s)+\sum_{m=0}^3\|D^{m}V(s,\cdot)\|_{L^2(\mcl{D})}^2
   +\|f_1^P(s,\cdot)\|_{H^2(\mcl{D})}^2+\|\der_1 W(s,\cdot)\|_{H^2(\mcl{D})}^2\right)\,ds.
   \end{split}
\end{equation}

\medskip
{\textbf{(2-3)}} We combine \eqref{gronwall-H4-part1}--\eqref{gronwall-H4-part2} altogether, and apply Proposition \ref{proposition-V-H3-final} and \eqref{H4-estimate of W final} to show that
\begin{equation}
\label{gronwall for H4-final}
\begin{split}
 Z(t)\le Z(0)+& C\Bigl(\int_0^t \kappa(s)(Z(s)+\|f_1^P(s,\cdot)\|_{H^2(\mcl{D})}^2+\|\der_1 W(s,\cdot)\|_{H^2(\mcl{D})}^2)\,ds\\
 &\phantom{a}\quad+\left(\|f_1^P\|_{H^2(\Om_L)}+\|f_2^P\|_{H^2(\Om_L)}
 +\|g_1\|_{C^2(\ol{\Gam_0})}\right)^2
 +\int_{\Om_L\cap\{x_1<t\}}|\der_{1}^3F_1|^2 \,d\rx\Bigr).
 \end{split}
\end{equation}

Before proceeding further, it is useful to make the following observation:
\begin{itemize}
\item[(i)] 
By the trace inequality, one can fix a constant $\tilde{C}_L>0$ depending only on $L$ to satisfy the estimate
    \begin{equation*}
        \esssup_{0<s<L-d} \|f_1^P(s,\cdot)\|_{H^2(\mcl{D})} \le \tilde{C}_Ld^{-1/2} \|f_1^P\|_{H^3_{\ast}(\Omega_L)}
                         \end{equation*}
                         for any constant $d\in(0,L)$;
\item[(ii)] For a fixed constant $d\in(0, L)$, suppose that $0<t<L-d$. Then, there exists a constant $C>0$ depending only on the data to satisfy the estimate
    \begin{equation*}
      \int_{\Om_L\cap\{x_1<t\}} |\der_1^3 F_1|^2\,d\rx\le \frac Cd \left(\|f_1^P\|_{{H}^{3}_{*}(\Om_L)}+\|W\|_{H^4_*(\Om_L)}\right)^2;
    \end{equation*}

\item[(iii)] By the trace inequality, one can fix a constant ${C}_L>0$ depending only on $L$ to satisfy the estimate
    \begin{equation*}
 \esssup_{0<x_1<L-d} \|\der_1 W(x_1, \cdot)\|_{H^2(\mcl{D})} \le {C}_L  d^{-1/2}\|W\|_{H^4_*(\Om_L)}
    \end{equation*}
    for any constant $d\in(0,L)$.
\end{itemize}
By applying the Gronwall inequality with using the properties (i)-(iii) stated right in the above, we obtain the following lemma:
\begin{lemma}
\label{lemma-H4 estimate-pre}
For the constant $\eps_5>0$ fixed in Lemma \ref{lemma-intermediate H4 est}, if the constant $\delta>0$ from \eqref{definition of iteration set} satisfies the inequality
\begin{equation*}
  0<\delta\le \eps_5,
\end{equation*}
then there exists a constant $C>0$ depending only on the data so that it holds that
\begin{equation*}
\begin{split}
&\sup_{d>0} d^{1/2}\esssup_{0<x_1<L-d}\|D\der_1^3V(x_1, \cdot)\|_{L^2(\mcl{D})}\\
&\le C \left(\|f_1^P\|_{H^3_*(\Om_L)}+\|f_2^P\|_{H^2(\Om_L)}+\|g_1\|_{C^3(\ol{\Gam_0})}\right).
\end{split}
\end{equation*}
\end{lemma}

Finally, the proof of Lemma \ref{lemma-H4 estimate-final} can be completed by combining Lemma \ref{lemma-H4 estimate-pre} with the estimate \eqref{H4-estimate of W final} and Lemma \ref{lemma-intermediate H4 est}.

\end{proof}

\subsubsection{Proof of Proposition \ref{proposition-apriori estimate of smooth solution}} By fixing the constant $L^*$ as the one given from Lemma \ref{lemma-coercivity}, and the constant $\bar{\eps}$ as $\bar{\eps}=\eps_5$ for the constant $\eps_5$ from Lemma \ref{lemma-H4 estimate-final}, Proposition \ref{proposition-apriori estimate of smooth solution} directly follows from Lemma \ref{lemma-H4 estimate-final}.\qed

\subsection{Proof of Proposition \ref{lemma-wp of lbvp for potential}}\label{subsection:proof of main proposition}

Now we establish the well-posedness of the boundary value problem \eqref{simple-lbvp} associated with any $P\in \mcl{J}_{\delta}$ by using Proposition \ref{proposition-apriori estimate of smooth solution}. The main idea is to apply the method of Galerkin's approximations and a limiting argument with an aid of the following technical lemma:

\begin{lemma}[Partially smooth approximations of $P=(\tpsi, \tPsi)\in \mcl{J}_{\delta}$]
\label{lemma-main section-smooth approx}
For any given $P=(\tpsi, \tPsi)\in \mcl{J}_{\delta}$, one can take a sequence $\{P_m=(\tpsi_m, \tPsi_m)\}_{m\in \mathbb{N}}$ that satisfies the following properties:
\begin{itemize}
\item[(a)] There exists a constant $\mu_0>0$ depending only on $(\mcl{D}, L)$ so that every $P_m$ satisfies the following two estimates:
    \begin{align}
    \label{norm-bd-H4}
     &\|P_m\|_{H^{4}_{*}(\Om_L)}\le \mu_0\|P\|_{H^{4}_{*}(\Om_L)},\\
     \label{norm-bd-W4}
      &\|\tpsi_m\|_{\mcl{W}^{4,\infty}_{*, \mcl{D}}(0,L)}\le \mu_0\|\tpsi\|_{\mcl{W}^{4,\infty}_{*, \mcl{D}}(0,L)};
    \end{align}
\item[(b)] $\displaystyle{\lim_{m\to \infty} \|P_m-P\|_{H^3(\Om_L)}=0}$;
\item[(c)] $\displaystyle{\der_{{\bf n}_w}\tpsi_m=\der_{{\bf n}_w}\tPsi_m=0}$ on $\displaystyle{\Gamw}$;
\item[(d)] $\displaystyle{\der_{x_1}^{k-1}\tpsi_m=\der_{x_1}^k \tPsi_m=0}$ on $\displaystyle{\Gam_0^{\bar{\epsilon}/2}}$ for $k=1,3$;
\item[(e)] For each fixed $\rx'\in\mcl{D}$, $P_m(x_1, \rx')$ and $D_{\rx}P_m(x_1, \rx')$ are $C^{\infty}$ with respect to $x_1\in [0,L]$.
\end{itemize}
\end{lemma}
One may be curious how Lemma \ref{lemma-main section-smooth approx} is used in this paper. The motivation and the usefulness of this lemma can be clearly understood in the proof of Proposition \ref{lemma-wp of lbvp for potential}.
We give a proof of Lemma \ref{lemma-main section-smooth approx} in Appendix \ref{appendix-smooth approximation}.
\medskip

For the small constant $\eps_0>0$ given in Lemma \ref{lemma-properties of coeffs}(d) and the constant $\mu_0>0$ from Lemma \ref{lemma-main section-smooth approx}(d), let us set
\begin{equation*}
  \delta_0:=\frac{\eps_0}{\mu_0}.
\end{equation*}
\begin{lemma}[Limiting argument]
\label{lemma-pre wp property}
Suppose that the inequality
\begin{equation*}
  \delta\le \delta_0
\end{equation*}
holds.
For a fixed $P\in \mcl{J}_{\delta}$, let $\{P_m\}$ be a sequence that satisfies all the properties stated in Lemma \ref{lemma-main section-smooth approx}.
Suppose that, for every $m$, the linear boundary value problem \eqref{simple-lbvp} associated with $P_m$ has a solution $Q_m=(V_m, W_m)$ that satisfies the estimate \eqref{VW-weight=est-smooth} with $P$ on the right-hand sides of the estimate being replaced by $P_m$ for some constant $C>0$. Then, the linear boundary value problem \eqref{simple-lbvp} associated with $P$ has a solution $Q=(V, W)$ that satisfies the estimates \eqref{VW-weight=est}.

\begin{proof}
{\textbf{1.}}
It easily follows from Lemma \ref{lemma-properties of coeffs}(d) and Lemma \ref{lemma-main section-smooth approx} that if the inequality
\begin{equation*}
  \delta\le \delta_0
\end{equation*}
holds, then
 we can fix a constant $C>0$ depending only on the data and $\mcl{D}$ so that, for every $P_m$, the following estimates hold:
\begin{equation*}
  \begin{split}
  &\|f_1^{P_m}\|_{H^3(\Om_L)}\le C\left(\|P\|_{*}^2
+\|E_{\rm en}\|_{C^3(\overline{\Gamma_0})}
+\|E_{\rm ex}-\bar{E}(L)\|_{C^3(\overline{\Gamma_L})}\right);\\
&\|f_2^{P_m}\|_{H^2(\Om_L)}
\le C\left(\|P\|_{*}^2+\|b-b_0\|_{C^2(\ol{\Om_L})}
+\|E_{\rm en}\|_{C^4(\overline{\Gamma_0})}
+\|E_{\rm ex}-\bar{E}(L)\|_{C^4(\overline{\Gamma_L})}\right).
  \end{split}
\end{equation*}
Since it is assumed that every solution $Q_m=(V_m, W_m)$ satisfies the estimate \eqref{VW-weight=est-smooth} with $P$ on the right-hand side of the estimate being replaced by $P_m$, the estimates of $(f_1^{P_m}, f_2^{P_m})$ stated in the right above directly yield the following result:
\begin{equation*}
  \|Q_m\|_*\le C^*\left(\|P\|_{*}^2+\sigma(b, u_{\rm en}, E_{\rm en}, E_{\rm ex})+\|g_1\|_{C^3(\ol{\Gam_0})}\right)
\end{equation*}
for the constant $\sigma(b, u_{\rm en}, E_{\rm en}, E_{\rm ex})$ defined by \eqref{definition of sigma} for some constant $C^*>0$ fixed depending only on the data and $\mcl{D}$. Therefore, the sequence of $\{V_m\}$ is bounded in $H^4_*(\Om_L)\cap \mcl{W}^{4,\infty}_{*, \mcl{D}}(0, L)$, and the sequence of $\{W_m\}$ is bounded in $H^4_*(\Om_L)$. Then it follows from Lemma \ref{lemma-weak star limit} that there exists a subsequence $\{Q_{m_k}=(V_{m_k}, W_{m_k})\}$ of $\{Q_m\}$, and an element $Q_{\infty}=(V_{\infty}, W_{\infty})\in [H^4_*(\Om_L)\cap \mcl{W}^{4,\infty}_{*, \mcl{D}}(0, L)]\times H^4_*(\Om_L)$ so that the following properties hold:
\begin{itemize}
\item[-] the sequence $\{Q_{m_k}\}$ converges to $Q_{\infty}$ in $C^1(\ol{\Om_L})\cap C^2(\ol{\Om_L}\setminus \Gam_L)$;
\item[-] the limit $Q_{\infty}$ satisfies the estimate
\begin{equation*}
\|Q_{\infty}\|_*\le C^*\left(\|P\|_{*}^2+\sigma(b, u_{\rm en}, E_{\rm en}, E_{\rm ex})+\|g_1\|_{C^3(\ol{\Gam_0})}\right).
\end{equation*}
\end{itemize}

\medskip

{\textbf{2.}} {\emph{Claim: The limit $Q_{\infty}=(V_{\infty}, W_{\infty})$ solves the linear boundary value problem \eqref{simple-lbvp} associated with the originally fixed $P\in \mcl{J}_{\delta}$.}}

For each $P_m$, let us set
\begin{equation*}
  \mathbb{A}^{P_m}:=[a_{ij}^{P_m}]_{i,j=1}^3
\end{equation*}
for $a_{ij}^{P_m}$ given by \eqref{definition of coeff}.
From Lemma \ref{lemma-main section-smooth approx} and Lemma \ref{lemma-properties of coeffs}(b) and (g), it follows that the sequences $\{\mathbb{A}^{P_m}\}$ and $\{(f_1^{P_m}, f_2^{P_m})\}$ converge to $\mathbb{A}^{P}$ and $(f_1^P, f_2^P)$ in $C^0(\ol{\Om_L})$, respectively, and this verifies the claim.
\end{proof}
\end{lemma}

Finally, we are ready to prove the well-posedness of the problem \eqref{simple-lbvp} associated with any $P\in \mcl{J}_{\delta}$ provided that $\delta$ is fixed appropriately.

\begin{proof}
[{\textbf{Proof of Proposition \ref{lemma-wp of lbvp for potential}}}]
{\textbf{1.}}\emph{(The existence of a solution)} {\textbf{(1-1)}} For the constant $\delta_0>0$ from Lemma \ref{lemma-pre wp property}, suppose that the inequality
\begin{equation*}
  0<\delta\le \delta_0
\end{equation*}
holds, and fix $P=(\tpsi, \tPsi)\in \mcl{J}_{\delta}$. And, take a sequence $\{P_m=(\tpsi_m, \tPsi_m)\}_{m\in \mathbb{N}}$ of {\emph{partially smooth approximations}} of $P$ that satisfy all the properties (a)--(e) stated in Lemma \ref{lemma-main section-smooth approx}.

Next, we fix $m\in \mathbb{N}$, and consider the linear boundary value problem:
\begin{equation}
\label{lbvp-Pm}
\left\{\begin{split}
   \mcl{L}^{P_m}(V,W)=f_1^{P_m},\quad
   \Delta W-\bar h_1W-\bar h_2\der_1V=f_2^{P_m}\quad&\tx{in\quad$\Om_L$},\\
   V=0,\quad\partial_{x_1}V=g_1,\quad
	\partial_{x_1}{W}=0\quad&\mbox{on}\quad\Gamma_0,\\
\partial_{{\bf n}_w}V=0,\quad\partial_{{\bf n}_w}W=0\quad&\mbox{on}\quad\Gamma_{{\rm w}},\\
\der_{x_1}{W}=0\quad&\mbox{on}\quad\Gamma_L.
  \end{split}\right.
\end{equation}

{\emph{Claim: For the constant $L^*$ from Lemma \ref{lemma-coercivity}, suppose that the inequality
\begin{equation*}
  0<L\le L^*
\end{equation*}
holds. Then, one can take a small constant $\delta_1\in(0, \delta_0]$ depending only on the data and $L$ so that if the inequality
\begin{equation*}
  0<\delta\le \delta_1
\end{equation*}
holds, then the linear boundary value problem \eqref{lbvp-Pm} has a solution $(V, W)$ that satisfies the estimate
\begin{equation}
          \label{VW-weight=est-smooth-Pm}
          \begin{split}
  &\|(V, W)\|_{*}\le C\left(\|f_1^{P_m}\|_{H^3_*(\Om_L)}+\|f_2^{P_m}\|_{H^2(\Om_L)}
  +\|g_1\|_{C^3(\ol{\Gam_0})}\right)
  \end{split}
         \end{equation}
for a constant $C>0$ fixed depending only on the data.          }}

Once this claim is verified, the existence of a solution to the linear boundary value problem \eqref{simple-lbvp} associated with $P\in \mcl{J}_{\delta}$ directly follows from Lemma \ref{lemma-pre wp property}.
\smallskip

{\textbf{(1-2)}} We shall verify the claim stated in the above by applying the method of Galerkin's approximations. Since this is a widely used approach, and there are many references(e.g. see \cite{bae2021structural} and \cite{bae2021three}) from which all the details can be derived, we shall only list out crucial ingredients required to adjust the arguments given in \cite{bae2021structural, bae2021three} to our case.
\begin{itemize}
\item[($G_1$)] Let us consider a two dimensional eigenvalue problem in $\mcl{D}$:
    \begin{equation*}
      -\Delta_{\rx'}\eta=\om\eta\quad\tx{in $\mcl{D}$},\quad \der_{{\bf n}_w}\eta=0\quad\tx{on $\der\mcl{D}$}.
    \end{equation*}
    Let $\{\om_k:k=0,1,2,\cdots\}$ be the set of all the eigenvalues with satisfying the inequality $\om_k\le \om_{k+1}$ for $k=0,1,2,\cdots$. Then, we have $\om_0=0$ and $\om_k>0$ for any $k>0$;

\item[($G_2$)] For each $k\in\mathbb{Z}_+$, let $\eta_k:\ol{\mcl{D}}\rightarrow \R$ be an eigenfunction corresponding to $\om_k$. We can take a sequence of eigenfunctions $\{\eta_k:k=0,1,2,\cdots\}$ so that it forms an orthonormal basis in $L^2(\mcl{D})$ with respect to the standard inner product $\langle \cdot, \cdot \rangle$ defined by
    \begin{equation*}
      \langle \xi, \eta \rangle :=\int_{\mcl{D}} \xi(\rx')\eta(\rx')\,d\rx'.
    \end{equation*}
Then the sequence $\{\frac{D_{\rx'}\eta_k}{\sqrt{\om_k}}:k=0,1,2,\cdots\}$ forms an orthonormal bases of $L^2(\mcl{D};\R^2)$ with respect to the inner product $\langle \cdot, \cdot \rangle_{\R^2}$ defined by
\begin{equation*}
      \langle {\bf u}, {\bf v} \rangle_{\R^2} :=\int {\bf u}(\rx')\cdot {\bf v}(\rx')\,d\rx'.
    \end{equation*}

\item[($G_3$)] Due to the compatibility condition for the function $u_{\rm en}$ given in \eqref{Thm-com-Irr}, the function $g_1$ given by \eqref{def-Wbd} satisfies the compatibility condition $\der_{{\bf n}_w} g_1=0$ on $\der \mcl{D}$.
\end{itemize}

{\textbf{Finite dimensional approximation of the boundary value problem \eqref{lbvp-Pm}:}} For a fixed $n\in \mathbb{N}$, let us find $(V_n, W_n)$ in the form of
\begin{equation*}
  (V_n, W_n)(x_1, \rx')=(\sum_{l=0}^n\theta_l(x_1)\eta_l(\rx'),
  \sum_{l=0}^n\Theta_l(x_1)\eta_l(\rx') )
\end{equation*}
to solve the following boundary value problem for $\{(\theta_l ,\Theta_l)(x_1):l=0, \cdots, n\}$:
\begin{equation}\label{Vm-eq}
\begin{split}
&\left\{\begin{split}
 	&\langle\mcl{L}^{P_m}(V_n,W_n)(x_1, \cdot),\eta_k\rangle=\langle f_1^{P_m}(x_1, \cdot),\eta_k\rangle\\
   	&\langle(\Delta W_n-\bar h_1W_n-\bar h_2\der_1V_n)(x_1, \cdot),\eta_k\rangle=\langle f_2^{P_m}(x_1, \cdot),\eta_k\rangle
	\end{split}\right.\quad\mbox{for $0<x_1<L$},\\
&\left\{\begin{split}
	&\langle V_n, \eta_k \rangle=0,\quad\langle\partial_{x_1}V_n, \eta_k \rangle=\langle g_1,\eta_k\rangle\\
&\langle\partial_{x_1}W_n,\eta_k \rangle =0 \end{split}\right. \quad\mbox{at $x_1=0$},\\
&\left\langle\der_{x_1}W_n, \eta_k \rangle=0 \right. \quad\mbox{at $x_1=L$}
	\end{split}
\end{equation}
for all $k=0,1,\cdots,n$.
\smallskip

Here is an easy but significant observation.
By a direct computation with using \eqref{definition of L operator}, we can check that
\begin{equation*}
\begin{split}
  &\langle\mcl{L}^{P_m}(V_n,W_n)(x_1, \cdot),\eta_k\rangle\\
  &=\theta_k''(x_1)+\bar a_1(x_1) \theta_k'+\bar b_1(x_1)\Theta_k'(x_1)+\bar b_2\Theta_k(x_1)+\sum_{l=0}^n \alp_l(x_1) \theta_l'(x_1)+\beta_l(x_1)\theta_l'(x_1)
  \end{split}
\end{equation*}
for
\begin{equation*}
  \begin{split}
  \alp_l(x_1):=2\sum_{j=2}^3\langle a_{1j}^{P_m}(x_1, \cdot)\der_j\eta_l, \eta_k\rangle,\quad
  \beta_l(x_1):=\sum_{i,j=2}^3 \langle a_{ij}^{P_m}(x_1,\cdot)\der_{ij}\eta_l, \eta_k\rangle.
  \end{split}
\end{equation*}
By the definition of $a_{ij}^P$ given by \eqref{definition of coeff}, it follows from the statement (e) of Lemma \ref{lemma-main section-smooth approx} that $\alp_l$ and $\beta_l$ are smooth with respect to $x_1\in[0,L]$. {\emph{(This is why we only seek for a partially smooth approximation of $P$ that is smooth with respect to $x_1$. And, this simple but very useful idea enables us to establish the existence of three dimensional supersonic solutions of potential flow models in a cylinder with an arbitrary cross-section.)}} Therefore, \eqref{Vm-eq} becomes a boundary value problem of a second order linear ODE system with smooth coefficients.
By applying Lemma \ref{proposition-H2-final} with some adjustments in the framework of Galerkin's approximations, we can show that if the inequality
\begin{equation*}
\delta\le \frac{\eps_3}{\mu_0}
\end{equation*}
holds for the constants $\mu_0$ and $\eps_3$ from Lemma \ref{lemma-main section-smooth approx} and Lemma \ref{proposition-H2-final}, respectively, then it follows from the Fredholm alternative theorem that the problem \eqref{Vm-eq} has a unique solution $\{(\theta_j ,\Theta_j)(x_1):j=0, \cdots, n\}$, that is smooth for $0\le x_1\le L$. Furthermore, lengthy but straight forward computations with using the properties ($G_1$)--($G_3$) stated in the above show that Proposition \ref{proposition-apriori estimate of smooth solution} applies to the solution $(V_n, W_n)$ of the problem \eqref{Vm-eq}. More precisely, we conclude that if the inequality
\begin{equation*}
  \delta\le \frac{\bar{\eps}}{\mu_0}
\end{equation*}
holds for the constant $\bar{\eps}$ from Proposition \ref{proposition-apriori estimate of smooth solution}, then there exists a constant $C>0$ depending only on the data and $L$ so that every $(V_n, W_n)$ satisfies the estimate:
\begin{equation*}
          \begin{split}
  &\|(V_n, W_n)\|_{*}\le C\left(\|f_1^{P_m}\|_{H^3_*(\Om_L)}+\|f_2^{P_m}\|_{H^2(\Om_L)}
  +\|g_1\|_{C^3(\ol{\Gam_0})}\right).
  \end{split}
         \end{equation*}
For further details on how to obtain the above estimate from Proposition \ref{proposition-apriori estimate of smooth solution}, one can refer to \cite[Appendix A]{bae2021structural}.

Then the claim is verified by taking a subsequence of $\{(V_n, W_n)\}$ that converges in $C^1(\ol{\Om_L})\cap C^2(\ol{\Om_L}\setminus \Gam_L)$ along with choosing the constant $\delta_1$ as
\begin{equation*}
  \delta_1=\min\left\{\delta_0, \frac{\bar{\eps}}{\mu_0}\right\}.
\end{equation*}

\medskip

{\textbf{2.}}(The uniqueness) Let $Q^{(1)}=(V^{(1)}, W^{(1)})$ and $Q^{(2)}=(V^{(2)}, W^{(2)})$ be two solutions to the linear boundary value problem \eqref{simple-lbvp} associated with a fixed $P\in \mcl{J}_{\delta}$, and suppose that they satisfy the estimate \eqref{VW-weight=est}. Next, let us set $\displaystyle{(v,w):=(V^{(1)}-V^{(2)}, W^{(1)}-W^{(2)})}$. Then $(v,w)$ solves the problem:
\begin{equation*}
\left\{\begin{split}
   \mcl{L}^P(v,w)=0,\quad
   \Delta w-\bar h_1w-\bar h_2\der_1v=0\quad&\tx{in\quad$\Om_L$},\\
   v=0,\quad\partial_{x_1}v=0,\quad
	\partial_{x_1}{w}=0\quad&\mbox{on}\quad\Gamma_0,\\
\partial_{{\bf n}_w}v=0,\quad\partial_{{\bf n}_w}w=0\quad&\mbox{on}\quad\Gamma_{{\rm w}},\\
\der_{x_1}{w}=0\quad&\mbox{on}\quad\Gamma_L.
  \end{split}\right.
\end{equation*}
By repeating the proof of Proposition \ref{proposition-H1-new}, we can show that $\|(v,w)\|_{H^1(\Om_L)}=0$, thus $v=w\equiv 0$ in $\Om_L$.
\medskip

{\textbf{3.}}(The compatibility conditions \eqref{Comp-VW})
Due to the boundary conditions stated in \eqref{simple-lbvp}, $V$ and $W$ clearly satisfy the compatibility conditions $V=0$ and $\der_{x_1} W=0$ on $\Gamma_0^{\bar{\epsilon}}$.
\medskip

By solving the equation $\mcl{L}^P(V, W)=f_1^P$ for $V$, we have
\begin{equation*}
\partial_{x_1}^2V=-\sum_{k=2}^32a_{1k}^P\partial_{1k}V-\sum_{i,j=2}^3a_{ij}^P\partial_{ij}V-\bar{a}_1\partial_1 V-\bar{b}_1\partial_1W-\bar{b}_2W+f_1^P \quad\tx{in $\Om_L$}.
\end{equation*}
Due to the $H^4$-estimate \eqref{VW-weight=est} of $(V, W)$ established away from $\Gam_L$, the representation given in the right above is valid pointwisely up to the entrance boundary $\Gam_0$. Then, by substituting the boundary conditions $V=\der_{x_1}W=0$ and $\der_{x_1} V=g_1$ on $\Gam_0$ into the above expression, we get
\begin{equation*}
  \partial_{x_1}^2V=-\sum_{k=2}^32a_{1k}^P\partial_{k}g_1
  -\bar{a}_1\partial_1 g_1-\bar{b}_2W+f_1^P\quad\tx{on $\Gam_0$}.
\end{equation*}
The assumption \eqref{special condition} combined with Definition \ref{definition:iteration}(ii)  implies that $\bar a_1=\bar b_2=0$ on $\Gam_0$. And, due to Lemma \ref{lemma-properties of coeffs}(f), we have $f_1^P=0$ on  $\Gamma_0^{\bar{\epsilon}}$. Finally, Definition \ref{definition:iteration}(i) combined with the compatibility condition $\til{\psi}=0$ on $\Gamma_0^{\bar{\epsilon}}$ prescribed for $P=(\tpsi, \tPsi)\in \mcl{J}_{\delta}$ implies that $a_{1k}^P=0$ on $\Gamma_0^{\bar{\epsilon}}$ for $k=2$ and $3$. Therefore we conclude that $V$ satisfies the compatibility condition $\der_{x_1}^2V=0$ on $\Gamma_0^{\bar{\epsilon}}$.
\medskip

Let us rewrite the second order equation for $W$ given in \eqref{simple-lbvp} as
\begin{equation}
\label{new repr for W}
  \der_{x_1}^2 W=-\Delta_{\rx'}W+\bar h_1 W+\bar h_2 \der_1 V+f_2^{P}\quad\tx{in $\Om_L$}.
\end{equation}
Note that the function $F^*_2:=\bar h_1 W+\bar h_2 \der_1 V+f_2^{P}$ is in $H^3(\Om_{4L/5})$. Due to the compatibility condition \eqref{comp-cond-F2} for $F_2$ on $\Gam_0^{\bar{\epsilon}}$, a local even extension of $F_2$ about $\Gam_0$ is in $H^3$, therefore we can easily check that $W$ is in $H^5(\Om_{2L/3})$ by applying a standard elliptic theory. Therefore, the representation \eqref{new repr for W} is valid pointwisely in $\ol{\Om_{L/2}}$. Next, we differentiate \eqref{new repr for W} with respect to $x_1$ to get
\begin{equation*}
  \der_{x_1}^3 W=-\Delta_{\rx'}\der_{x_1}W+\der_{x_1}F_2^*\quad\tx{on $\Gam_0$}.
\end{equation*}
Then it directly follows from the boundary condition $\der_{x_1}W=0$ on $\Gam_0$, Lemma \ref{lemma-properties of coeffs}(f) and  the compatibility condition \eqref{comp-cond-F2} that $W$ satisfies the compatibility condition $\der_{x_1}^3W=0$ on $\Gam_0^{\bar{\epsilon}}$.

\end{proof}

\subsection{Proof of Theorem \ref{Irr-MainThm}}
\label{subsection-proof of theorem-potential}
Now, we are ready to prove Theorem \ref{Irr-MainThm}.
\medskip


{\textbf{1.}} For $\eps_1$ from Proposition \ref{lemma-wp of lbvp for potential}, suppose that the constant $\delta$ in the definition \eqref{definition of iteration set} satisfies the condition
\begin{equation*}
  \delta\le \eps_1.
\end{equation*}
For the constant $L^*$ from Proposition \ref{lemma-wp of lbvp for potential}, let us fix $L\in(0,L^*]$.
By Proposition \ref{lemma-wp of lbvp for potential}, for each $P=(\tilde{\psi},\tilde{\Psi})\in \mcl{J}_{\delta}$, the boundary value problem \eqref{simple-lbvp} has a unique solution $(V,W)\in [C^1(\ol{\Om_L})\cap C^2(\Om_L)]^2$ that satisfies the estimate \eqref{VW-weight=est}. For the function $w_{\rm bd}$ given by \eqref{def-Wbd}, let us set
\begin{equation*}
  (v,w):=(V,W)+(0, w_{\rm bd})\quad\tx{in $\Om_L$}.
\end{equation*}
It is clear that $(v,w)$ solves the boundary value problem \eqref{lbvp associated with P}, and satisfies the estimate
\begin{equation}
\label{estimate-vw-Irr}
\|(v,w)\|_{*}
\le C_{\star}\left(\|P\|_*^2+\sigma(b,u_{\rm en},E_{\rm en},E_{\rm ex})\right)
\end{equation}
for a constant $C_{\star}>0$ depending only on the data, $L$ and $\bar{\epsilon}$.
In addition, it directly follows from \eqref{Comp-VW} and the compatibility condition \eqref{Thm-com-Irr} stated in Problem \ref{Irr-EP-Prob} that
\begin{equation*}
  \der_{x_1}^{k-1}v=0\quad\tx{and}\quad
  \der_{x_1}^k w=0\quad\tx{on $\Gamma_0^{\bar{\epsilon}}$}\quad\tx{for $k=1,3$.}
\end{equation*}
Let us define a map $\mathcal{I}$ by
\begin{equation}
\label{definition:iteration mapping}
\mathcal{I}(P):=(v,w).
\end{equation}
It is clear from \eqref{estimate-vw-Irr} that $\mcl{I}$ maps the iteration set $\mcl{J}_{\delta}$ into itself if the inequality
\begin{equation}
\label{condition-iteration 1-Irr}
  C_{\star}\left(\delta^2+\sigma(b,u_{\rm en},E_{\rm en},E_{\rm ex})\right)\le \delta
\end{equation}
holds.
\medskip

{\textbf{2.}} (The existence of a solution)
{\emph{Claim 1: One can fix a constant $\sigma_1>0$ sufficiently small depending only on the data and $L$ so that if the inequality
\begin{equation*}
 \sigma(b,u_{\rm en},E_{\rm en},E_{\rm ex})\le \sigma_1
\end{equation*}
holds, then Problem \ref{Irr-EP-Prob} has at least one solution $(\vphi, \Phi)$.
}}
\smallskip

We shall verify {\emph{claim 1}} by applying the Schauder fixed point theorem.
\medskip

Let us define a linear space $\mcl{B}$ by
\begin{equation}
\label{definition of B set}
  \mcl{B}:=(H^2(\Om_L)\cap C^{1,1/4}(\ol{\Om_L}))\times (H^2(\Om_L)\cap C^{1,1/4}(\ol{\Om_L})).
\end{equation}
Clearly, $\mcl{B}$ is a Banach space, and the set $\mcl{J}_{\delta}$ is compact in $\mcl{B}$ (Lemma \ref{lemma-general property of iter set-Irr}(b)).

Suppose that a sequence $\{P_k=(\tpsi_k ,\tPsi_k)\}_{k\in \mathbb{N}}$ in $ \mcl{J}_{\delta}$ converges to $P_{\infty}=(\tpsi_{\infty} ,\tPsi_{\infty})$ in the Banach space $\mcl{B}$.  For each $k\in\mathbb{N}\cup \{\infty\}$, let us set $(v_k, w_k):=\mcl{I}(P_k)$. Assuming that the condition \eqref{condition-iteration 1-Irr} holds, we have $(v_k, w_k)\in\mcl{J}_{\delta}$ for all $k\in \mathbb{N}\cup \{\infty\}$. Then it follows from Lemmas \ref{lemma-general property of iter set-Irr} and \ref{lemma-weak star limit} that the sequence $\{(v_k, w_k)\}_{k\in \mathbb{N}}$ has a subsequence $\{(v_{k_j}, w_{k_j})\}_{j\in \mathbb{N}}$ that converges to some $(v_{\infty}, w_{\infty})\in \mcl{J}_{\delta}$ in the Banach space $\mcl{B}$.
Furthermore, it is clear that $(v_{\infty}, w_{\infty})$ solves the boundary value problem \eqref{lbvp associated with P} associated with $P_{\infty}$, that is, we have $\mcl{I}(P_{\infty})=(v_{\infty}, w_{\infty})$. This shows that the iteration map $\mcl{I}:\mcl{J}_{\delta}\rightarrow \mcl{J}_{\delta}$ is continuous in $\mcl{B}$  due to the uniqueness of a solution to the linear boundary value problem \eqref{lbvp associated with P}. Therefore we can apply the Schauder fixed point theorem to get a fixed point $P_*=(\psi_*, \Psi_*)$ in $\mcl{J}_{\delta}$ provided that $\delta$ and $\sigma(b,u_{\rm en},E_{\rm en},E_{\rm ex})$ are fixed to  satisfy the condition \eqref{condition-iteration 1-Irr}. This proves the existence of a solution to the nonlinear boundary value problem \eqref{system for potential perturbation}--\eqref{bcs for potential perturbation}, thus proves the existence of a solution to Problem \ref{Irr-EP-Prob}.
\medskip

{\textbf{3.}}(Verification of the estimate \eqref{up-est}) For a fixed point $P_*=(\psi_*, \Psi_*)\in \mcl{J}_{\delta}$ of the iteration mapping $\mcl{I}$, we use the estimate \eqref{estimate-vw-Irr} to get
\begin{equation}
\label{F-est-phi}
  \|P_*\|_{*}\le C_{\star}(\delta \|P_*\|_{*}+\sigma(b,u_{\rm en},E_{\rm en},E_{\rm ex})).
\end{equation}
So if $\delta$ satisfies the condition
\begin{equation}
\label{condition-iteration 2-Irr}
C_{\star}\delta\le \frac 12,
\end{equation}
 then we obtain the estimate
\begin{equation*}
  \|P_*\|_{*}\le 2C_{\star} \sigma(b,u_{\rm en},E_{\rm en},E_{\rm ex}).
\end{equation*}
This verifies the estimate \eqref{up-est}.
\medskip

{\textbf{4.}}(The uniqueness)
Suppose that $(\varphi^{(1)},\Phi^{(1)})$ and $(\varphi^{(2)},\Phi^{(2)})$ are two solutions of Problem \ref{Irr-EP-Prob}. And, suppose that they satisfy the estimate \eqref{up-est}. Let us define
\begin{equation*}
(\psi^{(k)}, \Psi^{(k)}):=(\varphi^{(k)}-\varphi_0,\Phi^{(k)}-\Phi_0)  \quad\tx{for $k=1,2$}.
\end{equation*}
Finally, let us set $P_k:=(\psi^{(k)}, \Psi^{(k)})$ for each $k=1,2$, and
\begin{equation*}
  (\xi, \eta):=P_1-P_2\quad\tx{in $\Om_L$}.
\end{equation*}

For each $P_k$, let $(a_{ij}^{P_k}, \til f_1^{P_k}, \til f_2^{P_k})$ be given by \eqref{definition of coeff}. By Lemma \ref{lemma-properties of coeffs}, $(a_{ij}^{P_k}, \til f_1^{P_k}, \til f_2^{P_k})$ are well defined as long as the condition
\begin{equation}
\label{condition-contraction1}
  \|P_k\|_{*}\le \eps_0
\end{equation}
holds for $k=1$ and $2$, where $\eps_0$ is the small constant given from Lemma \ref{lemma-properties of coeffs}. Therefore, if the constant $\sigma(b,u_{\rm en},E_{\rm en},E_{\rm ex})$ satisfies the condition
\begin{equation}
\label{condition-sigma 1-Irr}
  C \sigma(b,u_{\rm en},E_{\rm en},E_{\rm ex})\le \eps_0
\end{equation}
for the constant $C>0$ from the estimate \eqref{up-est}, then we can easily derive that
\begin{equation*}
\left\{\begin{split}
   \mcl{L}^{P_1}(\xi,\eta)=F_1,\quad
   \Delta \eta-\bar h_1\eta-\bar h_2\der_1\xi=F_2\quad&\tx{in\quad$\Om_L$},\\
   \xi=0,\quad\partial_{x_1}\xi=0,\quad
	\partial_{x_1}{\eta}=0\quad&\mbox{on}\quad\Gamma_0,\\
\partial_{{\bf n}_w}\xi=0,\quad\partial_{{\bf n}_w}\eta=0\quad&\mbox{on}\quad\Gamma_{{\rm w}},\\
\der_{x_1}{\eta}=0\quad&\mbox{on}\quad\Gamma_L
  \end{split}\right.
\end{equation*}
for $F_1$ and $F_2$ given by
\begin{equation*}
  \begin{split}
  F_1&:=\til f_1^{P_1}-\til f_1^{P_2}+\sum_{i,j=1}^3 (a_{ij}^{P_1}-a_{ij}^{P_2})\der_{ij}\psi^{(2)},\\
  F_2&:=\til f_2^{P_1}-\til f_2^{P_2}.
  \end{split}
\end{equation*}
For the constant $\displaystyle{\hat{\eps}_1}$ from Proposition \ref{proposition-H1-new}, if it holds that
\begin{equation}
\label{condition-contraction2}
  \|P_1\|_*\le \hat{\eps}_1,
\end{equation}
then we can repeat the proof of Proposition \ref{proposition-H1-new}, and apply Lemma \ref{lemma-properties of coeffs}(h) to get the estimate
\begin{equation}
\label{estimate-contraction1}
\begin{split}
  &\|(\xi, \eta)\|_{H^1(\Om_L)}+\|\xi\|_{\mcl{W}^{1,\infty}_{\mcl{D}}(0,L)}\\
&\le C\left(\sigma(b,u_{\rm en},E_{\rm en},E_{\rm ex})\|(\xi, \eta)\|_{H^1(\Om_L)}+\|\sum_{i,j=1}^3 (a_{ij}^{P_1}-a_{ij}^{P_2})\der_{ij}\psi^{(2)}\|_{L^2(\Om_L)}\right).
\end{split}
\end{equation}
Similarly to the estimate \eqref{estimate-coeff}, we can directly check that, for any $s\in(0,L)$, the norm $\|D^2\psi^{(2)}\|_{C^0([0, s]\times \ol{\mcl{D}})}$ satisfies the estimate
\begin{equation*}
  \|D^2\psi^{(2)}\|_{C^0([0, s]\times \ol{\mcl{D}})}\le C\kappa(s)\|\psi^{(2)}\|_{H^4_*(\Om_L)}
\end{equation*}
for the function $\kappa:(0,L)\rightarrow \R$ given by \eqref{definition of kappa}. By combining this estimate with Lemma \ref{lemma-properties of coeffs}(g), we obtain that
\begin{equation*}
  \begin{split}
  \|\sum_{i,j=1}^3 (a_{ij}^{P_1}-a_{ij}^{P_2})\der_{ij}\psi^{(2)}\|_{L^2(\Om_L)}\le
  C\sigma(b,u_{\rm en},E_{\rm en},E_{\rm ex})\left(\|\xi\|_{\mcl{W}^{1,\infty}_{\mcl{D}}(0,L)}
  +\|\eta\|_{H^1(\Om_L)}\right).
  \end{split}
\end{equation*}
By substituting the above estimate into \eqref{estimate-contraction1}, we obtain the following inequality:
\begin{equation*}
\begin{split}
  &\|(\xi, \eta)\|_{H^1(\Om_L)}+\|\xi\|_{\mcl{W}^{1,\infty}_{\mcl{D}}(0,L)}\\
&\le C^{\natural}\sigma(b,u_{\rm en},E_{\rm en},E_{\rm ex})
\left(\|(\xi, \eta)\|_{H^1(\Om_L)}+\|\xi\|_{\mcl{W}^{1,\infty}_{\mcl{D}}(0,L)}\right)
\end{split}
\end{equation*}
for some constant $C^{\natural}>0$ depending only on the data, $L$ and $\bar{\epsilon}$. Therefore if the inequality
\begin{equation}
\label{condition-sigma 2-Irr}
  C^{\natural}\sigma(b,u_{\rm en},E_{\rm en},E_{\rm ex})<1
\end{equation}
holds, then we get $(\xi, \eta)\equiv (0,0)$ in $\Om_L$. Finally, the proof is completed by fixing a constant $\sigma_p>0$ sufficiently small depending only on the data, $L$ and $\bar{\epsilon}$ so that if the inequality $\sigma(b,u_{\rm en},E_{\rm en},E_{\rm ex})\le \sigma_p$ holds, then the conditions \eqref{condition-contraction1}--\eqref{condition-contraction2} and \eqref{condition-sigma 2-Irr} hold.
  \qed

\section{Nonzero vorticity flows (Proof of Theorem \ref{MainThm})}\label{Sec-HD}
Throughout Section \ref{Sec-HD}, we assume that $\Om_L$ is given by \eqref{def-Omega} with
\begin{equation*}
  \mcl{D}:=\{\rx'=(x_2, x_3)\in\R^2:|\rx'|<1\}.
\end{equation*}

\subsection{Reformulation of Problem \ref{EP-Prob}}
\label{subsection-reformulation via HD}
For each point $\rx\in \Om_L(=(0,L)\times \mcl{D})$, let us write as $\rx=(x_1,\rx')$ for $x_1\in (0,L)$ and $\rx'\in \mcl{D}$. Note that $\displaystyle{|\rx'|=r}$, $\frac{\rx'}{|\rx'|}={\bf e}_r$ and that $(\rx')^{\perp}=r{\bf e}_{\theta}$. We shall point out that these trivial equalities come in handy in various computations following in the below.

Suppose that $({\bf u}, \rho, p,\Phi)$ is an axisymmetric solution to Problem \ref{EP-Prob}. In the spirit of the Helmholtz decomposition, let us represent ${\bf u}$ as
\begin{equation*}
\label{HD for u}
{\bf u}=\nabla\varphi+\nabla\times {\bf V}\quad\mbox{in}\quad\Omega_L
\end{equation*}
for an axisymmetric function $\vphi$ and a vector field ${\bf V}$. And we further represent ${\bf V}$ as
\begin{equation}
\label{final choice for V}
{\bf V}:=h{\bf e}_r+\phi{\bf e}_{\theta}
\end{equation}
with two axisymmetric functions $(h,\phi)$. Finally, we define an axisymmetric function $\Lambda$ by
\begin{equation*}
  \Lambda:=r\der_{x_1}h.
\end{equation*}
Then the velocity field ${\bf u}$ can be represented as
\begin{equation}
\label{definition:vel-HD}
\begin{split}
  {\bf u}(\vphi, \phi{\bf e}_{\theta}, \Lambda)
  &=\nabla\vphi+\nabla\times (\phi {\bf e}_{\theta})+ \frac{\Lambda}{ r}{\bf e}_{\theta}\\
  &=\nabla\vphi+\nabla\times (\phi {\bf e}_{\theta})+\frac{\Lambda}{r^2}(\rx')^{\perp},
  \end{split}
\end{equation}
where $(\rx')^{\perp}$ represents the rotation of the vector $\rx'$ by the angle $\frac{\pi}{2}$ counterclockwise on the plane $\mcl{D}\subset \R^2$.
\medskip

Now we rewrite Problem \ref{EP-Prob} in terms of $(\vphi, \Phi, \phi, S, \Lambda)$.
\begin{problem}\label{EP-Prob-HD}
Fix an axisymmetric function $b\in C^2(\overline{\Omega_L})$. And, fix axisymmetric functions $u_{\rm en}\in C^3(\overline{\Gamma_0})$, $v_{\rm en}, w_{\rm en}, S_{\rm en}, E_{\rm en}\in C^4(\overline{\Gamma_0})$, and a function $E_{\rm ex}\in  C^4(\overline{\Gamma_L})$. Suppose that these functions satisfy all the compatibility conditions stated in Problem \ref{EP-Prob}.
And, solve the following nonlinear boundary value problem for $(\vphi, \Phi, \phi, S, \Lambda)$:
\begin{align}
\label{group1}
&\begin{cases}
{\rm div}\, \rho{\bf u} =0\\
\Delta \Phi=\rho-b\\
-\Delta(\phi{\bf e}_{\theta})=\frac{1}{{\bf u}\cdot {\bf e}_{x_1}}\left(\frac{\rho^{\gam-1}}{\gam-1}\der_rS+\frac{\Lambda}{r^2}\der_r\Lambda\right){\bf e}_{\theta}
\end{cases}\quad&\tx{in $\Om_L$},\\
\label{group2}
&\begin{cases}
{\bf m}\cdot \nabla S=0\\
{\bf m}\cdot \nabla \Lambda=0
\end{cases}\quad&\tx{in $\Om_L$},
\end{align}
\begin{equation}\label{w-ex-bd}
\left\{\begin{split}
	\varphi(0,r)=-\int_r^1 v_{\rm en}(t)dt,\,\,\partial_{x_1}\varphi=u_{\rm en}-[\nabla\times(\phi{\bf e}_{\theta})]\cdot{\bf e}_{x_1}\,\,&\mbox{on}\,\,\Gamma_0,\\
		\partial_{x_1}\Phi=E_{\rm en},\,\,\partial_{x_1}(\phi{\bf e}_{\theta})={\bf 0}\,\,&\mbox{on}\,\,\Gamma_0,\\
\partial_{r}\varphi=0,\,\,\partial_{r}\Phi=0,\,\, \phi{\bf e}_{\theta}={\bf 0}\,\,&\mbox{on}\,\,\Gamma_{{\rm w}},\\
\der_{x_1}\Phi=E_{\rm{ex}},\,\, \partial_{x_1}(\phi{\bf e}_{\theta})={\bf 0}\,\,&\mbox{on}\,\,\Gamma_L,\\
(S,\Lambda)=(S_{\rm en},|\rx'|w_{\rm en})\,\,&\mbox{on}\,\,\Gamma_0
\end{split}\right.
\end{equation}
for the density function $\rho$ and the pseudo-momentum density field ${\bf m}$ defined as follows:
\begin{equation}
\label{definition:density-HD}
  \begin{split}
  \varrho(\vphi, \Phi, \phi{\bf e}_{\theta}, \Lambda):=&
\left(\Phi-\frac 12\lvert\nabla\vphi+\nabla\times (\phi {\bf e}_{\theta})+\frac{\Lambda}{r^2}(\rx')^{\perp}\rvert^2\right)^{\frac{1}{\gam-1}},\\
  \rho(\vphi, \Phi, \phi{\bf e}_{\theta}, S, \Lambda):=&\left(\frac{\gam-1}{\gam S}\right)^{\frac{1}{\gam-1}}\varrho(\vphi, \Phi,  \phi{\bf e}, \Lambda),\\
  {\bf m}(\vphi, \Phi, \phi{\bf e}_{\theta}, \Lambda):=&\varrho(\vphi, \Phi, \phi{\bf e}_{\theta}, \Lambda){\bf u}(\vphi, \phi{\bf e}_{\theta}, \Lambda).
  \end{split}
\end{equation}

\end{problem}
One can refer to \cite{bae20183} for a detailed derivation of Problem \ref{EP-Prob-HD}.
Before we state the main proposition, which is essential to prove Theorem \ref{MainThm}, let us define norms that we shall use hereafter.
\begin{definition}
For each $k\in \mathbb{N}$, let a norm $\|\cdot\|_{\mcl{M}^k_{*}((0,L)\times \mcl{D})}$ be defined by
\begin{equation*}
  \|\cdot\|_{\mcl{M}^k_{*}((0,L)\times \mcl{D})}:=\|\cdot\|_{H^{k}_*(\Om_L)}+\|\cdot\|_{\mcl{W}^{k,\infty}_{*, \mcl{D}}(0,L)}.
\end{equation*}
More precisely, for a function $f:\Om_L\rightarrow \R$, we define
\begin{equation*}
  \|f\|_{\mcl{M}^k_{*}((0,L)\times \mcl{D})}:=\|f\|_{H^{k}_*(\Om_L)}+\|f\|_{\mcl{W}^{k,\infty}_{*, \mcl{D}}(0,L)}.
\end{equation*}
For a vector field ${\bf v}=(v_1, \cdots, v_n):\Om_L\rightarrow \R^n$, we define
\begin{equation*}
 \|{\bf v}\|_{\mcl{M}^k_{*}((0,L)\times \mcl{D})}:= \sum_{j=1}^n \|v_j\|_{\mcl{M}^k_{*}((0,L)\times \mcl{D})}.
\end{equation*}
Finally, we define
\begin{equation*}
\begin{split}
  &\mcl{M}^k_*((0,L)\times \mcl{D}):=H^k_*(\Om_L)\cap \mcl{W}^{k,\infty}_{*, \mcl{D}}(0,L),\\
  &H^k_*(\Om_L;\R^n):=\left\{{\bf v}=(v_1, \cdots, v_n): \Om_L\rightarrow \R^n\; \left\vert\; \sum_{j=1}^n\|v_j\|_{H^k_*(\Om_L)}<\infty\right.\right\},\\
  &\mcl{M}^k_*((0,L)\times \mcl{D};\R^n):=\left\{{\bf v}=(v_1, \cdots, v_n): \Om_L\rightarrow \R^n\;\left\vert\; \sum_{j=1}^n\|v_j\|_{\mcl{M}^k_*((0,L)\times \mcl{D}}<\infty\right.\right\}.
  \end{split}
\end{equation*}
Clearly, $\mcl{M}^k_*((0,L)\times \mcl{D})$ is a normed vector space thus so is $\mcl{M}^k_*((0,L)\times \mcl{D};\R^n)$.
\end{definition}

\begin{proposition}\label{Thm-HD}
Suppose that all the assumptions stated in Theorem \ref{MainThm} hold. Then, there exists a constant $L^{\ast}\in(0,\bar{L}]$ depending only on the data so that the following properties hold:
For any given $L<L^{\ast}$, one can fix a small constant $\sigma_2>0$ depending on the data and $(L,\bar{\epsilon})$ so that if the inequality
\begin{equation}
\label{condition for sigma}
\sigma(b,u_{\rm en},v_{\rm en},w_{\rm en},S_{\rm en},E_{\rm en},E_{\rm ex})\le \sigma_2
\end{equation}
holds, then Problem \ref{EP-Prob-HD} has a unique axisymmetric solution $(\vphi, \Phi, \phi, S, \Lambda)$ that satisfies the following properties:\\
\quad\\
(a) There exists a constant $C>0$ fixed depending only on the data and $(L,\bar{\epsilon})$ that satisfies the estimate
\begin{equation}\label{Thm-HD-est}
\begin{split}
&\|(\vphi-\bvphi, S-S_0, \Lambda )\|_{\mcl{M}^{4}_*((0,L)\times \mcl{D})}+\|\Phi-\bPhi\|_{H^4_*(\Om_L)}+\|\phi {\bf e}_{\theta}\|_{H^5_*(\Om_L)}\\
&\le C\sigma(b,u_{\rm en},v_{\rm en},w_{\rm en},S_{\rm en},E_{\rm en},E_{\rm ex});
\end{split}
\end{equation}
\smallskip
\quad\\
(b) Furthermore, one can adjust the estimate constant $C>0$ depending only on the data and $(L,\bar{\epsilon})$ so that the velocity field ${\bf u}$ given by \eqref{definition:vel-HD} satisfies the estimate
\begin{equation}
\label{Thm-HD-est2}
\|{\bf u}-\bar u {\bf e}_{1}\|_{\mcl{M}^{3}_*((0,L)\times \mcl{D})}\le C\sigma(b,u_{\rm en},v_{\rm en},w_{\rm en},S_{\rm en},E_{\rm en},E_{\rm ex}),
\end{equation}
where $\bar u$ is the velocity component of the background solution $(\bar{\rho}, \bar u, S_0, \bPhi)$ associated with $(\rho_0, E_0)$ with $E_0=0$ in the sense of Definition \ref{definition:background sol};
\smallskip
\quad\\
(c) There exists a positive constant $\om_0>0$ fixed depending only on the data and $(L,\bar{\epsilon})$ so that the following inequalities hold:
    \begin{equation}
    \label{HD-sol-lower bd}
   \min \left\{\rho(\vphi, \Phi, \phi{\bf e}_{\theta}, S, \Lambda),\,\, {\bf u}\cdot {\bf e}_{x_1},\,\, \frac{|{\bf u}|}{\sqrt{\gam S\rho^{\gam-1}(\vphi, \Phi, \phi{\bf e}_{\theta}, S, \Lambda)}}-1\right\}\ge \om_0\quad\tx{in $\ol{\Om_L}$}.
    \end{equation}

\end{proposition}

By the method of iterations with using Theorem \ref{Irr-MainThm}, we shall prove Proposition \ref{Thm-HD} in \S \ref{subsec-wp-HD}. Since the same approach is already introduced and extensively studied in the works \cite{bae2021structural, bae2014subsonic, bae20183}, we shall provide details only for the issues caused by the weighted Sobolev norms newly introduced in this paper. In \S \ref{Sec-Last}, we show that Theorem \ref{MainThm} follows from Proposition \ref{Thm-HD}.

\subsection{Proof of Proposition \ref{Thm-HD}}\label{subsec-wp-HD}

For a constant $\eps>0$ to be determined later, let us define two sets $\mcl{T}_{\eps}$ and $\mcl{U}_{\eps}$ by
\begin{equation}\label{iterset-T}
\begin{split}
&\mcl{T}_{\eps}:=\left\{
  S\in \mcl{M}^4_{*}((0,L)\times \mcl{D})\;\middle\vert\;
  \begin{aligned}
  &\tx{A function $S$ is axisymmetric in ${\Om_L}$},\\
  & \|S-S_0\|_{\mcl{M}^4_*((0,L)\times \mcl{D})}\le\eps,\\
  &\left\|\frac{\der_r S}{|\rx'|}\right\|_{\mcl{M}^3_*((0,L)\times \mcl{D})}\le \eps,\\
&S=S_{\rm en}\quad\tx{on $\Gam_0$},\\
 &(S,\der_{x_1} S)=(S_0,0)\,\,\tx{on }\Gamma_0^{\bar{\epsilon}},\\
 &S=S_0\,\,\tx{on $\left\{(x_1, \rx')\in \ol{\Om_L}:|\rx'|\ge1-\frac{\bar{\epsilon}}{2}\right\}$}\\
  \end{aligned}
\right\},\\
&\mcl{U}_{\eps}:=\left\{
  \Lambda\in \mcl{M}^4_{*}((0,L)\times \mcl{D})\;\middle\vert\;
  \begin{aligned}
  &\tx{A function $\Lambda$ is axisymmetric in ${\Om_L}$},\\
  & \|\Lambda\|_{\mcl{M}^4_*((0,L)\times \mcl{D})}\le\eps,\\
  &\sum_{k=0,1}\left\|\frac{\der_r^k\Lambda}{|\rx'|^{2-k}}\right\|_{\mcl{M}^3_*((0,L)\times \mcl{D})}\le \eps,\\
 &\Lambda(0, \rx')=|\rx'|w_{\rm en}(\rx')\,\,\tx{on $\Gam_0$},\\
 &(\Lambda,\der_{x_1} \Lambda)=(0,0)\,\,\tx{on }\Gamma_0^{\bar{\epsilon}},\\
 &\Lambda=0\,\,\tx{on $\left\{(x_1, \rx')\in \ol{\Om_L}:|\rx'|\ge1-\frac{\bar{\epsilon}}{2}\right\}$}\\
  \end{aligned}
\right\}.
\end{split}
\end{equation}
And, we define an iteration set $\mcl{J}_{\eps}$ by
\begin{equation}
  \label{J1}
  \mcl{J}_{\eps}:=\mcl{T}_{\eps}\times \mcl{U}_{\eps}.
\end{equation}
Note that, for any $(S, \Lambda)\in \mcl{J}_{\eps}$, it holds that
\begin{equation*}
  \der_r S=0\quad\tx{and}\quad \der_r \Lambda=0\quad\tx{on $\Gamw$}.
\end{equation*}

\subsubsection{Step 1: Second order nonlinear PDE system for $(\vphi, \Phi, {\bf W})$ associated with $(S, \Lambda)\in \mcl{J}_{\eps}$}

\begin{definition}
\label{definition-nl-operators}
Let $\vphi$ and $\Phi$ be two functions given in $\Om_L$. And, let ${\bf W}: \Om_L\rightarrow \R^3$ be a vector field.
\begin{itemize}
\item[(i)] For the density function $\rho$ and the velocity field ${\bf u}$ given in \eqref{definition:density-HD} and \eqref{definition:vel-HD}, respectively, let us define two nonlinear differential operators $\mcl{N}_1$ and $\mcl{N}_2$ associated with $(S, \Lambda, {\bf W})$ by
    \begin{equation*}
    \begin{split}
     &\mcl{N}_1^{(S, \Lambda, {\bf W}}(\vphi,\Phi):= {\rm div}\left(\rho(\vphi, \Phi,{\bf W}, S, \Lambda){\bf u}(\vphi, {\bf W}, \Lambda)\right),\\
     &\mcl{N}_2^{(S, \Lambda, {\bf W})}(\vphi,\Phi):=\Delta\Phi-\rho(\vphi, \Phi, {\bf W}, S, \Lambda)+b.
     \end{split}
    \end{equation*}

\item[(ii)] And, let us define a nonlinear vector field ${\bf F}^{(S, \Lambda, {\bf W}}(\vphi,\Phi)$ by
    \begin{equation}
    \label{definition-F-vec-field}
      {\bf F}^{(S, \Lambda, {\bf W})}(\vphi,\Phi):=\frac{\frac{\rho^{\gam-1}(\vphi, \Phi, {\bf W}, S, \Lambda)}{\gam-1}(\der_rS{\bf e}_{\theta})+\frac{\Lambda}{r^2}(\der_r\Lambda{\bf e}_{\theta})}{{\bf u}(\vphi, {\bf W}, \Lambda)\cdot {\bf e}_{x_1}}.
    \end{equation}

\end{itemize}
\end{definition}

\begin{problem}\label{nlbvp:potentials}
For fixed $(S_*, \Lambda_*)\in \mcl{J}_{\eps}$, solve the following nonlinear boundary value problem for $(\vphi, \Phi, {\bf W})$:
\begin{equation}
\label{bvp for potentials}
\begin{cases}
\mcl{N}_1^{(S_*, \Lambda_*, {\bf W})}(\vphi,\Phi)=0\\
\mcl{N}_2^{(S_*, \Lambda_*, {\bf W})}(\vphi,\Phi)=0\\
-\Delta{\bf W}={\bf F}^{(S_*, \Lambda_*, {\bf W})}(\vphi,\Phi)
\end{cases}\mbox{in}\,\,\Omega_L,
\end{equation}
\begin{equation}
\label{bc for potentials}
\left\{\begin{split}
	\varphi(0,r)=-\int_r^1 v_{\rm en}(t)dt,\,\,\partial_{x_1}\varphi+(\nabla\times{\bf W})\cdot{\bf e}_{x_1}=u_{\rm en}\,\,&\mbox{on}\,\,\Gamma_0,\\
		\partial_{x_1}\Phi=E_{\rm en},\,\,\partial_{x_1}{\bf W}={\bf 0}\,\,&\mbox{on}\,\,\Gamma_0,\\
\partial_{r}\varphi=0,\,\,\partial_{r}\Phi=0,\,\, {\bf W}={\bf 0}\,\,&\mbox{on}\,\,\Gamma_{{\rm w}},\\
\der_{x_1}\Phi=E_{\rm{ex}},\,\, \partial_{x_1}{\bf W}={\bf 0}\,\,&\mbox{on}\,\,\Gamma_L.
\end{split}\right.
\end{equation}
\end{problem}

\begin{lemma}\label{pro-1}
Fix a constant $\bar{\delta}>0$, and let $\bar{L}$ be given from Lemma \ref{Lem1}. Under the same assumptions of Theorem \ref{Thm-HD}, one can fix a constant $L^{\ast}\in(0,\bar{L}]$ depending only on the data  so that the following statement holds:
If the nozzle length $L$ of the domain $\Om_L$ satisfies the inequality $L\le L^{\ast}$, then one can fix a small constant $\bar{\eps}>0$ depending only on the data and $(L,\bar{\epsilon})$ so that whenever it holds that
\begin{equation}
\label{condition-for-eps-sig}
 \eps+\sigma(b,u_{\rm en},v_{\rm en},0,S_0 ,E_{\rm en},E_{\rm ex})\le \bar{\eps},
 \end{equation}
then, for any $(S_*,\Lambda_*)\in\mathcal{J}_{\eps}$,
Problem \ref{nlbvp:potentials} associated with $(S_*,\Lambda_*)$
has a unique axisymmetric solution $(\vphi,\Phi,{\bf W})$ that satisfies the estimate
\begin{equation}\label{pphi-est}
\begin{split}
&\|\vphi-\bvphi\|_{\mcl{M}^4_*((0,L)\times \mcl{D})}+\|\Phi-\bPhi\|_{H^4_*(\Omega_L)}+\|{\bf W}\|_{H^5_*(\Omega_L;\R^3)}\\
&\le C\left(\eps+\sigma(b,u_{\rm en},v_{\rm en},0,S_0 ,E_{\rm en},E_{\rm ex})\right)
\end{split}
\end{equation}
for a constant $C>0$ fixed depending only on the data and $(L,\bar{\epsilon})$. Furthermore, the following properties hold:
\begin{itemize}
\item[(a)] the velocity potential function $\vphi$ satisfies the compatibility condition
    \begin{equation*}
      \der_{11}\vphi=0\quad\tx{on $\Gam_0^{\bar{\epsilon}}$};
    \end{equation*}
\item[(b)] the vector field ${\bf W}$ can be represented as
\begin{equation*}
  {\bf W}=\phi{\bf e}_{\theta}
\end{equation*}
for an axisymmetric function $\phi:\ol{\Om_L}\rightarrow \R$.
\end{itemize}

\end{lemma}
In order to prove Lemma \ref{pro-1} by the method of iteration, we shall introduce an iteration set. For $\eps>0$ from \eqref{iterset-T}, let us define another iteration set
\begin{equation}\label{J3}
\begin{split}
&\mfrak{W}_{M\eps}:=\left\{{\bf W}\in H^5_*(\Omega_{L};\R^3)\;\middle\vert\;
	\begin{split}
&\tx{{\bf W} is axisymmetric in $\Om_L$},\\
	&\|{\bf W}\|_{H^5_*(\Omega_{L};\R^3)}\le M\eps,\\
	&{\bf W}={\bf 0}\mbox{ on }\Gamma_{\rm w},\,\partial_{x_1}{\bf W}={\bf 0}\mbox{ on }\Gamma_0,\\
	&\partial_r(\nabla\times{\bf W})\cdot{\bf e}_{x_1}=\partial_r(\nabla\times{\bf W})\cdot{\bf e}_{\theta}=0\mbox{ on }\Gamma_{\rm w}
	\end{split}\right\}
\end{split}
\end{equation}
for a constant $M>0$ to be fixed later.
And, for fixed ${\bf W}_*\in \mfrak{W}_{M\eps}$, let us consider a nonlinear boundary value problem for $(\vphi, \Phi)$
\begin{equation}
\label{bvp for potentials-iters}
\begin{cases}
\mcl{N}_1^{(S_*, \Lambda_*, {\bf W}_*)}(\vphi,\Phi)=0\\
\mcl{N}_2^{(S_*, \Lambda_*, {\bf W}_*)}(\vphi,\Phi)=0
\end{cases}\mbox{in}\,\,\Omega_L,
\end{equation}
\begin{equation}
\label{bc for potentials-iters}
\left\{\begin{split}
	\varphi(0,r)=-\int_r^1 v_{\rm en}(t)dt,\,\,\partial_{x_1}\varphi=u_{\rm en}-(\nabla\times{\bf W}_*)\cdot{\bf e}_{x_1}\,\,&\mbox{on}\,\,\Gamma_0,\\
		\partial_{x_1}\Phi=E_{\rm en}\,\,&\mbox{on}\,\,\Gamma_0,\\
\partial_{r}\varphi=0,\,\,\partial_{r}\Phi=0\,\,&\mbox{on}\,\,\Gamma_{{\rm w}},\\
\der_{x_1}\Phi=E_{\rm{ex}}\,\,&\mbox{on}\,\,\Gamma_L.
\end{split}\right.
\end{equation}

\begin{lemma}\label{lemma:wp of nlbvp for potentials}
 One can fix $L_*^{\rm v}\in(0,\bar{L}]$ depending on the data, and fix $\eps_1^{\rm v}>0$ sufficiently small depending only on the data and $(L,\bar{\epsilon})$ so that if $L\le L_*^{\rm v}$, and if
\begin{equation}
\label{condition-eps}
  (1+M)\eps+\sigma(b,u_{\rm en},v_{\rm en},0,S_0 ,E_{\rm en},E_{\rm ex})\le \eps_1^{\rm v},
\end{equation}
then, for every $(S_*, \Lambda_*, {\bf W}_*)\in \mcl{T}_{\eps}\times \mcl{U}_{\eps}\times \mfrak{W}_{M\eps}$, the boundary value problem \eqref{bvp for potentials-iters}--\eqref{bc for potentials-iters} acquires the unique solution $(\vphi, \Phi)$ that satisfies
\begin{equation}\label{pphi-est-1}
\begin{split}
&\|\vphi-\bvphi\|_{\mcl{M}^4_*((0,L)\times \mcl{D})}+\|\Phi-\bPhi\|_{H^4_*(\Omega_L)}\\
&\le C\left((1+M)\eps+\sigma(b,u_{\rm en},v_{\rm en},0,S_0 ,E_{\rm en},E_{\rm ex})\right)
\end{split}
\end{equation}
for $C>0$ fixed depending only on the data and $(L,\bar{\epsilon})$.
\medskip

Furthermore, the solution satisfies the following properties:
\begin{itemize}
\item[(a)] $(\vphi, \Phi)$ are axisymmetric in $\Om_L$;
\item[(b)] $\displaystyle{\der_{11}(\vphi-\bar{\vphi})=0\tx{ and } \der_{111}(\Phi-\bar{\Phi})=0}$ on $\Gam_0^{\bar{\epsilon}}$.
\end{itemize}

\begin{proof} {\textbf{1.}}
First of all, we shall rewrite \eqref{bvp for potentials-iters} similarly to the potential flow model of the Euler-Poisson system (see \eqref{back-HD}).
\begin{itemize}
\item[-]
For $z\in \R$, ${\bf u}\in \R^3$ and $S\in \R$, let us set
\begin{equation}
\label{new-def-density}
  \til{\rho}(S, z, {\bf u}):=\left(\frac{\gam-1}{\gam S}\right)^{\frac{1}{\gam-1}}\left(z-\frac 12 |{\bf u}|^2\right)^{\frac{1}{\gam-1}}.
\end{equation}
\item[-] Define a vector field ${\bf q}_*$ associated with $(\Lambda_*, {\bf W}_*)$ by
    \begin{equation*}
      {\bf q_*}:=\nabla \times {\bf W}_*+\frac{\Lambda_*}{|\rx'|^2}(\rx')^{\perp}.
    \end{equation*}
\item[-] Define two differential operators for $(\vphi, \Phi)$ by
    \begin{equation*}
\begin{split}
 \widetilde{\mcl{N}}_1^{(S_*, \Lambda_*, {\bf W}_*)}(\vphi, \Phi)
 :=&\,\til{\rho}(S_*, \Phi, \nabla\vphi+ {\bf q}_*)\Delta\vphi\\
 &+[\til{\rho}_{\bf u}(S_*, \Phi, \nabla\vphi+ {\bf q}_*)]^TD^2\vphi\cdot(\nabla\vphi+{\bf q}_*)\\
 &+\til{\rho}_z(S_*, \Phi, \nabla\vphi, {\bf q}_*)\nabla\Phi\cdot \nabla\vphi,
 \end{split}
\end{equation*}
\begin{equation*}
\begin{split}
 \widetilde{\mcl{N}}_2^{(S_*, \Lambda_*, {\bf W}_*)}(\vphi, \Phi)
 :=\Delta \Phi-\til{\rho}(S_*, \Phi, \nabla\vphi)+b.
 \end{split}
\end{equation*}
\item[-] Define two functions associated with $(S_*, \Lambda_*, {\bf W}_*)$ and $(\vphi, \Phi)$ by
\begin{equation*}
\begin{split}
  F_1^{(S_*, \Lambda_*, {\bf W}_*)}(\vphi, \Phi)
  :=&-[\til{\rho}_{\bf u}(S_*, \Phi, \nabla\vphi+ {\bf q}_*)]^T(D{\bf q}_*)(\nabla\vphi+{\bf q}_*)\\
  &-\til{\rho}_{z}(S_*, \Phi, \nabla\vphi+ {\bf q}_*)\nabla\Phi\cdot {\bf q}_*\\
  &-\til{\rho}_{S}(S_*, \Phi, \nabla\vphi+ {\bf q}_*)\nabla S_*\cdot (\nabla\vphi+{\bf q}_*),
  \end{split}
\end{equation*}
\begin{equation}
\label{definition of F2-HD}
\begin{split}
  F_2^{(S_*, \Lambda_*, {\bf W}_*)}(\vphi, \Phi)
  :=&\til{\rho}(S_*, \Phi, \nabla\vphi+ {\bf q}_*)-\til{\rho}(S_*, \Phi, \nabla\vphi).
  \end{split}
\end{equation}

\end{itemize}
Finally, we rewrite \eqref{bvp for potentials-iters} as
\begin{equation}
\label{bvp for potentials-iters-new}
\begin{cases}
\widetilde{\mcl{N}}_1^{(S_*, \Lambda_*, {\bf W}_*)}(\vphi,\Phi)=F_1^{(S_*, \Lambda_*, {\bf W}_*)}(\vphi, \Phi)\\
\widetilde{\mcl{N}}_2^{(S_*, \Lambda_*, {\bf W}_*)}(\vphi,\Phi)=F_2^{(S_*, \Lambda_*, {\bf W}_*)}(\vphi, \Phi)
\end{cases}\mbox{in}\,\,\Omega_L.
\end{equation}

{\textbf{2.}} For a small constant $\delta>0$ to be determined later, let $\mcl{H}_{\delta}^P$ and $\mcl{I}_{\delta}^E$ be defined as in \eqref{definition-potentials}. And, we define an iteration set $\mcl{K}_{\delta}^{A}$ by
\begin{equation}\label{iterset-potentials-full case}
 \mcl{K}_{\delta}^A:=\left\{
(\vphi, \Phi)\in \mcl{M}^4_{*}((0,L)\times \mcl{D})\times H^4_*(\Om_L)\;\middle\vert\; \begin{split}
&(\vphi-\bar{\vphi}, \Phi-\bar{\Phi})\in
\mcl{H}_{\delta}^P\times \mcl{I}_{\delta}^E\\
&\tx{$(\vphi, \Phi)$ are axisymmetric in $\Om_L$}\end{split}\right\}
\end{equation}
The set $\mcl{K}_{\delta}^{A}$ is nonempty because $(\bar{\vphi}, \bar{\Phi})\in \mcl{K}_{\delta}^A$.
\smallskip

Note that $(S_*, \Lambda_*,{\bf W}_*,v_{\rm en}, w_{\rm en})$ satisfy the following compatibility conditions:
\begin{itemize}
\item[(i)] On $\Gamw$:
\begin{itemize}
\item[-] $\displaystyle{(S_*,\der_{r}S_*)}=(S_0,0);$
\item[-] $\displaystyle{{\bf W}_*={\bf 0},\,\, \der_r \left((\nabla \times{\bf W}_*)\cdot{\bf e}_{x_1}\right)=0.}$
\end{itemize}
\item[(ii)] On $\Gam_0^{\bar{\epsilon}}$:
\begin{itemize}
\item[-] $\displaystyle{(S_*,\der_{x_1}S_*)}=(S_0,0);$
\item[-] $\displaystyle{(\Lambda_*,\der_{x_1}\Lambda_*)=(0,0)};$
\item[-] $\displaystyle{\der_{x_1}{\bf W}_*=0};$
\item[-] $\displaystyle{v_{\rm en}=w_{\rm en}=0}.$
\end{itemize}
\end{itemize}
Then one can directly check that if $\delta>0$ is fixed sufficiently small, and if $(\tilde{\vphi}, \til{\Phi})\in \mcl{K}_{\delta}^A$,
then, for $(F_1, F_2):=(F_1^{(S_*, \Lambda_*, {\bf W}_*)}, F_2^{(S_*, \Lambda_*, {\bf W}_*)})(\tilde{\vphi}, \til{\Phi})$, the following compatibility conditions hold:
\begin{align}
\label{comp for F-1}
\der_rF_1=\der_r F_2=0\quad&\tx{on $\Gamw$},\\
\label{comp for F-2}
  F_1=\der_{x_1}F_2=0\quad&\tx{on $\Gam_0^{\bar{\epsilon}}$},
\end{align}
which corresponds to the statement (e) and (f) of Lemma \ref{lemma-properties of coeffs}.

In \eqref{comp for F-1}, the condition $\der_r F_2=0$ on $\Gamw$ can be checked by a direct computation with using \eqref{new-def-density}, the compatibility conditions for $(S_*,\Lambda_*,{\bf W}_*)$ on $\Gamw$(see \eqref{iterset-T} and \eqref{J3}), and the compatibility conditions of $(\tilde{\vphi}, \til{\Phi})$(see \eqref{definition-potentials}). In particular, one of essential ingredients used to verify $\der_rF_2=0$ on $\Gamw$ is \begin{equation*}
\begin{split}
  &(\nabla \til{\vphi}+ {\bf q}_*)\cdot \der_r (\nabla \til{\vphi}+ {\bf q}_*)\\
  &=(\der_{x_1} \til{\vphi}+ {\bf q}_*\cdot {\bf e}_1)\underset{(=0\,\,\tx{on $\Gamw$})}{\underbrace{\der_r(\nabla\times {\bf W}_*\cdot {\bf e}_1)}}
  +({\bf q}_*\cdot {\bf e}_{\theta})\underset{(=0\,\,\tx{on $\Gamw$})}{\underbrace{\der_r(\nabla\times {\bf W}_*\cdot {\bf e}_{\theta})}}=0\quad\tx{on $\Gamw$.}
  \end{split}
\end{equation*}

The compatibility condition $\der_rF_1=0$ on $\Gamw$ can be directly checked if we show that
\begin{equation}
\label{nontrivial term in F1}
  \der_r G=0\quad\tx{on $\Gamw$}
\end{equation}
for
\begin{equation*}
G:=[\til{\rho}_{\bf u}(S_*, \til{\Phi}, \nabla\til{\vphi}+ {\bf q}_*)]^T(D{\bf q}_*)(\nabla\til{\vphi}+{\bf q}_*).
\end{equation*}
Set ${\bf V}:=(S_*, \til{\Phi}, \nabla\til{\vphi}+ {\bf q}_*)$. In order to compute $\der_r G$ on $\Gamw$, we use that, for any $\rx\in \left\{(x_1, \rx')\in \ol{\Om_L}:\rx'\in \mcl{D},|\rx'|\ge1-\frac{\bar{\epsilon}}{2}\right\}$
\begin{equation}
\label{representation of G}
\begin{split}
  G=&\left(\til{\rho}_{\bf u}({\bf V})\cdot \der_1(\nabla\times{\bf W}_*)\right)(\nabla\til{\vphi}+\nabla\times {\bf W}_*)\cdot {\bf e}_1\\
  &+\left(\til{\rho}_{\bf u}({\bf V})\cdot \der_r(\nabla\times{\bf W}_*)\right)(\nabla\til{\vphi}+\nabla\times {\bf W}_*)\cdot {\bf e}_r\\
  &+\left(\til{\rho}_{\bf u}({\bf V})\cdot \der_{\theta}(\nabla\times{\bf W}_*)\right)(\nabla\til{\vphi}+\nabla\times {\bf W}_*)\cdot {\bf e}_{\theta}=:g_1+g_r+g_{\theta}.
  \end{split}
\end{equation}

First of all, we shall explain how to show that $\der_r g_1=0$ on $\Gamw$. It directly follows from \eqref{new-def-density} that
\begin{equation*}
  g_1=-\frac{1}{\gam S_*\til{\rho}^{\gam-2}({\bf V})}\underset{(=:g_{1,1})}{\underbrace{\left((\nabla\til{\vphi}+\nabla\times{\bf W}_*)\cdot\der_1(\nabla\times {\bf W}_*)\right)}}\underset{(=:g_{1,2})}{\underbrace{(\nabla\til{\vphi}+\nabla\times {\bf W}_*)\cdot {\bf e}_1}}.
\end{equation*}
Due to the axisymmetry of $\til{\vphi}$, the slip boundary condition $\der_r\til{\vphi}=0$ on $\Gamw$ and the compatibility conditions of ${\bf W}_*$ on $\Gamw$(see \eqref{J3}),
\begin{equation*}
  \der_r (\nabla\til{\vphi}+\nabla\times{\bf W}_*)=(\der_{rr}\til{\vphi}+\der_r\left((\nabla\times {\bf W}_*)\cdot{\bf e}_r\right)){\bf e}_r\quad\tx{on $\Gamw$}.
\end{equation*}
By using the axisymmetry of ${\bf W}_*$ and the boundary condition ${\bf W}_*=0$ on $\Gamw$, we get
\begin{equation*}
\nabla\times{\bf W}_*=\der_r ({\bf W}_*\cdot {\bf e}_1){\bf e}_{\theta}  \quad\tx{on $\Gamw$},
\end{equation*}
and this yields
\begin{equation*}
  \der_r (\nabla\til{\vphi}+\nabla\times{\bf W}_*) \cdot \der_1(\nabla\times{\bf W}_*)=0\quad\tx{on $\Gamw$}.
\end{equation*}
Due to the compatibility conditions $\partial_r(\nabla\times{\bf W}_*)\cdot{\bf e}_{x_1}=\partial_r(\nabla\times{\bf W}_*)\cdot{\bf e}_{\theta}=0$ on $\Gamw$,
\begin{equation*}
 \der_{r1}(\nabla\times{\bf W}_*)=\der_{1r}((\nabla\times {\bf W}_*)\cdot {\bf e}_r){\bf e}_r\quad\tx{on $\Gamw$},
\end{equation*}
from which it directly follows that
$(\nabla\til{\vphi}+\nabla\times{\bf W}_*) \cdot \der_{r1}(\nabla\times{\bf W}_*)=0$ on $\Gamw$, thus we obtain that
\begin{equation*}
  \der_rg_{1,1}=0\quad\tx{on $\Gamw$}.
\end{equation*}
By using the slip boundary condition $\der_r\til{\vphi}=0$ on $\Gamw$, and the compatibility condition $\partial_r(\nabla\times{\bf W}_*)\cdot{\bf e}_{x_1}=0$ on $\Gamw$, we easily get $\der_{r} g_{1,2}=0$ on $\Gamw$. Therefore we conclude that
\begin{equation*}
  \der_r g_1=0\quad\tx{on $\Gamw$}.
\end{equation*}
With more careful but tedious computations, one can similarly check that $\der_r g_r=\der_r g_{\theta}=0$ on $\Gamw$, then \eqref{nontrivial term in F1} is obtained. This verifies \eqref{comp for F-1} completely.
\medskip

It is even easier to check the compatibility conditions stated in \eqref{comp for F-2} by using the compatibility conditions of $(\tilde{\vphi}, \til{\Phi})$(see \eqref{definition-potentials}) and the compatibility conditions for $(S_*, \Lambda_*, {\bf W}_*)$ on $\Gam_0^{\bar{\epsilon}}$($\Gam_0$ for ${\bf W_*}$) stated in \eqref{iterset-T} and \eqref{J3}.
\medskip

{\textbf{3.}} For the rest of the proof, $(S_*, \Lambda_*, {\bf W}_*)$ remain to be fixed.

{\textbf{(3-1)}} Now, we introduce a linear boundary value problem associated with each $(\til{\vphi}, \til{\Phi})\in \mcl{K}_{\delta}^A$ so that we prove Lemma \ref{lemma:wp of nlbvp for potentials} by iterations. Given ${\bm\xi}=(\til{\vphi}, \til{\Phi})\in \mcl{K}_{\delta}^A$, let us define two linear differential operators as follows: For
\begin{equation*}
{\bf Q}:=(S_*, \til{\Phi}, \nabla\til{\vphi}+{\bf q}_*)\quad\tx{and}\quad
{\bf R}_t:=t(\til{\Phi}, \nabla\til{\vphi})+(1-t)(\bar{\Phi}, \nabla\bar{\vphi}),
\end{equation*}
let us define $\mcl{L}_1^{\bm\xi}$ and $\mcl{L}_2^{\bm\xi}$ by
\begin{equation*}
  \begin{split}
  \mcl{L}_1^{\bm\xi}(\psi, \Psi):=&\til{\rho}({\bf Q})\Delta \psi+[\til{\rho}_{\bf u}({\bf Q})]^T [D^2\psi] (\nabla \til{\vphi}+{\bf q}_*)+\left({\bf e}_{x_1}\cdot\til{\rho}_{\bf u}(S_0, \bar{\Phi}, \nabla\bar{\vphi})\right)\bar u'\der_1 \psi\\
  &+\bar u'(\nabla\til{\vphi}+{\bf q}_*)\cdot {\bf e}_{x_1}
  \biggl\{\nabla \psi \cdot\int_0^1 \der_{\bf u}\left({\bf e}_{x_1}\cdot\til{\rho}_{\bf u}\right)(S_0,{\bf R}_t)\,dt\\
  &\qquad\qquad\qquad\qquad\qquad+\Psi\int_0^1 \der_z\left({\bf e}_{x_1}\cdot\til{\rho}_{\bf u}\right)(S_0,{\bf R}_t)\,dt\biggr\},\\
 \mcl{L}_2^{\bm\xi}(\psi,\Psi):=&\Delta \Psi -(\Psi, \nabla \psi)\cdot \int_0^1 (\der_z, \der_{\bf u})\til{\rho}(S_0,{\bf R}_t)\,dt.
 \end{split}
\end{equation*}
Next, we define two functions $f_1^{\bm\xi}$ and $f_2^{\bm\xi}$ by
\begin{equation*}
  \begin{split}
  f_1^{\bm\xi}:=&F_1^{(S_*, \Lambda_*, {\bf W}_*)}(\til{\vphi}, \til{\Phi})-\left[{\bf e}_{x_1}\cdot \til{\rho}_{\bf u}(S_0, \bar{\Phi}, \nabla\bar{\vphi})\right](\bar{u}'{\bf q}_*\cdot{\bf e}_{x_1})\\
  &-\left[{\bf e}_{x_1}\cdot\left(\til{\rho}_{\bf u}(S_*, \til{\Phi}, \nabla\til{\vphi}+{\bf q}_*)-\til{\rho}_{\bf u}(S_0, \til{\Phi}, \nabla\til{\vphi})\right)\right] [\bar{u}'(\nabla\bar{\vphi}+{\bf q}_*)\cdot {\bf e}_{x_1}],\\
  f_2^{\bm\xi}:=&F_2^{(S_*, \Lambda_*, {\bf W}_*)}(\til{\vphi}, \til{\Phi})-(b-b_0)+\left(\til{\rho}(S_*, \til{\Phi}, \nabla\til{\vphi})-\til{\rho}(S_0, \til{\Phi}, \nabla\til{\vphi})\right).
  \end{split}
\end{equation*}
Finally, we introduce a linear boundary value problem for $(\psi,\Psi)$ associated with ${\bm\xi}\in \mcl{K}_{\delta}^A$:
\begin{equation}
\label{lbvp-potential part-full system}
\begin{split}
  &\begin{cases}
    \mcl{L}_1^{\bm\xi}(\psi,\Psi)=f_1^{\bm\xi}\\
    \mcl{L}_2^{\bm\xi}(\psi,\Psi)=f_2^{\bm\xi}
  \end{cases}\quad\tx{in $\Om_L$},\\
  &\begin{cases}
	\psi(0,r)=-\int_r^1 v_{\rm en}(t)dt,\,\,\partial_{x_1}\psi=u_{\rm en}-u_0-(\nabla\times{\bf W}_*)\cdot{\bf e}_{x_1}\,\,&\mbox{on}\,\,\Gamma_0,\\
		\partial_{x_1}\Psi=E_{\rm en}\,\,&\mbox{on}\,\,\Gamma_0,\\
\partial_{r}\psi=0,\,\,\partial_{r}\Psi=0\,\,&\mbox{on}\,\,\Gamma_{{\rm w}},\\
\der_{x_1}\Psi=E_{\rm{ex}}-\bar{E}(L)\,\,&\mbox{on}\,\,\Gamma_L.
\end{cases}
\end{split}
\end{equation}
\smallskip

{\textbf{(3-2)}} By repeating the proof of Proposition \ref{lemma-wp of lbvp for potential} with minor adjustments, we can fix a constant $L_*^{\rm v}\in(0, \bar L]$ depending on the data, and two constants $\eps_0^{\rm v}>0$ and $C_{\rm v}>0$ depending on the data and $(L,\bar{\epsilon})$ so that if $L\le L_*^{\rm v}$, and if
\begin{equation*}
  \max\{\delta, (1+M)\eps\}\le \eps_0^{\rm v},
\end{equation*}
then, for each ${\bm\xi}\in \mcl{K}_{\delta}^A$, the linear boundary value problem \eqref{lbvp-potential part-full system} has a unique solution $(\psi,\Psi)$ that satisfies the estimate
\begin{equation}
\label{estimate-lbvp-pot-sol}
\|\psi\|_{\mcl{M}^4_*((0,L)\times \mcl{D})}+\|\Psi\|_{H^4_*(\Omega_L)}\le C_{\rm v}\left((1+M)\eps+\sigma(b,u_{\rm en},v_{\rm en},0,S_0 ,E_{\rm en},E_{\rm ex})\right),
\end{equation}
and the compatibility conditions
\begin{equation*}
 \der_{x_1}^{k-1}\psi=0\,\,\tx{and}\,\, \der^k_{x_1}\Psi=0\quad\tx{on $\Gamma_0^{\bar{\epsilon}}$ in the trace sense for $k=1,3$.}
\end{equation*}

\smallskip

{\textbf{(3-3)}} {\emph{Claim: The solution $(\psi, \Psi)$ is axisymmetric.}}

Fix $\beta\in(0, 2\pi)$. Given a point $\rx=(x_1, \rx')\in \Om_L(=(0,L)\times \mcl{D})$, let $\rx'_{\beta}\in \R^2$ denote the vector obtained by rotating $\rx'$ on $\R^2$ by $\beta$ counterclockwise. And we shall denote $(x_1, \rx'_{\beta})$ by $\rx_{\beta}$. For a function $g:\ol{\Om_L}\rightarrow \R$ or a vector field ${\bf V}:\ol{\Om_L}\rightarrow \R^3$, let us set
\begin{equation*}
  g^{\beta}(\rx):=g(\rx_{\beta}),\quad
  {\bf V}^{\beta}(\rx):={\bf V}(\rx_{\beta}).
\end{equation*}

From the definitions of the iteration sets stated in \eqref{iterset-T}, \eqref{J3} and \eqref{iterset-potentials-full case}, it follows that
\begin{equation*}
  (f_1^{\bm\xi}, f_2^{\bm\xi})^{\beta}=(f_1^{\bm\xi}, f_2^{\bm\xi})\quad\tx{in $\Om_L$}.
\end{equation*}
And, one can directly check that
\begin{equation*}
  \left(\mcl{L}_2^{\bm\xi}(\psi, \Psi)\right)^{\beta}=\mcl{L}_2^{\bm\xi}(\psi^{\beta}, \Psi^{\beta})\quad\tx{in $\Om_L$}.
\end{equation*}
Also, it is clear that
\begin{equation*}
\begin{split}
  &\left(\mcl{L}_1^{\bm\xi}(\psi, \Psi)-[\til{\rho}_{\bf u}({\bf Q})]^T [D^2\psi] (\nabla \til{\vphi}+{\bf q}_*)\right)^{\beta}\\
  &=\mcl{L}_1^{\bm\xi}(\psi^{\beta}, \Psi^{\beta})-[\til{\rho}_{\bf u}({\bf Q})]^T [D^2\psi^{\beta}] (\nabla \til{\vphi}+{\bf q}_*)\qquad\tx{in $\Om_L$}.
  \end{split}
\end{equation*}
Similarly to \eqref{representation of G}, observe that
\begin{equation*}
  \begin{split}
  &\left([\til{\rho}_{\bf u}({\bf Q})]^T [D^2\psi] (\nabla \til{\vphi}+{\bf q}_*)\right)^{\beta}=\sum_{j=1}^3\til{\rho}_{\bf u}({\bf Q^{\beta}})\cdot \der_{j}(\nabla\psi^{\beta}) [(\nabla \til{\vphi}^{\beta}+({\bf q}_*)^{\beta})\cdot {\bm\zeta}_j]
  \end{split}
\end{equation*}
for $(\der_1, \der_2, \der_3):=(\der_{x_1}, \der_r, \der_{\theta})$ and $({\bm\zeta}_1, {\bm\zeta}_2, {\bm\zeta}_3 ):=({\bf e}_1, {\bf e}_r, {\bf e}_{\theta})$. Due to the axisymmetry of the vector fields ${\bf Q}$ and $\nabla\til{\vphi}+{\bf q}_*$, the representation given in the right above directly yields that
\begin{equation*}
  \left([\til{\rho}_{\bf u}({\bf Q})]^T [D^2\psi] (\nabla \til{\vphi}+{\bf q}_*)\right)^{\beta}=[\til{\rho}_{\bf u}({\bf Q})]^T[D^2\psi^{\beta}](\nabla \til{\vphi}+{\bf q}_*)\quad\tx{in $\Om_L$}.
\end{equation*}
Therefore, we obtain that
\begin{equation*}
   \left(\mcl{L}_1^{\bm\xi}(\psi, \Psi)\right)^{\beta}=\mcl{L}_1^{\bm\xi}(\psi^{\beta}, \Psi^{\beta})\quad\tx{in $\Om_L$}.
\end{equation*}

Note that all the boundary data are axisymmetric, thus $(\psi^{\beta}, \Psi^{\beta})$ satisfy all the boundary conditions stated in \eqref{lbvp-potential part-full system}. Based on the observations made so far, we conclude that $(\psi^{\beta}, \Psi^{\beta})$ solve the linear boundary value problem \eqref{lbvp-potential part-full system}, and this implies that
\begin{equation*}
  (\psi, \Psi)=(\psi^{\beta}, \Psi^{\beta})\quad\tx{in $\ol{\Om_L}$}
\end{equation*}
for all $\beta\in(0, 2\pi)$, due to the uniqueness of a solution. This verifies the claim.
\medskip

{\textbf{4.}}
Now we fix the constant $\delta$ in \eqref{iterset-potentials-full case} as
\begin{equation*}
  \delta:=2C_{\rm v}\left((1+M)\eps+\sigma(b,u_{\rm en},v_{\rm en},0,S_0 ,E_{\rm en},E_{\rm ex})\right).
\end{equation*}
If the constants $M$, $\eps$ and $\sigma(b,u_{\rm en},v_{\rm en},0,S_0 ,E_{\rm en},E_{\rm ex})$ satisfy the condition
\begin{equation*}
  \max\{2C_{\rm v}\left((1+M)\eps+\sigma(b,u_{\rm en},v_{\rm en},0,S_0 ,E_{\rm en},E_{\rm ex})\right), (1+M)\eps\}\le \eps_0^{\rm v},
\end{equation*}
then, for each ${\bm\xi}\in \mcl{K}_{\delta}^{A}$, the associated linear boundary value problem \eqref{lbvp-potential part-full system} has a unique solution $(\psi, \Psi)^{({\bm\xi})}$ that is axisymmetric, and that satisfies the estimate \eqref{estimate-lbvp-pot-sol}. Then we can define an iteration map $\mathfrak{I}_{\rm v}:\mcl{K}_{\delta}^{A}\rightarrow \mcl{K}_{\delta}^{A}$ by
\begin{equation*}
\mathfrak{I}_{\rm v}:{\bm\xi}\in \mcl{K}_{\delta}^{A}\mapsto
(\psi, \Psi)^{({\bm\xi})}+(\bar{\vphi}, \bar{\Phi}).
\end{equation*}
By repeating the arguments of steps 1--4 given in the proof of Theorem \ref{Irr-MainThm} with minor adjustments (see \S \ref{subsection-proof of theorem-potential}), we can fix a constant $\eps_1^{\rm v}\in(0, \eps_0^{\rm v}]$ depending only on the data and $(L,\bar{\epsilon})$ so that if the condition \eqref{condition-eps} holds, then, for each $(S_*, \Lambda_*, {\bf W}_*)\in \mcl{T}_{\eps}\times \mcl{U}_{\eps}\times \mfrak{W}_{M\eps}$, the associated nonlinear boundary value problem \eqref{bvp for potentials-iters}--\eqref{bc for potentials-iters} has a unique solution that satisfies the estimate \eqref{pphi-est-1}, and that is axisymmetric in $\Om_L$.

\end{proof}
\end{lemma}

\begin{proof}[Proof of Lemma \ref{pro-1}]
In this proof, we keep assuming that the constants $M$, $\eps$ and $\sigma(b,u_{\rm en},v_{\rm en},0,S_0 ,E_{\rm en},E_{\rm ex})$ satisfy the condition \eqref{condition-eps}. Throughout the proof, $(S_*, \Lambda_*)$ remains to be fixed in $\mcl{T}_{\eps}\times \mcl{U}_{\eps}$.
\smallskip

{\textbf{1.}} For a fixed $\widehat{\bf W}\in \mfrak{W}_{M\eps}$, let $(\hat{\vphi}, \hat{\Phi})$ be the solution to the nonlinear boundary value problem \eqref{bvp for potentials-iters}--\eqref{bc for potentials-iters} associated with $(S_*, \Lambda_*, \widehat{\bf W})$. The unique existence of a solution $(\hat{\vphi}, \hat{\Phi})$ directly follows from Lemma \ref{lemma:wp of nlbvp for potentials}. Moreover, the solution is axisymmetric in $\Om_L$.

For the vector field ${\bf F}^{(S_*, \Lambda_*,\widehat{{\bf W}})}(\hat{\vphi},\hat{\Phi})$ given by \eqref{definition-F-vec-field}, we now consider a linear boundary value problem for ${\bf W}$:
\begin{equation}
\label{lbvp for W}
\begin{cases}
-\Delta{\bf W}={\bf F}^{(S_*, \Lambda_*,\widehat{{\bf W}})}(\hat{\vphi},\hat{\Phi})\quad&\mbox{in $\Om_L$},\\
\partial_{x_1}{\bf W}={\bf 0}\quad&\mbox{on $\Gamma_0$},\\
{\bf W}={\bf 0}\quad&\mbox{on $\Gamma_{{\rm w}}$},\\
 \partial_{x_1}{\bf W}={\bf 0}\quad&\mbox{on $\Gamma_L$}.
\end{cases}
\end{equation}

By using \eqref{iterset-T}, \eqref{definition-F-vec-field}, \eqref{J3} and \eqref{pphi-est-1}, we can directly check that ${\bf F}^{(S_*, \Lambda_*,\widehat{{\bf W}})}(\hat{\vphi},\hat{\Phi})$ satisfies the following properties:
\begin{itemize}
\item[(i)] There exists a constant $k_0>0$ depending only on the data and $(L, \bar{\epsilon})$ so that
    \begin{equation*}
    \begin{split}
      \|{\bf F}^{(S_*, \Lambda_*,\widehat{{\bf W}})}(\hat{\vphi},\hat{\Phi})\|_{H^3_*(\Om_L)}&\le k_0\left(\left\|\frac{\der_rS}{|\rx'|}\right\|_{\mcl{M}^3_*(0,L)\times \mcl{D}}+\left\|\frac{\der_r\Lambda}{|\rx'|}\right\|_{\mcl{M}^3_*(0,L)\times \mcl{D}}\right)\\
      &\le 2k_0\eps;
      \end{split}
    \end{equation*}
\item[(ii)] On $(\Gam_0\cup\Gam_L)\cap\{|\rx'|>1-\frac{\bar{\epsilon}}{2}\}$, the compatibility condition
    \begin{equation*}
    \der_{x_1}{\bf F}^{(S_*, \Lambda_*,\widehat{{\bf W}})}(\hat{\vphi},\hat{\Phi})=0
    \end{equation*}
   holds.
\end{itemize}
Then we apply the method of reflection and standard regularity theorems for a linear elliptic boundary value problem (e.g. see \cite[Theorems 8.10 and 8.13]{GilbargTrudinger} ), and the uniqueness of a weak solution, then adjust the argument of Step 1 in the proof of Lemma \ref{lemma-H4 estimate-final} to show that the linear elliptic boundary value problem \eqref{lbvp for W} acquires a unique solution ${\bf W}\in H^4(\Om_L)\cap H^5_{\rm{loc}}(\Om_L)$, and that there exists a constant $C_{\rm v}$ fixed depending only on the data and $(L,\bar{\epsilon})$ so that the following estimate holds:
\begin{equation}
\label{estimate of W}
  \|{\bf W}\|_{H^5_{*}(\Om_L)}\le C_{\rm v}k_0\eps.
\end{equation}
\smallskip

{\textbf{2.}} {\emph{Claim: The solution ${\bf W}$ to \eqref{lbvp for W} is axisymmetric, and satisfies the compatibility conditions:}}
\begin{equation*}
  \der_r(\nabla\times {\bf W})\cdot {\bf e}_{x_1}=0\quad\tx{and}\quad \der_r(\nabla\times{\bf W})\cdot{\bf e}_{\theta}=0\quad\tx{on $\Gamw$}.
\end{equation*}

Note that ${\bf F}^{(S_*, \Lambda_*,\widehat{{\bf W}})}(\hat{\vphi},\hat{\Phi})=\mfrak{g}{\bf e}_{\theta}$ for
\begin{equation*}
  \mfrak{g}:=\frac{\frac{\rho^{\gam-1}(\hat{\vphi},\hat{\Phi}, \widehat{\bf W}, S_*, \Lambda_*)}{\gam-1}\der_rS_*+\frac{\Lambda_*}{|\rx'|^2}\der_r\Lambda_*}{{\bf u}(\hat{\vphi}, \widehat{\bf W}, \Lambda_*)\cdot {\bf e}_{x_1}}.
\end{equation*}
According to \cite[Proposition 3.3]{bae20183}, if ${\bf F}^{(S_*, \Lambda_*,\widehat{{\bf W}})}(\hat{\vphi},\hat{\Phi})=\mfrak{g}{\bf e}_{\theta}\in C^{\alp}(\ol{\Om_L};\R^3)$ for some $\alp\in(0,1)$, then the solution ${\bf W}$ to \eqref{lbvp for W} can be represented as
\begin{equation}
\label{representation of W}
  {\bf W}=\phi{\bf e}_{\theta}\quad\tx{in $\ol{\Om_L}$}
\end{equation}
for the function $\phi:\ol{\Om_L}\rightarrow \R$ satisfying the following properties:
\begin{itemize}
\item[(i)] The function $\phi$ is axisymmetric in $\ol{\Om_L}$;
\item[(ii)] As a function of $(x_1, r)\in [0,L]\times [0,1]$, $\phi\in C^{2,\alp}([0,L]\times [0,1])$ solves
    \begin{equation}
    \label{lbvp for phi}
      \begin{cases}
      -\left(\partial_{x_1x_1}+\frac{1}{r}\partial_r(r\partial_r)-\frac{1}{r^2}\right)\phi=\mathfrak{g}\quad&\mbox{in $R_L:=(0,L)\times (0,1)$},\\
      \phi=0\quad&\mbox{on $\der R_L\cap \{r=0,1\}$},\\
      \der_{x_1}\phi=0\quad&\mbox{on $\der R_L\cap\{x_1=0,L\}$}.
      \end{cases}
    \end{equation}
\end{itemize}

Clearly, it follows from ${\bf F}^{(S_*, \Lambda_*,\widehat{{\bf W}})}(\hat{\vphi},\hat{\Phi})\in H^3_*(\Om_L)(\subset H^2(\Om_L)\subset W^{1,6}(\Om_L))$ that ${\bf F}^{(S_*, \Lambda_*,\widehat{{\bf W}})}(\hat{\vphi},\hat{\Phi})\in C^{1/2}(\ol{\Om_L})$, thus ${\bf W}$ can be represented as \eqref{representation of W}. This directly implies that ${\bf W}$ is axisymmetric in $\Om_L$.

A straightforward computation yields
\begin{equation}
\label{curl of W}
  \nabla\times {\bf W}=\nabla\times (\phi {\bf e}_{\theta})=\left(\der_r\phi+\frac{\phi}{r}\right){\bf e}_{x_1}-\der_1\phi\,{\bf e}_r\quad\tx{in $\ol{\Om_L}$}.
\end{equation}
Then we immediately obtain that
\begin{equation*}
  \der_r(\nabla \times {\bf W})\cdot {\bf e}_{\theta}=0\quad\tx{on $\Gamw$}.
\end{equation*}
Now we shall compute $\der_r(\nabla\times {\bf W})\cdot{\bf e}_{x_1}=\der_{rr}\phi+\der_r\left(\frac{\phi}{r}\right)$ on $\Gamw$. Since $\phi$, as a function of $(x_1, r)$, is $C^2$ up to $\der R_L$, it follows from the equation $\left(\partial_{x_1x_1}+\frac{1}{r}\partial_r(r\partial_r)
-\frac{1}{r^2}\right)\phi=-\mathfrak{g}$ in $\ol{\Om_L}$ that
\begin{equation*}
  \der_r(\nabla\times {\bf W})\cdot{\bf e}_{x_1}=-\mfrak g-\der_{x_1x_1}\phi\quad\tx{on $\Gamw$}.
\end{equation*}
Due to the boundary condition $\phi=0$ on $r=1(\Gam_w)$, it clearly holds that $\der_{x_1x_1}\phi=0$ on $\Gamw$. In addition, the compatibility conditions $\der_rS_*=\der_r\Lambda_*=0$ on $\Gamw$, prescribed in \eqref{iterset-T}, directly imply that $\mfrak{g}=0$ on $\Gamw$, therefore we conclude that the compatibility condition $\der_r(\nabla\times {\bf W})\cdot{\bf e}_{x_1}=0$ holds on $\Gamw$. The claim is verified.
\smallskip

{\textbf{3.}} Now we define an iteration map $\mcl{V}^{(S_*, \Lambda_*)}: \mfrak{W}_{M\eps}\rightarrow H^5_*(\Om_L;\R^3)$ by
\begin{equation*}
  \mcl{V}^{(S_*, \Lambda_*)} (\widehat{{\bf W}}):={\bf W}
\end{equation*}
for the solution ${\bf W}$ to \eqref{lbvp for W}. Then $\mcl{V}^{(S_*, \Lambda_*)}$ maps the iteration set $\mfrak{W}_{M\eps}$ into itself if we fix the constant $M$ given in \eqref{J3} as
\begin{equation}
\label{choice of M}
  M=2C_{\rm v}k_0.
\end{equation}
Now, suppose that $\eps$ and $\sigma(b,u_{\rm en},v_{\rm en},0,S_0 ,E_{\rm en},E_{\rm ex})$ are fixed sufficiently small to satisfy the inequality
\begin{equation*}
  (1+2C_{\rm v}k_0)\eps+\sigma(b,u_{\rm en},v_{\rm en},0,S_0 ,E_{\rm en},E_{\rm ex})\le \eps_1^{\rm v}
\end{equation*}
so that the condition \eqref{condition-eps} holds. By the Sobolev embedding theorem and the Arzel\`a-Ascoli theorem, the iteration set $\mfrak{W}_{M\eps}$ is compact in $C^{2, 1/4}(\ol{\Om_L})$. And, we can easily adjust Step 2 of the proof of Theorem \ref{Irr-MainThm}(see \S \ref{subsection-proof of theorem-potential}) to show that  the map $ \mcl{V}^{(S_*, \Lambda_*)}$ is continuous in $C^{2, 1/4}(\ol{\Om_L})$. Then the Schauder fixed point theorem combined with Lemma \ref{lemma:wp of nlbvp for potentials} implies the existence of a solution $(\vphi, \Phi, {\bf W})$ to Problem \ref{nlbvp:potentials} associated with $(S_*, \Lambda_*)\in \mcl{J}_{\eps}$.
\smallskip

{\textbf{4.}} Let $(\vphi^{(1)}, \Phi^{(1)}, {\bf W}^{(1)})$ and  $(\vphi^{(2)}, \Phi^{(2)}, {\bf W}^{(2)})$ be two solutions of Problem \ref{nlbvp:potentials} associated with $(S_*, \Lambda_*)\in \mcl{J}_{\eps}$. And, suppose that they satisfy the estimate \eqref{pphi-est-1}. Next, let us set
\begin{equation*}
\eta:=\|(\vphi^{(1)}, \Phi^{(1)})-(\vphi^{(2)}, \Phi^{(2)})\|_{H^1(\Om_L)}+\|{\bf W}^{(1)}-{\bf W}^{(2)}\|_{H^2(\Om_L)}.
\end{equation*}
Then, by adjusting the proof of Proposition \ref{proposition-H1-new}, we can show that
\begin{equation}
\label{contraction of potentials-HD}
  \|(\vphi^{(1)}, \Phi^{(1)})-(\vphi^{(2)}, \Phi^{(2)})\|_{H^1(\Om_L)}\le C_{\flat}\left(\eps+\sigma(b,u_{\rm en},v_{\rm en},0,S_0 ,E_{\rm en},E_{\rm ex})\right)\eta
\end{equation}
for some constant $C_{\flat}>0$ fixed depending only on the data. And, by using \eqref{definition-F-vec-field}, one can easily check that
\begin{equation*}
  \|{\bf W}^{(1)}-{\bf W}^{(2)}\|_{H^2(\Om_L)}\le C_{\sharp}\left(\eps+\sigma(b,u_{\rm en},v_{\rm en},0,S_0 ,E_{\rm en},E_{\rm ex})\right)\eta
\end{equation*}
for some constant $C_{\sharp}>0$ fixed depending only on the data. Finally, we add the previous two estimates to get
\begin{equation*}
  \eta\le (C_{\flat}+C_{\sharp})\left(\eps+\sigma(b,u_{\rm en},v_{\rm en},0,S_0 ,E_{\rm en},E_{\rm ex})\right)\eta.
\end{equation*}
Therefore, we can fix a constant $\bar{\eps}>0$ sufficiently small so that if the inequality
\begin{equation*}
  \eps+\sigma(b,u_{\rm en},v_{\rm en},0,S_0 ,E_{\rm en},E_{\rm ex})\le \bar{\eps},
\end{equation*}
then all the arguments up to Step 3 hold, and $\eta=0$ holds. Moreover, such a constant $\bar{\eps}$ can be fixed depending only on the data and $(L,\bar{\epsilon})$. This completes the proof.

\end{proof}

\subsubsection{Step 2: Initial value problem for $(S, \Lambda)$}
Assume that the condition \eqref{condition-for-eps-sig} holds. Then, for each $(S_*,\Lambda_*)\in \mcl{J}_{\eps}$, Lemma \ref{pro-1} yields a unique solution $(\vphi, \Phi, {\bf W})$ (with ${\bf W}=\phi {\bf e}_{\theta}$) to Problem \ref{nlbvp:potentials} associated with $(S_*,\Lambda_*)$ so that the solution satisfies the estimate \eqref{pphi-est}. We define a vector field ${\bf m}_*:\ol{\Om_L}\rightarrow \R^3$ by
\begin{equation}
\label{definition of m field}
  {\bf m}_*={\rho}(\vphi, \Phi, \phi{\bf e}_{\theta},S_*,\Lambda_*)\left(\nabla\vphi+ \nabla \times (\phi{\bf e}_{\theta})+\frac{\Lambda_*}{|\rx'|^2}(\rx')^{\perp}\right)
\end{equation}
for ${\rho}(\vphi, \Phi, \phi{\bf e}_{\theta},S_*,\Lambda_*)$ given by \eqref{definition:density-HD}.
Note that ${\bf m}_*$ is axisymmetric in the sense of Definition \ref{definition-axixymmetry}.
By Definition \ref{definition-nl-operators}, the equation $\mcl{N}_1^{(S_*, \Lambda_*, {\bf W}_*)}(\vphi,\Phi)=0$ in $\Om_L$ immediately yields that
\begin{equation}
\label{div-free of m field}
  \nabla\cdot {\bf m}_*=0\quad\tx{in $\Om_L$}.
\end{equation}

\begin{problem}
\label{problem-ivp for S and Lambda}
Given radial functions $w_{\rm en},\,S_{\rm en}\in C^4(\Gam_0)$ that satisfy the compatibility conditions:
\begin{itemize}
\item[-] $\displaystyle{w_{\rm en}=0\quad\tx{on $ \Gam_0^{\bar{\epsilon}} $},\quad \der_r^k w_{\rm en}({\bf 0})=0\quad\tx{for $k=0,1,2,3$}},$
    \smallskip
   \item[-] $\displaystyle{S_{\rm en}=S_0\quad\tx{on $ \Gam_0^{\bar{\epsilon}} $},\quad \der_r^kS_{\rm en}({\bf 0})=0\quad\tx{for $k=1,2,3$}}$,
\end{itemize}
find an axisymmetric solution $(S, \Lambda)$ to the linear initial value problem:
\begin{equation}\label{INI-S}
\begin{cases}
{\bf m}_{\ast}\cdot \nabla S=0\\
{\bf m}_{\ast}\cdot \nabla \Lambda=0
\end{cases}\,\,\mbox{in}\,\,\Omega_L,\quad
\begin{cases}
S(0, \rx')=S_{\rm en}(\rx')\\
\Lambda(0, \rx')=|\rx'|w_{\rm en}(\rx')
\end{cases}\,\,\mbox{on}\,\,\Gamma_0.
\end{equation}

\end{problem}

\begin{proposition}\label{pro-tt}
For the constant $\bar{\eps}>0$ from Lemma \ref{pro-1}, one can fix a constant $\eps^*\in(0, \bar{\eps}]$ depending only on the data and $(L,\bar{\epsilon})$ so that if
\begin{equation}
\label{condition for transp eqn}
\eps+\sigma(b,u_{\rm en},v_{\rm en},0,S_{0},E_{\rm en},E_{\rm ex})\le\eps^{\ast},
\end{equation}
then Problem \ref{problem-ivp for S and Lambda} acquires a unique solution $(S,\Lambda)$ that satisfies the following properties:
\begin{itemize}
\item[(a)] there exists a constant $C_{\natural}>0$ depending only on the data and $L$ to satisfy
    \begin{equation}
    \label{tran-est}
    \begin{split}
    &\|S-S_0\|_{\mcl{M}^4_*((0,L)\times \mcl{D})}+\left\|\frac{\der_rS}{|\rx'|}\right\|_{\mcl{M}^3_*((0,L)\times \mcl{D})}\le C_{\natural}\|S_{\rm en}-S_0\|_{C^4(\ol{\Gam_0})},\\
    &\|\Lambda\|_{\mcl{M}^4_*((0,L)\times \mcl{D})}+\sum_{k=0,1}\left\|\frac{\der_r^k\Lambda}{|\rx'|^{2-k}}\right\|_{\mcl{M}^3_*((0,L)\times \mcl{D})}\le C_{\natural}\|w_{\rm en}\|_{C^4(\ol{\Gam_0})};
    \end{split}
    \end{equation}
 \item[(b)] $\displaystyle{S=S_0\,\,\tx{and}\,\, \Lambda=0\,\,\tx{hold on}\,\, \left\{(x_1, \rx')\in \ol{\Om_L}:|\rx'|\ge1-\frac{\bar{\epsilon}}{2}\right\}}$;
 \item[(c)] $\displaystyle{\der_r S=0\,\,\tx{and}\,\, \der_r\Lambda=0\,\,\tx{hold on }\,\,\Gamw}$;
 \item[(d)] $S$ and $\Lambda$ are axisymmetric in $\Om_L$.

\end{itemize}

\begin{proof} {\textbf{1.}} In this proof, we shall discuss in details on the initial value problem for $S$ only because one can repeat the same argument for $\Lambda$ with minor adjustments.

Since we seek for an axisymmetric solution to
\begin{equation}
\label{ivp-S}
  \begin{cases}
  {\bf m}_{\ast}\cdot \nabla S=0\quad&\mbox{in $\Om_L$}\\
  S(0, \rx')=S_{\rm en}(\rx')\quad&\mbox{on $\Gam_0$}
  \end{cases},
\end{equation}
we set $\mcl{S}(x_1, r):=S(x_1, \rx')$ and $\mcl{S}_{\rm en}(r):=S_{\rm en}(\rx')$ for $r=|\rx'|$, and it suffices to solve the following initial value problem for $\mcl{S}$:
\begin{equation}
\label{ivp-S-axisym}
  \begin{cases}
  (r{\bf m}_{\ast}\cdot {\bf e}_{x_1})\der_{x_1}\mcl{S} +(r{\bf m}_{\ast}\cdot {\bf e}_{r})\der_r\mcl{S}=0\quad&\mbox{in $R_L:=\left\{(x_1, r)\in \R^2\left|\, \begin{split}&0<x_1<L,\\
  			&0<r<1\end{split}\right.\right\}$},\\
  \mcl{S}(0, r)=\mcl{S}_{\rm en}(r)\quad&\mbox{on $\Sigma_0:=\der R_L\cap\{x_1=0\}$}.
  \end{cases}
\end{equation}

An explicit computation with using \eqref{div-free of m field} and the axisymmetry of ${\bf m}_*$ shows that
\begin{equation}
\label{div-free-in cyl-coord}
  \der_{1}(r{\bf m}_{\ast}\cdot {\bf e}_{x_1})+\der_r(r{\bf m}_{\ast}\cdot {\bf e}_{r})=0\quad\tx{for all $(x_1, r)\in R_L$}.
\end{equation}
Define a function $w:[0,L]\times[0,1]\to\mathbb{R}$ by
\begin{equation*}
w(x_1,r):=\int_0^r \tau{\bf m}_{\ast}\cdot{\bf e}_{x_1}(x_1,\tau)d\tau.
\end{equation*}
Then it easily follows from \eqref{div-free-in cyl-coord} that
\begin{equation}
\label{property of w}
(\partial_{x_1} w, \der_rw)(x_1,r)=r(-{\bf m}_{\ast}(x_1, \rx')\cdot{\bf e}_r,{\bf m}_{\ast}(x_1, \rx')\cdot{\bf e}_{x_1})\quad\mbox{for}\quad(x_1,r)\in R_L.
\end{equation}
By using \eqref{representation of W}, the boundary condition $\phi=0$ on $\der R_L\cap\{r=0,1\}$ given in \eqref{lbvp for phi}, \eqref{curl of W}, \eqref{definition of m field}, and the boundary condition $\der_r\vphi=0$ on $\Gamw$ given in \eqref{bc for potentials-iters}, it can be directly checked that
\begin{equation}\label{px-w}
\partial_{x_1} w=0\quad\mbox{on $\der R_L\cap\{r=0,1\}$}.
\end{equation}

Let us set
\begin{equation*}
\overline{\bf m}:=\bar{\rho}\bar{u}{\bf e}_{x_1}=J_0{\bf e}_{x_1}.
\end{equation*}
By using the estimate \eqref{pphi-est}, the definition of ${\bf m}_*$ stated in \eqref{definition of m field} and \eqref{choice of M}, we can directly show that
\begin{equation}
\label{estimate of m field}
  \|{\bf m}_*-\overline{\bf m}\|_{\mcl{M}^3_*((0, L)\times \mcl{D};\R^3)}\le C\left(\eps+\sigma(b,u_{\rm en},v_{\rm en},0,S_{0},E_{\rm en},E_{\rm ex})\right)
\end{equation}
for some constant $C>0$ depending only on the data and $(L, \bar{\epsilon})$. Therefore, one can choose a small constant $\eps^{\ast}\in(0,\bar{\eps}]$ depending only on the data and $(L, \bar{\epsilon})$ so that if
\begin{equation*}
\eps+\sigma(b,u_{\rm en},v_{\rm en},0,S_{0},E_{\rm en},E_{\rm ex})\le\eps^{\ast},
\end{equation*}
then
\begin{equation}\label{pr-w}
\frac{2}{5}J_0\le \min_{\ol{\Om_L}}{\bf m}_{\ast}\cdot{\bf e}_{x_1}\le \max_{\ol{\Om_L}}{\bf m}_{\ast}\cdot{\bf e}_{x_1} \le  \frac{8}{5}J_0.
\end{equation}
\medskip

For the rest of the proof, the constant $C$ that appears in various estimates may be fixed differently but it is regarded to be fixed depending only on the data and $(L, \bar{\epsilon})$ unless otherwise specified.
\medskip

Let us define a function $\mcl{G}:[0,1]\rightarrow \R$ by
\begin{equation*}
\mathcal{G}(r):=w(0,r).
\end{equation*}
The properties \eqref{px-w} and \eqref{pr-w} imply that
\begin{itemize}
\item[-] $\displaystyle{w(x_1,0)=w(0,0),\,\, w(x_1, 1)=w(0,1)}$ for $0\le x_1\le L$;
\item[-] $\displaystyle{\der_r w(x_1, r)>0}$ for $(x_1,r)\in [0,L]\times (0,1]$.
\end{itemize}
Therefore, $\mcl{G}:[0,1]\rightarrow [w(0,0), w(0,1)]$ is invertible and $\mcl{G}^{-1}$ is differentiable. Let us define a function $\mcl{T}:\ol{R}_L\rightarrow [0,1](=\der R_L\cap\{x_1=0\})$ by
\begin{equation}
\label{definition of T map}
\mcl{T}:=\mcl{G}^{-1}\circ w.
\end{equation}
Let us set $\mcl{S}: \ol{R_L}\rightarrow \R $ as
\begin{equation}
\label{sol-S}
  \mcl{S}(x_1, r):=\mcl{S}_{\rm en}\circ \mcl{T}(x_1, r).
\end{equation}
Then it can be directly checked by using \eqref{property of w} that the function $\mcl{S}$ solves \eqref{ivp-S-axisym}. This proves that the initial value problem \eqref{ivp-S} has at least one solution
\begin{equation}
\label{sol-S-final}
S(x_1, \rx')=\mcl{S}(x_1, |\rx'|).
\end{equation}

\medskip

{\textbf{2.}}{\textbf{(2-1)}} We differentiate $\mcl{G}\circ \mcl{T}(x_1, r)=w(x_1, r)$ and use \eqref{property of w} to get
\begin{equation}
\label{derivative of T map}
  D_{(x_1, r)}\mcl{T}(x_1, r)=\frac{r}{\mcl{T}(x_1, r)}\cdot
  \frac{(-{\bf e}_r\cdot{\bf m}_*, {\bf e}_{x_1}\cdot {\bf m}_*)(x_1, r)}{{\bf e}_{x_1}\cdot {\bf m}_*(0, \mcl{T}(x_1, r))}\quad \tx{in $R_L$}.
\end{equation}

From the definition of $\mcl{T}$ and the property $w(x_1,0)=w(0,0)$ for $x_1\in[0,L]$, it directly follows that
\begin{equation*}
  w(x_1, r)-w(x_1, 0)=w(0, \mcl{T}(x_1, r))-w(0,0)\quad\tx{for all $(x_1,r)\in\ol{R}_L$}.
\end{equation*}
We differentiate the above equation with respect to $r$, and use \eqref{property of w} to get
\begin{equation*}
  \frac{r}{\mcl{T}(x_1,r)}=\sqrt{\frac{\int_0^1 (t{\bf e}_1\cdot {\bf m}_*)(0,t\mcl{T}(x_1,r) )\,dt}{\int_0^1 (t{\bf e}_1\cdot {\bf m}_*)(x_1, tr )\, dt}},
\end{equation*}
which combined with \eqref{pr-w} yields that
\begin{equation}
\label{r vs T}
 \frac{1}{2}\le \frac{r}{\mcl{T}(x_1, r)}\le 2\quad\tx{in $R_L$.}
\end{equation}
By using \eqref{estimate of m field}, \eqref{pr-w} and \eqref{r vs T}, we easily obtain from \eqref{derivative of T map} that
\begin{equation}
\label{C1 estimate of T map}
  \|D_{\rx}\mcl{T}\|_{C^0(\ol{\Om_L})}\le C.
\end{equation}

\smallskip

{\textbf{(2-2)}} Due to the conditions $(S_*, \Lambda_*)(0, \rx')=(S_{\rm en}(\rx'), |\rx'|w_{\rm en}(\rx'))$ on $\Gam_0$ stated in \eqref{iterset-T}, and the boundary conditions for $\vphi$ prescribed on $\Gam_0$ (see \eqref{bc for potentials}), we have
\begin{equation}\label{mC3}
{\bf e}_{x_1}\cdot{\bf m}_{\ast}=\left[\frac{\gamma-1}{\gamma S_{\rm en}}\left(\Phi-\frac{1}{2}(u_{\rm en}^2+v_{\rm en}^2+w_{\rm en}^2)\right)\right]^{1/(\gamma-1)}u_{\rm en}\quad\tx{on $\Gam_0$}.
\end{equation}

Let us rewrite the equation $\mcl{N}_2^{(S_*, \Lambda_*, {\bf W})}(\vphi, \Phi)=0$ given in Problem \ref{nlbvp:potentials} as
\begin{equation*}
  \Delta \Phi=\rho(\vphi, \Phi, {\bf W}, S_*, \Lambda_*)-b=:f_2\quad\tx{in $\Om_L$}.
\end{equation*}
It follows from a direct computation with using Lemma \ref{pro-1} that $f_2\in H^3(\Om_L\cap\{x_1< \frac L2\})$ and $\der_{x_1}f_2=0$ on $\Gam_0^{\bar{\epsilon}}$. Therefore, we can apply the method of (even) reflection about with respect to $x_1$-variable in a small neighborhood of $\ol{\Gam_0}\cap \ol{\Gamw}$ with using standard regularity theorems for a linear elliptic boundary value problem (e.g. see \cite[Theorems 8.10 and 8.13]{GilbargTrudinger} ) to conclude that there exists a constant $C>0$ such that
\begin{equation}\label{Psi-C3}
\|\Phi\|_{C^3(\overline{\Omega_L\cap\{x\le \frac{L}{4}\}})}\le C,
\end{equation}
and this finally yields that
\begin{equation}
\label{estimate-m field-ent}
  \|{\bf e}_{x_1}\cdot{\bf m}_{\ast}\|_{C^3(\ol{\Gam_0})}\le C.
\end{equation}
By a lengthy but straightforward computation with using \eqref{derivative of T map}, \eqref{r vs T} and \eqref{estimate-m field-ent}, it can be directly checked that, for $k=1,2,3,4$,
\begin{equation}
\label{final estimate of T map}
  |D^k_{\rx}\mcl{T}(x_1, |\rx'|)|\le \frac{C}{|\rx'|^{k-1}}\sum_{j=0}^k |D^j_{\rx}({\bf e}_1\cdot {\bf m}_*)(x_1, \rx')|\quad\tx{for $(x_1, \rx')\in\ol{\Om_L}$}.
\end{equation}

{\textbf{3.}} Now, we shall briefly discuss how to establish the first estimate in \eqref{tran-est}. First of all, one can easily check from \eqref{sol-S-final} and \eqref{C1 estimate of T map} that
\begin{equation*}
  \|S-S_0\|_{C^1(\ol{\Om_L})}\le C\|S_{\rm en}-S_0\|_{C^1(\ol{\Gam_0})}.
\end{equation*}
Next, we observe that
\begin{equation*}
  \begin{split}
  |D^2_{\rx}S(x_1, \rx')|&=|D^2_{\rx}(\mcl{S}_{\rm{en}}\circ  \mcl{T})(x_1, |\rx'|)|\\
  &\le |D_{\rx}\mcl{T}(x_1, |\rx'|)|^2|\mcl{S}_{\rm{en}}^{''}\circ \mcl{T}(x_1, |\rx'|)|+|D^2_{\rx}\mcl{T}(x_1, |\rx'|)||\mcl{S}_{\rm en}^{'}\circ \mcl{T}(x_1, |\rx'|)|\\
  &\le \|S_{\rm en}-S_0\|_{C^2(\ol{\Gam_0})}
  \left(|D_{\rx}\mcl{T}(x_1, |\rx'|)|^2+\mcl{T}(x_1, \rx')|D^2_{\rx}\mcl{T}(x_1, |\rx'|)|\right)\\
  &\le C\|S_{\rm en}-S_0\|_{C^2(\ol{\Gam_0})}\left(1+\underset{(\le 2\,\,\tx{by \eqref{r vs T}})}{\underbrace{\frac{\mcl{T}(x_1, |\rx'|)}{|\rx'|}}}\sum_{j=0,1} |D^j_{\rx}({\bf e}_1\cdot {\bf m}_*)(x_1, \rx')|\right),
  \end{split}
\end{equation*}
where we use the compatibility condition $\mcl{S}'_{\rm en}(0)=\der_r S_{\rm en}({\bf 0})=0$ stated in Problem \ref{EP-Prob} to obtain the last inequality.
This estimate naturally yields
\begin{equation*}
  \|S-S_0\|_{H^2(\Om_L)}
  +\|S-S_0\|_{\mcl{W}^{2,\infty}_{\mcl{D}}(0,L)}\le
  C\|S_{\rm en}-S_0\|_{C^2(\ol{\Gam_0})}.
\end{equation*}
The rest of the part for the estimate $S-S_0$ stated in \eqref{tran-est} can be checked similarly with more tedious computations.
\medskip

{\textbf{4.}} Suppose that $S^{(1)}$ and $S^{(2)}$ are two (not necessarily axisymmetric) solutions to
\begin{equation}
\label{bvp for S}
\begin{cases}
  {\bf m}_*\cdot \nabla S=0\quad\mbox{in $\Om_L$}\\
   S=S_{\rm en}\quad\mbox{on $\Gam_0$}
  \end{cases}.
\end{equation}
And, suppose that both solutions satisfy the first estimate stated in \eqref{tran-est}. Let us set
\begin{equation*}
  Z(\rx):=(S^{(1)}-S^{(2)})(\rx)\quad\tx{in $\Om_L$}.
\end{equation*}
Owing to \eqref{pr-w}, the vector field $\til{\bf m}_*$ given by
\begin{equation*}
  \til{\bf m}_*:=\frac{-{\bf m}_*}{{\bf m}_*\cdot {\bf e}_{x_1}}\quad\tx{in $\ol{\Om_L}$}
\end{equation*}
is well defined. Furthermore, it follows from Remark \ref{remark-embedding} that $Z\in C^1(\ol{\Om_L})$ and $\til{\bf m}_*\in C^0(\ol{\Om_L})\cap C^1(\ol{\Om_L\cap\{x_1<L-d\}})$ for any $d\in(0,L)$.
Also, $Z$ clearly solves the problem
\begin{equation*}
\begin{cases}
  \til{\bf m}_*\cdot \nabla Z=0\quad\mbox{in $\Om_L$}\\
   Z=0\quad\mbox{on $\Gam_0$}
  \end{cases}.
\end{equation*}

For a fixed point $\rx_0=(a, {\rx}'_a)\in \ol{\Om_L}\cap \{x_1\le L-d\}$ for some $d\in(0,L)$, let us consider the initial value problem for ${\bf X}:[0,a]\rightarrow \ol{\Om_L}$:
\begin{equation*}
  \begin{cases}
  {\bf X}'(t)=\til{\bf m}_*({\bf X}(t))\quad\tx{for $0<t\le a$}\\
  {\bf X}(0)=\rx_0.
  \end{cases}
\end{equation*}
Since $\til{\bf m}_*\in C^{1}(\ol{\Om_L\cap\{x_1<L-d\}})$, it easily follows from the unique existence theorem of ODEs that there exists a unique $C^1$-solution ${\bf X}$. Furthermore, we have
${\bf X}(a)\in\Gam_0$ because $\til{\bf m}\cdot{\bf e}_{x_1}\equiv -1$ for all $t\in[0,a]$. Then $Z$ satisfies
\begin{equation*}
  \frac{d}{dt}Z({\bf X})(t)={\bf X}'\cdot \nabla Z=\til{\bf m}_*\cdot \nabla Z=0\quad\tx{for all $t\in[0,a]$},
\end{equation*}
thus we obtain that $Z(\rx_0)=Z({\bf X}(a))=0$ due to the boundary condition $Z=0$ on $\Gam_0$. This combined with the continuity of $Z$ up to $\Gam_L$ implies that $Z\equiv 0$ in $\ol{\Om_L}$. This proves the uniqueness the solution to \eqref{bvp for S}.

\medskip
{\textbf{5.}} Finally, we briefly mention that the solution $\Lambda$ to the initial value problem
\begin{equation*}
  \begin{cases}
  {\bf m}_{\ast}\cdot \nabla \Lambda=0\quad&\mbox{in $\Om_L$},\\
  \Lambda(0, \rx')=|\rx'|w_{\rm en}(\rx')\quad&\mbox{on $\Gam_0$}
  \end{cases}
\end{equation*}
is given by
\begin{equation}
\label{representation of Lambda}
  \Lambda(x_1, \rx')=\mcl{T}(x_1, \rx')[w_{\rm en}\circ \mcl{T}](x_1, \rx')\quad\tx{in $\Om_L$}.
\end{equation}
The proof of Proposition \ref{pro-tt} is completed.
\end{proof}
\end{proposition}

\subsubsection{Step 3: The finalization of the proof}
\label{subsubsection: finalization of the proof}
{\textbf{1.}}
Under the condition of \eqref{condition for transp eqn}, let us define an iteration map $\mcl{I}_T:\mcl{J}_{\eps}\rightarrow [H^4_*(\Om_L)\cap \mcl{W}^{4,\infty}_{*, \mcl{D}}(0,L)]^2$ by
\begin{equation*}
  \mcl{I}_T(S_*, \Lambda_*):=(S, \Lambda)
\end{equation*}
for the solution $(S, \Lambda)$ to the initial value problem \eqref{INI-S} associated with $(S_*, \Lambda_*)\in \mcl{J}_{\eps}$. Owing to Proposition \ref{pro-tt}, the map $\mcl{I}_T$ is well defined. If we fix the constant $\eps$ as
\begin{equation}
\label{choice of eps}
  \eps=2C_{\natural}(\|S_{\rm en}-S_0\|_{C^4(\ol{\Gam_0})}+\|w_{\rm en}\|_{C^4(\ol{\Gam_0})})
\end{equation}
for the constant $C_{\natural}>0$ from \eqref{tran-est}, then we have
\begin{equation}
\label{estimate for S and Lambda}
\begin{split}
  &\|S-S_0\|_{\mcl{M}^4_*((0,L)\times \mcl{D})}+\left\|\frac{\der_rS}{|\rx'|}\right\|_{\mcl{M}^3_*((0,L)\times \mcl{D})}\le \frac{\eps}{2},\\
  &\|\Lambda\|_{\mcl{M}^4_*((0,L)\times \mcl{D})}+\sum_{k=0,1}\left\|\frac{\der_r^k\Lambda}{|\rx'|^{2-k}}\right\|_{\mcl{M}^3_*((0,L)\times \mcl{D})}\le \frac{\eps}{2}.
\end{split}
\end{equation}
Given a point $\rx=(x_1, \rx')\in \Om_L$, if $|\rx'|\le \frac{\bar{\epsilon}_0}{2}$, then \eqref{r vs T} implies that $\mcl{T}(x_1, |\rx'|)\le \bar{\epsilon}_0$. Note that $S_{\rm{en}}=S_0$ and $w_{\rm{en}}=0$ on $\Gam_0^{\bar{\epsilon}}$ according to the compatibility conditions stated in Problem \ref{EP-Prob}. Therefore we derive from \eqref{sol-S} and \eqref{representation of Lambda} that
\begin{itemize}
\item[-] $\displaystyle{\der_{x_1} S=0,\,\,\tx{and}\,\,\der_{x_1}\Lambda=0\,\,\tx{on $\Gam_0^{\bar{\epsilon}_0}$}}$;
\item[-] $\displaystyle{S(\rx)=S_0\,\,\tx{and}\,\,\Lambda(\rx)=0\,\,\tx{in $\{\rx=(x_1, \rx')\in \ol{\Om_L}: |\rx'|\ge 1-\frac{\bar{\epsilon}_0}{2}\}$}.}$
\end{itemize}
So we conclude that if $\sigma(b,u_{\rm en},v_{\rm en},w_{\rm en},S_{\rm en},E_{\rm en},E_{\rm ex})$ (see Theorem \ref{MainThm} for the definition of $\sigma$) satisfies
\begin{equation}
\label{sigma condition 1}
  (2C_{\natural}+1)\sigma(b,u_{\rm en},v_{\rm en},w_{\rm en},S_{\rm en},E_{\rm en},E_{\rm ex})\le \eps^*
\end{equation}
(so that  the condition \eqref{condition for transp eqn} holds), then the iteration map $\mcl{I}_T$ maps $\mcl{J}_{\eps}$ into itself.
\medskip

{\textbf{2.}} Clearly, we get to prove the existence of a solution to Problem \ref{EP-Prob-HD} if we show that the iteration map $\mcl{I}_T$ has a fixed point in $\mcl{J}_{\eps}$. In order to prove the existence of a fixed point of $\mcl{I}_{T}$, we apply the Schauder fixed point theorem. As in the proof of Theorem \ref{Irr-MainThm} in \S\ref{subsection-proof of theorem-potential}, the iteration set $\mcl{J}_{\eps}$ is a compact subset of the Banach space $\mcl{B}$ given by \eqref{definition of B set}. Given a sequence $\{(S_k, \Lambda_k)\}_{k\in \mathbb{N}}\subset \mcl{J}_{\eps}$, that converges to $(S_{\infty}, \Lambda_{\infty})\in \mcl{J}_{\eps}$ in $[H^2(\Om_L)\cap C^{1,1/4}(\ol{\Om_L})]^2$, let ${\bf m}_*^{(k)}$ and ${\bf m}_*^{(\infty)}$ be given by \eqref{definition of m field} associated with $(S_k, \Lambda_k)$ and $(S_{\infty}, \Lambda_{\infty})$, respectively. For each $j\in \mathbb{N}\cup\{\infty\}$, let $(S^{(j)}, \Lambda^{(j)})$ be the solution to Problem \ref{problem-ivp for S and Lambda} for ${\bf m}_*={\bf m}_*^{(j)}$. Then the following properties can be easily checked:
\begin{itemize}
\item[(i)] The sequence $\{{\bf m}_*^{(k)}\}_{k\in \mathbb{N}}$ is bounded in the norm $\|\cdot\|_{\mcl{M}^3_*((0,L)\times \mcl{D})}$, thus it has a subsequence $\{{\bf m}_*^{(k_l)}\}_{l\in\mathbb{N}}$ that converges to ${\bf m}^{(\infty)}$ in $C^0(\ol{\Om_L})$;
\item[(ii)] By using \eqref{estimate for S and Lambda} and the compactness of the set $\mcl{J}_{\eps}$ in the Banach space $\mcl{B}$, we can conclude from the statement (i) that the sequence $\displaystyle{\{(S^{(k_l)}, \Lambda^{(k_l)})\}(=\{\mcl{I}_T(S_{k_l}, \Lambda_{k_l})\})}$ has a subsequence that converges to $\displaystyle{(S^{(\infty)}, \Lambda^{(\infty)})(=\mcl{I}_T(S_{\infty}, \Lambda_{\infty}))}$ in $\mcl{B}$;
\item[(iii)] Due to the compactness of $\mcl{J}_{\eps}$ in $\mcl{B}$, any subsequence of $\{(S^{(k)}, \Lambda^{(k)})\}_{k\in \mathbb{N}}$ has its subsequence that converges in $\mcl{B}$. Furthermore, the uniqueness of a solution to Problem \ref{problem-ivp for S and Lambda} implies that the limit of the convergent subsequence is $(S^{(\infty)}, \Lambda^{(\infty)})$.
\end{itemize}
Therefore, we conclude that $\mcl{I}_T:\mcl{J}_{\eps}\rightarrow \mcl{J}_{\eps}$ is continuous in $\mcl{B}$. Finally, it follows from the Schauder fixed point theorem that the map $\mcl{I}_T$ has a fixed point in $\mcl{J}_{\eps}$, and this combined with \eqref{pphi-est}, \eqref{representation of W}, \eqref{tran-est} and \eqref{choice of eps} implies that Problem \ref{EP-Prob-HD} has at least one solution $(\vphi, \Phi, \phi, S, \Lambda)$ that satisfies the estimates \eqref{Thm-HD-est} and \eqref{Thm-HD-est2}. Furthermore, one can fix a constant $\sigma_2>0$ sufficiently small depending only on the data and $(L,\bar{\epsilon})$ so that if the inequality \eqref{condition for sigma} holds, then Lemma \ref{Lem1} combined with \eqref{pphi-est} and Morrey's inequality yields the estimate \eqref{HD-sol-lower bd} for the constant $\om_0$ given by
\begin{equation*}
  \om_0:=\frac 12 \min\left\{\bar{\delta},\hat{\delta},\frac{J_0}{\rhos-\bar{\delta}}\right\}
\end{equation*}
for the constants $\bar{\delta}$ and $\hat{\delta}$ from Lemma \ref{Lem1}.
This finally proves that Problem \ref{EP-Prob-HD} has at least one solution that satisfies all the properties stated in the statements (a)--(c) of Proposition \ref{Thm-HD} provided that $\sigma(b,u_{\rm en},v_{\rm en},w_{\rm en},S_{\rm en},E_{\rm en},E_{\rm ex})$ is fixed sufficiently small.
\medskip

{\textbf{3.}} To complete the proof of Proposition \ref{Thm-HD}, it remains to prove the uniqueness of a solution. Overall procedure is similar to \cite[Proof of Theorem 1.7, Step 3]{bae2021structural} with minor differences. Suppose that Problem \ref{EP-Prob-HD} has two solutions $(S^{(1)},\Lambda^{(1)},\varphi^{(1)},\Phi^{(1)},\phi^{(1)})$ and $(S^{(2)},\Lambda^{(2)},\varphi^{(2)},\Phi^{(2)},\phi^{(2)})$ that satisfy the estimate \eqref{pphi-est}. And, let us set
\begin{equation*}
\begin{split}
&(\breve{S}, \breve{\Lambda}):=(S^{(1)}-S^{(2)}, \Lambda^{(1)}-\Lambda^{(2)}),\\
&(\breve{\vphi},\breve{\Phi},\breve{\bf W}):=(\varphi^{(1)}-\varphi^{(2)},\Phi^{(1)}-\Phi^{(2)},{ \phi}^{(1)}{\bf e}_{\theta}-{\phi}^{(2)}{\bf e}_{\theta}),\\
&E_1:=\|\breve{S}\|_{H^1(\Omega_L)}+\left\|\frac{\breve{\Lambda}}{|\rx'|^2}\right\|_{H^1(\Omega_L)}
+\left\|\frac{\der_r \breve{\Lambda}}{|\rx'|}\right\|_{L^2(\Om_L)},
\\
&E_2:=\|\breve{\phi}\|_{H^1(\Omega_L)}+\|\breve{\Phi}\|_{H^1(\Omega_L)}+\|\breve{\bf W}\|_{H^2(\Omega_L)}.
\end{split}
\end{equation*}
{\textbf{(3-1)}}
For each $k=1,2$, let a map $\mcl{T}^{(k)}$ be given by \eqref{definition of T map} associated with $(S^{(k)},\Lambda^{(k)},\varphi^{(k)},\Phi^{(k)},\phi^{(k)})$. Then we use \eqref{sol-S} and \eqref{representation of Lambda} to get
\begin{equation*}
\begin{split}
  &(S^{(1)}-S^{(2)})(x_1, \rx')=(S_{\rm en}\circ \mcl{T}^{(1)}-S_{\rm en}\circ \mcl{T}^{(2)})(x_1, |\rx'|),\\
  &(\Lambda^{(1)}-\Lambda^{(2)})(x_1, \rx')=(\mcl{T}^{(1)}[w_{\rm en}\circ \mcl{T}^{(1)}]-\mcl{T}^{(2)}[w_{\rm en}\circ \mcl{T}^{(2)}])(x_1, |\rx'|).
  \end{split}
\end{equation*}
By lengthy but straightforward computations with using \eqref{property of w}, \eqref{derivative of T map}, \eqref{r vs T}, and the compatibility conditions (see Problem \ref{EP-Prob}):
\begin{equation*}
  \begin{split}
  &\der_r^k w_{\rm en}({\bf 0})=0\quad\mbox{for $k=0,1,2,3$},\quad
  \der_r^{\til{k}} S_{\rm en}({\bf 0})=0\quad\mbox{for $\til{k}=1,2,3$},
  \end{split}
\end{equation*}
it can be shown that there exists a small constant $\bar{\sigma}_T>0$ and a constant $C_{T}>0$ so that if $\sigma:=\sigma(b,u_{\rm en},v_{\rm en},w_{\rm en},S_{\rm en},E_{\rm en},E_{\rm ex})$ satisfies the inequality
\begin{equation}
\label{sigma condition 2}
  \sigma\le \bar{\sigma}_T,
\end{equation}
then it holds that
\begin{equation}
\label{estimate for E1}
  E_1\le C_T\sigma E_2.
\end{equation}
\smallskip

{\textbf{(3-2)}} For each $k=1,2$, $(\vphi^{(k)}, \Phi^{(k)})$ solves
\begin{equation*}
\begin{split}
&\begin{cases}
\mcl{N}_1^{(S^{(k)}, \Lambda^{(k)}, {\phi^{(k)}{\bf e}_{\theta}})}(\vphi^{(k)},\Phi^{(k)})=0\\
\mcl{N}_2^{(S^{(k)}, \Lambda^{(k)}, {\phi^{(k)}{\bf e}_{\theta}})}(\vphi^{(k)},\Phi^{(k)})=0
\end{cases}\mbox{in}\,\,\Omega_L,\\
&\left\{\begin{split}
	\varphi^{(k)}(0,r)=-\int_r^1 v_{\rm en}(t)dt,\,\,\partial_{x_1}\varphi^{(k)}=u_{\rm en}-(\nabla\times (\phi^{(k)}{\bf e}_{\theta}))\cdot{\bf e}_{x_1}\,\,&\mbox{on}\,\,\Gamma_0,\\
		\partial_{x_1}\Phi^{(k)}=E_{\rm en}\,\,&\mbox{on}\,\,\Gamma_0,\\
\partial_{r}\varphi^{(k)}=0,\,\,\partial_{r}\Phi^{(k)}=0\,\,&\mbox{on}\,\,\Gamma_{{\rm w}},\\
\der_{x_1}\Phi^{(k)}=E_{\rm{ex}}\,\,&\mbox{on}\,\,\Gamma_L
\end{split}\right.
\end{split}
\end{equation*}
for the nonlinear differential operators $\mcl{N}_1$ and $\mcl{N}_2$ given by Definition \ref{definition-nl-operators}. By subtracting the problem for $(\vphi^{(2)}, \Phi^{(2)})$ from the one for $(\vphi^{(1)}, \Phi^{(1)})$, we can derive a linear boundary value problem for $(\breve{\vphi},\breve{\Phi})$ similar to \eqref{lbvp associated with P} with homogeneous boundary conditions. So we repeat the argument given in the proof of Proposition \ref{proposition-H1-new} to show that
\begin{equation}
\label{intermediate contraction}
  \|(\breve{\vphi},\breve{\Phi})\|_{H^1(\Om_L)}\le C_{P}
  \left(\|\breve{\bf W}\|_{H^2(\Om_L)}+E_1\right)
\end{equation}
for some constant $C_P>0$.
\smallskip

{\textbf{(3-3)}} For each $k=1,2$, ${\bf W}^{(k)}:=\phi^{(k)}{\bf e}_{\theta}$ solves
\begin{equation*}
\begin{cases}
-\Delta{\bf W}^{(k)}={\bf F}^{(S^{(k)}, \Lambda^{(k)},{\bf W}^{(k)})}({\vphi}^{(k)},{\Phi}^{(k)})(=:{\bf G}^{(k)})\quad&\mbox{in $\Om_L$},\\
\partial_{x_1}{\bf W}^{(k)}={\bf 0}\quad&\mbox{on $\Gamma_0$},\\
{\bf W}^{(k)}={\bf 0}\quad&\mbox{on $\Gamma_{{\rm w}}$},\\
 \partial_{x_1}{\bf W}^{(k)}={\bf 0}\quad&\mbox{on $\Gamma_L$}
\end{cases}
\end{equation*}
for the vector field ${\bf F}^{(S^{(k)}, \Lambda^{(k)},{\bf W}^{(k)})}({\vphi}^{(k)},{\Phi}^{(k)})$ given by \eqref{definition-F-vec-field}. So $\breve{\bf W}$ satisfies the estimate
\begin{equation*}
  \|\breve{\bf W}\|_{H^2(\Om_L)}\le C\|{\bf G}^{(1)}-{\bf G}^{(2)}\|_{L^2(\Om_L)}
\end{equation*}
for some constant $C>0$. By using \eqref{definition:vel-HD}, \eqref{definition:density-HD}, \eqref{pphi-est}, \eqref{estimate for E1} and \eqref{intermediate contraction}, we can directly show from the above estimate that there exists a constant $C_{V}>0$ to satisfy
\begin{equation*}
   \|\breve{\bf W}\|_{H^2(\Om_L)}\le C_V\sigma E_2.
\end{equation*}
By combining this inequality with \eqref{estimate for E1} and \eqref{intermediate contraction}, we finally obtain
\begin{equation}
\label{contraction of E2}
  E_2\le C_P(C_V+C_T)\sigma E_2.
\end{equation}
Therefore, if the inequality
\begin{equation*}
  \sigma\le \min\left\{\frac{3}{4C_P(C_V+C_T)},\,\,\bar{\sigma}_T,
  \,\,\frac{\eps^*}{(2C_{\natural}+1)}\right\}
\end{equation*}
holds, then it immediately follows from \eqref{contraction of E2} that $E_2=0$. And, this combined with \eqref{estimate for E1} yields $E_1=0$ as well. Thus we conclude that
\begin{equation*}
(S^{(1)},\Lambda^{(1)},\varphi^{(1)},\Phi^{(1)},\phi^{(1)})=
(S^{(2)},\Lambda^{(2)},\varphi^{(2)},\Phi^{(2)},\phi^{(2)})\quad\tx{in $\ol{\Om_L}$}.
\end{equation*}
\smallskip

The proof of Proposition \ref{Thm-HD} is completed by choosing $L^*$ and $\sigma_2$ as
\begin{equation*}
  L^*:=L_*^{\rm v},\quad \sigma_2:=\min\left\{\frac{3}{4C_P(C_V+C_T)},\,\,\bar{\sigma}_T,
  \,\,\frac{\eps^*}{(2C_{\natural}+1)}\right\},
\end{equation*}
respectively, for the constant $L_*^{\rm v}\in(0, \bar{L}]$ from Lemma \ref{lemma:wp of nlbvp for potentials}.

\qed

\subsection{Proof of Theorem \ref{MainThm}}\label{Sec-Last}

\begin{proof}

Let $L^*$ be fixed to be same with the one from Proposition \ref{Thm-HD}.
\smallskip

{\textbf{1.}} {\emph{(The existence)}} Given boundary data $(b,u_{\rm en},v_{\rm en},w_{\rm en},S_{\rm en},E_{\rm en},E_{\rm ex})$, suppose that the inequality
\begin{equation*}
\sigma(b,u_{\rm en},v_{\rm en},w_{\rm en},S_{\rm en},E_{\rm en},E_{\rm ex})\le \sigma_2
\end{equation*}
holds. Then, Proposition \ref{Thm-HD} yields an axisymmetric solution $(\vphi, \Phi, \phi, S, \Lambda)$ to Problem \ref{EP-Prob-HD}. For such a solution, let $\rho$ and ${\bf u}$ be given by \eqref{definition:density-HD} and \eqref{definition:vel-HD}, respectively. And, let us set the pressure $p$ as
\begin{equation*}
  p:=S\rho^{\gam}.
\end{equation*}
Then ${\bf U}=({\bf u}, \rho, p, \Phi)$ is an axisymmetric solution that solves the full Euler-Poisson system \eqref{E-B}. Also, it satisfies all the boundary conditions stated in Problem \ref{EP-Prob}(see \eqref{Phy-bd}). Furthermore, one can directly check from the properties (a)--(c) stated in Proposition \ref{Thm-HD} that the solution ${\bf U}$ satisfies the estimate \eqref{up-est-vol} and the conditions (ii) and (iii) stated in Problem \ref{EP-Prob}.
\medskip

{\textbf{2.}}{\emph{(The uniqueness)}}
Suppose that  Problem \ref{EP-Prob} has two solutions $({\bf u}^{(1)}, \rho^{(1)}, p^{(1)},\Phi^{(1)})$ and $({\bf u}^{(2)}, \rho^{(2)}, p^{(2)},\Phi^{(2)})$ that satisfy the estimate \eqref{up-est-vol}.

For each $k=1,2$, let us define functions $(S^{(k)},\Lambda^{(k)},\varphi^{(k)},\Phi^{(k)},\phi^{(k)}{\bf e}_{\theta})$ as follows:
\begin{itemize}
\item[(i)] $\displaystyle{S^{(k)}:=\frac{p^{(k)}}{(\rho^{(k)})^{\gamma}}}$;
\item[(ii)] $\displaystyle{\Lambda^{(k)}(x_1, \rx'):=|\rx'|[{\bf e}_{\theta}\cdot{\bf u}^{(k)}](x_1,|\rx'|)}$;
\item[(iii)] the linear boundary value problem
\begin{equation*}
\left\{\begin{split}
&-\Delta {\bf W}=(\partial_{x_1}({\bf u}^{(k)}\cdot{\bf e}_r)-\partial_r({\bf u}^{(k)}\cdot{\bf e}_{x_1})){\bf e}_{\theta}\quad\mbox{in}\,\,\Omega_L,\\
&\partial_{x_1}{\bf W}={\bf 0}\quad\mbox{on}\,\,\Gamma_0,\quad{\bf W}={\bf 0}\quad\mbox{on}\,\,\partial\Omega_L\backslash\Gamma_0
\end{split}\right.
\end{equation*}
has a unique axisymmetric solution ${\bf W}^{(k)}=\phi^{(k)}{\bf e}_{\theta}\in H^5(\Omega_L;\R^3)$;
\item[(iv)] $\displaystyle{\varphi^{(k)}(x_1, \rx'):=\int_0^{x_1}\left[{\bf u}^{(k)}\cdot{\bf e}_{x_1}-\frac{1}{|\rx'|}\partial_r(|\rx'|\phi^{(k)})\right](y,\rx')dy}$.
\end{itemize}
A straightforward computation with using the estimate \eqref{up-est-vol} shows that each $(S^{(k)},\Lambda^{(k)},\varphi^{(k)},\Phi^{(k)},\phi^{(k)}{\bf e}_{\theta})$ satisfies the estimate \eqref{Thm-HD-est} for some constant $C>0$. Furthermore, by using the equation $-\Delta {\bf W}=(\partial_{x_1}({\bf u}^{(k)}\cdot{\bf e}_r)-\partial_r({\bf u}^{(k)}\cdot{\bf e}_x)){\bf e}_{\theta}$ in $\Om_L$, one can directly check from the definition of $\vphi^{(k)}$ that
\begin{equation*}
\nabla\vphi^{(k)}={\bf u}^{(k)}-[{\bf e}_{\theta}\cdot{\bf u}^{(k)}]{\bf e}_{\theta}-\nabla\times(\phi^{(k)}{\bf e}_{\theta})\quad\tx{in $\ol{\Om_L}$}.
\end{equation*}
By combining this equation with the definition of $\Lambda^{(k)}$ given in (ii), we obtain that
\begin{equation*}
  {\bf u}^{(k)}=\nabla\vphi^{(k)}+\nabla\times (h^{(k)}{\bf e}_r+\phi^{(k)} {\bf e}_{\theta})\quad\tx{for $h^{(k)}(x_1, \rx'):=\int \frac{\Lambda^{(k)}(x_1, \rx')}{|\rx'|}\,d x_1$}.
\end{equation*}
This implies that $(S^{(k)},\Lambda^{(k)},\varphi^{(k)},\Phi^{(k)},\phi^{(k)}{\bf e}_{\theta})$ is a solution to Problem \ref{EP-Prob-HD}. Therefore, we can fix a constant $\sigma_1\in(0, \sigma_2]$ depending only on the data and $(L, \bar{\epsilon})$ so that if the inequality
\begin{equation*}
\sigma(b,u_{\rm en},v_{\rm en},w_{\rm en},S_{\rm en},E_{\rm en},E_{\rm ex})\le \sigma_1
\end{equation*}
holds, then the argument given in \S\ref{subsubsection: finalization of the proof}(see Step {\textbf{3}}) yields
\begin{equation*}
 (S^{(1)},\Lambda^{(1)},\varphi^{(1)},\Phi^{(1)},\phi^{(1)}{\bf e}_{\theta})=(S^{(2)},\Lambda^{(2)},\varphi^{(2)},\Phi^{(2)},\phi^{(2)}{\bf e}_{\theta}) \quad\tx{in $\ol{\Om_L}$},
\end{equation*}
from which it directly follows that
\begin{equation*}
  ({\bf u}^{(1)}, \rho^{(1)}, p^{(1)},\Phi^{(1)})=({\bf u}^{(2)}, \rho^{(2)}, p^{(2)},\Phi^{(2)})\quad\tx{in $\ol{\Om_L}$}.
\end{equation*}
This finishes the proof of Theorem \ref{MainThm}.

\end{proof}

\appendix
\section{Proof of Lemma \ref{lemma-main section-smooth approx}}
\label{appendix-smooth approximation}
\begin{proof}
In $\R^2$, fix an open, connected and bounded domain $\mcl{D}$ with a smooth boundary $\der \mcl{D}$.
For a constant $L\in(0,\bar{L}]$, define a three dimensional cylinder $\Omega_L$ by
\begin{equation*}
\Omega_L:=\left\{\rx=(x_1,\rx')\in\mathbb{R}^3:\mbox{ }  0<x_1<L,\, \rx'=(x_2,x_3)\in\mathcal{D}\right\}.
\end{equation*}
And, fix a function $u:\ol{\Om_L}\rightarrow \R$.
\medskip

{\textbf{1.}}(Extension about $\Gam_0$) Suppose that $u$ satisfies the compatibility condition
    \begin{equation}
   \label{comp-cond-u-2}
      \der_{x_1}^{k-1}u=0\quad\tx{on $\Gam_0^{\bar{\epsilon}}$ for $k=1,3$.}
    \end{equation}

     First of all, we define an extension of $u$ by
    \begin{equation*}
      \mcl{R}u(x_1, \rx'):=\begin{cases}
      u(x_1, \rx')\quad&\mbox{for $0\le x_1<L$},\\
      -u(-x_1, \rx')\quad&\mbox{for $-\frac L2<x_1<0$}.
      \end{cases}
    \end{equation*}
    Next, we define another extension of $u$ by
    \begin{equation*}
      \mcl{S}u(x_1, \rx'):=\begin{cases}
     u(x_1, \rx')\quad&\mbox{for $0\le x_1<L$},\\
      \sum_{j=0}^4 c_ju(-\frac{x_1}{2^j}, \rx')\quad&\mbox{for $-\frac L2<x_1<0$}
      \end{cases}
    \end{equation*}
    for $(c_0, c_1, \cdots, c_4)$ solving the following linear system:
    \begin{equation*}
      \sum_{j=0}^4 \left(-\frac{1}{2^j}\right)^kc_j=1\quad\tx{for }k=0, 1, 2, 3, 4.
    \end{equation*}

    Since $\der \mcl{D}$ is smooth, we can fix a smooth function $\xi:\ol{\mcl{D}}\rightarrow \R_+ $ to satisfy the following properties:
    \begin{equation*}
      \xi(\rx')=\begin{cases}
      1\quad&\mbox{if ${\rm dist}(\rx', \der\mcl{D})\le\frac{3\bar{\epsilon}}{4}$}\\
      0\quad&\mbox{if ${\rm dist}(\rx', \der\mcl{D})\ge \frac{4\bar{\epsilon}}{5}$}
      \end{cases}\quad\tx{and}\quad 0\le \xi\le 1\,\,\tx{in $\ol{\mcl{D}}$}.
    \end{equation*}
Finally, we define an extension $\mcl{E}u$ of $u$ onto $(-\frac L2,L)\times \mcl{D}_*$ by
\begin{equation*}
 \mcl{E}u(x_1, \rx'):=\begin{cases}
u(x_1, \rx')\quad&\mbox{for $\rx=(x_1, \rx')\in\ol{\Om_L}$},\\
\xi(\rx')\mcl{R}u(x_1, \rx')
+(1-\xi(\rx'))\mcl{S}u\quad&\mbox{for $\rx=(x_1, \rx')\in (\ol{\Om_L})^c$}.
  \end{cases}
\end{equation*}

The linear extension operator $\mcl{E}:H^4_*(\Om_L)\cap \mcl{W}^{4,\infty}_{*, \mcl{D}}(0,L)\rightarrow H^4_*((-\frac L2, L)\times \mcl{D})\cap \mcl{W}^{4,\infty}_{*, \mcl{D}}(-\frac L2,L)$, given in the above, satisfies the following properties provided that a function $u$ satisfies \eqref{comp-cond-u-2}:
\begin{itemize}
\item[(i)] $\displaystyle{\mcl{E} u=u\quad\tx{in $\Om_L$}}$;
\item[(ii)] $\mcl{E}u$ is odd with respect to $x_1$ about $x_1=0$ for $\rx=(x_1, \rx')$ with $|x_1|<\frac L4$ and ${\rm dist}(\rx', \der\mcl{D})\le \frac{3\bar{\epsilon}}{4}$;
\item[(iii)] there exists a constant $\mu_1>0$ depending only on $(\mcl{D}, L)$ so that
    \begin{equation}
    \label{estimate-extension1}
    \begin{split}
      &\|\mcl{E}u\|_{H^4_*((-\frac L2, L)\times \mcl{D})}\le \mu_1\|u\|_{H^4_*(\Om_L)},\\
      &\|\mcl{E}u\|_{\mcl{W}^{4,\infty}_{*,\mcl{D}}(-\frac L2,L)}\le \mu_1\|u\|_{\mcl{W}^{4,\infty}_{*,\mcl{D}}(0,L)}.
      \end{split}
    \end{equation}
\end{itemize}
\medskip

{\textbf{2.}} Let $\chi:\R\rightarrow \R$ be a smooth function that satisfies the following conditions:
\begin{itemize}
\item[-] $\chi(x_1)\ge 0$ for all $x_1\in \R$;
\item[-] $\chi(x_1)=\chi_1(-x_1)$ for all $x_1\in \R$;
\item[-] ${\rm spt}\,\chi\subset (-1,1)$;
\item[-] $\int_{\R}\chi(x_1)\,dx_1=1$.
\end{itemize}
For a constant $\tau>0$, we define a function $\chi^{(\tau)}:\R\rightarrow \R$ by
\begin{equation*}
  \chi^{(\tau)}(x_1):=\frac{1}{\tau}\chi\left(\frac{x_1}{\tau}\right).
\end{equation*}

We define {\emph{a partially smooth approximation}} of $u$ away from $\Gam_L$ by
\begin{equation}
\label{definition-u3}
  u_1^{(\tau)}(x_1, \rx'):= \int_{\R} \mcl{E}u(x_1-y_1, \rx')\chi^{(\tau)}(y_1)\,dy_1.
\end{equation}

\begin{lemma}
\label{lemma-approx type 1}
For every constant $\tau>0$ satisfying the inequality $0<\tau<\frac{1}{10}\min\{1, \frac L2\}$, the function $u_1^{(\tau)}$ given by \eqref{definition-u3} satisfies the following properties:
\begin{itemize}
\item[(a)] for each fixed $\rx'\in\mcl{D}$, $u_1^{(\tau)}(\cdot, \rx')$ is $C^{\infty}$ with respect to $x_1\in [0, \frac{9}{10}L]$;
\item[(b)] the compatibility condition $\displaystyle{\der_{{\bf n}_w}u_1^{(\tau)}=0}$ holds on $\Gamw$;
\item[(c)] the compatibility condition $\displaystyle{\der_{x_1}^{k-1}u_1^{(\tau)}=0}$ holds on $\Gam_0^{\bar{\epsilon}/2}$ for $k=1,3$;
\item[(d)] there exists a constant $\til{\mu}>0$ depending only on $(\mcl{D}, L)$ to satisfy the following estimates:
    \begin{align}
    \label{estimate-approximation-1-for u3}
      &\|u_1^{(\tau)}\|_{H^4(\Om_{9L/10})}\le \til{\mu} \|u\|_{H^4(\Om_{9L/10})},\\
       \label{estimate-approximation-2-for u3}
      &\|u_1^{(\tau)}\|_{\mcl{W}^{4,\infty}_{\mcl{D}}(0,9L/10)}\le \til{\mu} \|u\|_{\mcl{W}^{4,\infty}_{\mcl{D}}(0,9L/10)};
      \end{align}
    \item[(e)] $\displaystyle{\lim_{\tau\to 0+} \|u_1^{(\tau)}-u\|_{H^4(\Om_{9L/10})}=0}$.
\end{itemize}
\end{lemma}

{\textbf{3.}}
To define an extension of the function $u$ up to the boundary $\Gam_L$, we shall use a translation operator.
For a constant $\tau$ satisfying the inequality $0<\tau<\frac{1}{10}\min\{1,  \frac L2\}$,
let us define a translation operator $T_{\tau}$ by
    \begin{equation*}
      T_{\tau}u(x_1, \rx'):=u(x_1-2\tau, \rx').
    \end{equation*}
    And, we define a function $u_2^{(\tau)}$ by
    \begin{equation*}
      u_2^{(\tau)}(x_1, \rx'):=(T_{\tau}u_1^{(\tau)})(x_1, \rx').
    \end{equation*}
    Then the function $u_2^{(\tau)}$ satisfies the following properties:
    \begin{itemize}
\item[(i)] for each fixed $\rx'\in\mcl{D}$, $u_2^{(\tau)}(\cdot, \rx')$ is $C^{\infty}$ with respect to $x_1\in [\frac{L}{2}, L]$;

\item[(ii)] $\displaystyle{\lim_{\tau\to 0+} \|u_2^{(\tau)}-u\|_{H^3(\Om_L\cap\{x_1> \frac L2\})}=0};$
\item[(iii)] the compatibility condition
    $\displaystyle{\der_{{\bf n}_w}u_2^{(\tau)}=0}$ holds on $\Gamw\cap\left\{x_1>\frac 45 L\right\}$.
\end{itemize}

{\textbf{4.}} Let $\zeta$ be a cut-off function that satisfies the following properties:
    \begin{itemize}
    \item[(i)] $\displaystyle{\zeta\in C^{\infty}(\R)};$
    \item[(ii)] $\displaystyle{\zeta(x_1)=\begin{cases}1&\quad\mbox{for $x_1<\frac L2$},\\
        0&\quad\mbox{for $x_1>\frac 23 L$};\end{cases}}$
    \item[(iii)] $\displaystyle{0\le \zeta\le 1\quad\tx{and}\quad \zeta'\le 0\quad\tx{on $\R$}}$.
    \end{itemize}

Finally, we define a partially smooth global approximation of $u$ by
    \begin{equation}
    \label{definition-smooth approximation}
      u^{(\tau)}(\rx):=\zeta(x_1)u_1^{(\tau)}(\rx)+(1-\zeta(x_1))u_2^{(\tau)}(\rx)
    \end{equation}
for $\rx=(x_1, \rx')\in \Om_L$

\begin{lemma}
\label{lemma-smooth approximation}
Suppose that a function $u\in H^4_*(\Om_L)\cap \mcl{W}^{4,\infty}_{*,\mcl{D}}(0,L)$ satisfies the compatibility condition \eqref{comp-cond-u-2}.
Then, there exists a constant $\bar{\tau}>0$ sufficiently small depending only on $(\mcl{D}, L, \bar{\epsilon})$, so that, for any $\tau\in(0,\bar{\tau}]$,  the function $u^{(\tau)}$ given by \eqref{definition-smooth approximation} satisfies the following properties:
\begin{itemize}
\item[(a)] for each fixed $\rx'\in\mcl{D}$, $u^{(\tau)}(\cdot, \rx')$ is $C^{\infty}$ with respect to $x_1\in [0, L]$;
\item[(b)] the compatibility condition $\displaystyle{\der_{{\bf n}_w}u^{(\tau)}}=0$ holds on $\Gamw$;
\item[(c)] the compatibility condition $\displaystyle{\der_{x_1}^{k-1}u^{(\tau)}=0}$ holds on $\Gam_0^{\bar{\epsilon}/2}$ for $k=1,3$;
\item[(d)] $\displaystyle{\lim_{\tau \to 0+} \|u^{(\tau)}-u\|_{H^3(\Om_L)}=0}$;
\item[(e)] there exists a constant $\mu>0$ depending only on $(\mcl{D}, L)$ so that, for any given $\tau$ satisfying the inequality $0<\tau<\frac{1}{10}\min\{1, \frac L2\}$, the function $u^{(\tau)}$ satisfies the following estimates:
    \begin{align}
    \label{estimate-approximation-1}
      &\|u^{(\tau)}\|_{H^4_*(\Om_L)}\le \mu \|u\|_{H^4_*(\Om_L)},\\
       \label{estimate-approximation-2}
      &\|u^{(\tau)}\|_{\mcl{W}^{4,\infty}_{*,\mcl{D}}(0,L)}\le \mu \|u\|_{\mcl{W}^{4,\infty}_{*,\mcl{D}}(0,L)}.
    \end{align}
\end{itemize}

\begin{proof}
The statements given in (a)--(d) can be directly checked from the definition \eqref{definition-smooth approximation}. So we only need to prove the statement (e).
\medskip

For $k=0,1,\cdots, 4$, we have
\begin{equation}
\label{derivative of Du-tau}
\begin{split}
  D^ku^{(\tau)}(\rx)&=\zeta(x_1)D^ku_1^{(\tau)}(\rx)
  +(1-\zeta(x_1))T_{\tau}(D^ku_1^{(\tau)})(\rx)\\
  &+\sum_{j=1}^k D^j\zeta(x_1)\left(D^{k-j}u_1^{(\tau)}
  -T_{\tau}(D^{k-j}u_1^{(\tau)})\right)(\rx).
\end{split}
\end{equation}
By using the expression given in the right above, one can directly check that
\begin{equation*}
      \begin{split}
      &\|u^{(\tau)}\|_{H^3(\Om_L)}\le \mu \|u\|_{H^3(\Om_L)}\quad
   \tx{and}\quad   \|u^{(\tau)}\|_{\mcl{W}^{3,\infty}_{\mcl{D}}(0,L)}\le \mu \|u\|_{\mcl{W}^{3,\infty}_{\mcl{D}}(0,L)}
      \end{split}
    \end{equation*}
    for some constant $\mu>0$ depending only on $(\mcl{D}, L)$.

By taking $k=4$ in \eqref{derivative of Du-tau}, we can write $D^4u^{(\tau)}$ as
\begin{equation*}
  D^4u^{(\tau)}=\mcl{M}_{\tau}+\mcl{R}_{\tau}
\end{equation*}
for
\begin{equation*}
  \begin{split}
  &\mcl{M}_{\tau}:=\zeta(x_1)D^4u_1^{(\tau)}+(1-\zeta(x_1))D^4u_2^{(\tau)},\\
  &\mcl{R}_{\tau}:=\sum_{m=1}^4D^m\zeta(x_1)\left(D^{4-m}u_1^{(\tau)}
  -D^{4-m}u_2^{(\tau)}\right).
  \end{split}
\end{equation*}
Since $|D^m\zeta|=0$ for $x_1\ge \frac 23 L$, a direct computation shows that there exists a constant $\lambda_0>0$ depending only on $(\mcl{D}, L)$ so that, for any given $\tau$ satisfying the inequality $0<\tau<\frac{1}{10}\min\{1, \frac L2\}$, we have the estimate
\begin{equation*}
\|\mcl{R}_{\tau}\|_{L^2(\Om_L)}\le \lambda_0\|u\|_{H^4(\Om_L\cap \{x_1<\frac 56 L\})}.
\end{equation*}
For $d\ge \frac L4$, one can directly check that if $\tau<\frac{L}{10}$, then $\mcl{M}_{\tau}$ satisfies the estimate
\begin{equation*}
\|\mcl{M}_{\tau}\|_{L^2(\Om_L\cap\{x_1<L-d\})}\le \lambda_1\|u\|_{H^4(\Om_L)\cap\{x_1< \frac{9}{10}L\}}
\end{equation*}
for some constant $\lambda_1>0$ fixed depending only on $(\mcl{D}, L)$.

Next, let us fix $d\in(0, \frac L4)$. Since $\zeta(x_1)=0$ for $x_1\ge \frac 34 L$, we have the estimate
\begin{equation*}
  \|\zeta D^4 u_1^{(\tau)}\|_{L^2(\Om_{L}\cap\{x_1<L-d\})}\le \lambda_2 \|u\|_{H^4(\Om_L\cap \{x_1< \frac 34 L\})}
\end{equation*}
for some constant $\lambda_2>0$ fixed depending only on $(\mcl{D}, L)$.

Now we shall estimate $\|(1-\zeta)D^4 u_2^{(\tau)}\|_{L^2(\Om_L\cap\{x_1<L-d\})}$. First of all, we can directly check that
\begin{equation*}
  \begin{split}
  &\int_{\Om_L\cap\{x_1<L-d\}} |(1-\zeta(x_1))D^4 u_2^{(\tau)}|^2\,d\rx\\
  &\le \int_{\Om_L\cap\{ \frac L2<x_1<L-d\}}|D^4T_{\tau}u_1^{(\tau)}(\rx)|^2\,d\rx\\
  &\le \int_{\Om_L\cap\{ \frac L2<x_1<L-d\}}\int_{(-\tau,\tau)}|D^4\mcl{E}u(x_1-2\tau-y_1, \rx')|^2\chi_1^{(\tau)}(y_1)\,dy_1 d\rx\\
  &\le  \int_{\Om_L\cap\{ \frac L4<x_1<L-d-\tau\}}|D^4u(\rx)|^2\,d\rx\\
  &\le (d+\tau)^{-1}\|u\|_{H^4_*(\Om_L)}^2\\
  &\le d^{-1} \|u\|_{H^4_*(\Om_L)}.
  \end{split}
\end{equation*}
So we can prove the estimate \eqref{estimate-approximation-1}. Since the estimate \eqref{estimate-approximation-2} can be verified similarly, we skip its proof.
\end{proof}
\end{lemma}

With a minor adjustment in Step 1 (simply define $\mcl{R}$ as an even extension (about $x_1=0$) operator), we easily obtain the following lemma:
\begin{lemma}
\label{lemma-smooth approximation-type2}
Suppose that a function $w\in H^4_*(\Om_L)$ satisfies the following two compatibility conditions:
\begin{equation*}
 \begin{split}
&\der_{{\bf n}_w} w=0\quad\tx{on $\Gamw$},\\
&\der_{x_1}^{k}w=0\quad\tx{on $\Gam_0^{\bar{\epsilon}}$ for $k=1,3$.}
\end{split}
\end{equation*}
Then one can fix a constant $\bar{\tau}>0$ sufficiently small depending only on $(\mcl{D}, L, \bar{\epsilon})$, so that, for any $\tau\in(0,\bar{\tau}]$,  one can define an approximation $w^{(\tau)}$ of $w$ so that the following properties are satisfied:
\begin{itemize}
\item[(a)] for each fixed $\rx'\in\mcl{D}$, $w^{(\tau)}(\cdot, \rx')$ is $C^{\infty}$ with respect to $x_1\in [0, L]$;
\item[(b)] the compatibility condition $\displaystyle{\der_{{\bf n}_w}w^{(\tau)}}=0$ holds on $\Gamw$;
\item[(c)] the compatibility condition $\displaystyle{\der_{x_1}^{k}w^{(\tau)}=0}$ holds on $\Gam_0^{\bar{\epsilon}/2}$ for $k=1,3$;
\item[(d)] $\displaystyle{\lim_{\tau \to 0+} \|w^{(\tau)}-w\|_{H^3(\Om_L)}=0}$;
\item[(e)] there exists a constant $\mu>0$ depending only on $(\mcl{D}, L)$ so that, for any given $\tau$ satisfying the inequality $0<\tau<\frac{1}{10}\min\{1, \frac L2\}$, the function $w^{(\tau)}$ satisfies the following estimates:
    \begin{align*}
    \label{estimate-approximation-1-type2}
      &\|w^{(\tau)}\|_{H^4_*(\Om_L)}\le \mu \|w\|_{H^4_*(\Om_L)}.
    \end{align*}
\end{itemize}
\end{lemma}

{\textbf{5.}} For each $m\in \mathbb{N}$, let us set
\begin{equation*}
  \tau_m:=\frac{1}{10 m}\min\left\{1, \frac{L}{2}\right\}.
\end{equation*}
Finally, the proof of Lemma \ref{lemma-main section-smooth approx} is complete if we define $\tpsi_m$ by
\begin{equation*}
  \tpsi_m:=\tpsi^{(\tau_m)}
\end{equation*}
for $\tpsi^{(\tau)}$ given by applying Lemma \ref{lemma-smooth approximation}, and if we define $\tPsi_m$ by
\begin{equation*}
  \tPsi_m:=\tPsi^{(\tau_m)}
\end{equation*}
for $\tPsi^{(\tau)}$ given by applying Lemma \ref{lemma-smooth approximation-type2}.
\end{proof}

\newpage

\section{A comment on the proof of Lemma \ref{lemma-weak star limit} }
\label{appendix-B}
\begin{proposition}
\label{proposition-pre BA}
Given a Hilbert space $H$, we have
\begin{equation*}
  (L^1(0,L;H))^*=L^{\infty}(0,L;H).
\end{equation*}
\end{proposition}

In order to verify Proposition \ref{proposition-pre BA}, we shall apply the following lemma:
\begin{lemma}
\label{lemma-VM thm1}\cite[Chapter IV, Theorem 1]{Diestel}
Let $(\Om, \Sigma, \mu)$ be a finite measure space, $1\le p<\infty$, and $X$ be a Banach space. Then $L_p(\mu, X)^*=L_q(\mu, X^*)$ where $\displaystyle{\frac 1p+\frac 1q=1}$, if and only if $X^*$ has the Radon-Nikod\'{y}m property with respect to $\mu$.
\end{lemma}

In applying Lemma \ref{lemma-VM thm1} to prove Proposition \ref{proposition-pre BA}, we take $X=H$, thus $X^*=H$ by the Riesz representation theorem. And, we take
\begin{equation*}
  \Om=(0, L),\quad \mu=m(\tx{Lebesgue measure}).
\end{equation*}
So the main question to investigate is the following:
\begin{question}
Does a Hilbert space $H$ satisfy the Radon-Nikod\'{y}m property with respect to the Lebesgue measure $m$ on $(0,L)$?
\end{question}

\begin{definition}
\label{definition-RN property}
A Banach space $X$ has the {\emph{Radon-Nikod\'{y}m property}} with respect to $(\Om, \Sigma, \mu)$ if for each $\mu$-continuous vector measure $G:\Sigma\rightarrow X$ of bounded variation, there exists $g\in L_1(\mu, X)$ such that
$$
G(E)=\int_{E}g\,d\mu
$$
for all $E\in \Sigma$.
\end{definition}

\begin{lemma}[{von Neumann, \cite[Chapter IV, Corollary 4]{Diestel}}]
\label{lemma-von Neumann-RN property}
Hilbert spaces have the {\emph{Radon-Nikod\'{y}m property}}.
\end{lemma}

\begin{proof}[Proof of Proposition \ref{proposition-pre BA}]
Now, Proposition \ref{proposition-pre BA} easily follows from Lemmas \ref{lemma-VM thm1} and \ref{lemma-von Neumann-RN property}.
\end{proof}

\newpage
\section{A comment on Lemma \ref{lemma-coercivity}}
\label{appendix-C}
The upper bound $L^*$ of the nozzle length $L$ is given for the sole purpose of achieving Lemma \ref{lemma-coercivity}, the essential ingredient to establish an a priori $H^1$-estimate (see Proposition \ref{proposition-H1-new}) for a solution to a linear system consisting of a second order hyperbolic differential equation and a second order elliptic equation weakly coupled together.
As we have mentioned before, the proof of Lemma \ref{lemma-coercivity} can be verified by following the arguments in \cite{bae2021structural, bae2021three}, in which one introduces an energy weight function, and use it to derive an a prior $H^1$-estimate.
In this appendix, we give an energy weight function different from the one used in \cite{bae2021structural, bae2021three} so that one can fix the upper bound $L^*$ differently under some additional assumptions. The idea is from the work \cite{bae2023two} that investigates the existence of multi- dimensional $C^1$-accelerating transonic flows of Euler-Poisson system. In this work, we use a translation of the background Mach number as an energy weight function. This particular choice is possible due to the condition \eqref{special condition}. This condition is given to overcome a technical difficulty arising because the wall boundary $\Gam_w$ of $\Om_L$ is non-flat. Interestingly, the condition \eqref{special condition} implies that a background solution has a deceleration near the entrance of the nozzle, and this observation is a key for Lemma \ref{App-lemma}.
\medskip

\begin{psfrags}
\begin{figure}[htp]
\centering
\psfrag{a}[cc][][0.8][0]{${\bf 0}$}
\psfrag{b}[cc][][0.8][0]{$\bar u$}
\psfrag{c}[cc][][0.8][0]{$\bar E$}
\psfrag{d}[cc][][0.8][0]{$\frac{b_0}{J_0}$}
\psfrag{e}[cc][][0.8][0]{$u_s\phantom{aa}$}
\psfrag{h}[cc][][0.7][0]{(i)}
\psfrag{g}[cc][][0.7][0]{(ii)}
\psfrag{f}[cc][][0.7][0]{(iii)}
\includegraphics[scale=0.7]{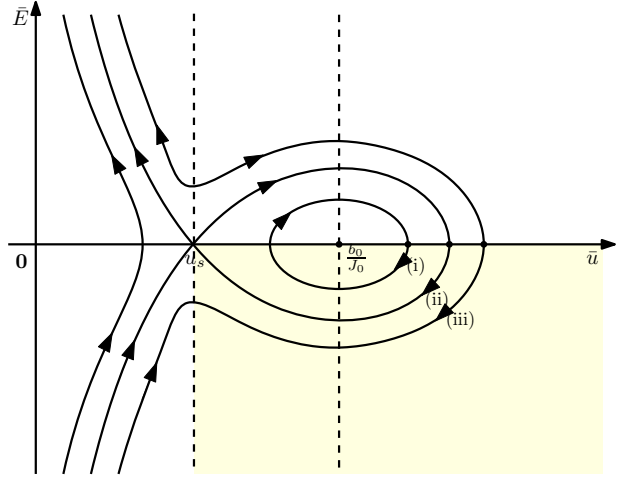}
\caption{The $\bar u$-$\bar E$ phase plane: In (i), the background solution is period, therefore we have $L^*<\bar{L}$. In (ii) and (iii), we can fix $L^*$ as $L^*=\bar L$.}
\label{figurec}
\end{figure}
\end{psfrags}
\medskip

Given a constant $\rho_0\in(0, \rhos)$, let $(\bar{\rho}, {\bar u}, S_0, \Phi_0)$ be the background solution to \eqref{E-B} associated with $(\rho_0, 0)$ in the sense of Definition \ref{definition:background sol}. For a fixed constant $\bar{\delta}\in(0, \rhos)$, let $\bar{L}$ be the constant given as in Lemma \ref{Lem1}. Then, one can fix a constant $L^*\in(0, \bar{L}]$ so that the background solution satisfies the following properties(see Figure \ref{figurec}):
\begin{equation*}
  \bar{\rho}'(x_1)>0(\Leftrightarrow \bar u'(x_1)<0),\quad
  \bar{u}(x_1)>u_s(:=\frac{J_0}{\rhos}),
  \quad
  \bar{E}(x_1)<0\quad
  \mbox{ for }x_1\in(0,{L}^{\ast}].
\end{equation*}
\medskip

\begin{lemma}\label{App-lemma} Fix a constant $\mfrak{a}\in(0, L^*)$, and let $\eps_0$ be the small constant given in Lemma \ref{lemma-properties of coeffs}. Then, there exists a small constant $d\in(0,1)$ depending only on the data and $\mfrak{a}$ so that if
\begin{itemize}
\item[-] $L$ satisfies $L\le L^{*}-\mathfrak{a}$,
\item[-] $\delta$ from the definition \eqref{definition of iteration set} of $\mcl{J}_{\delta}$ satisfies $\delta\le \eps_0$,
\item[-] for $x_1\in[\mathfrak{a},L+\mathfrak{a}]$, it holds that
\begin{equation}\label{range of u}
\frac{\bar{u}(\mathfrak{a})}{u_s}-d\le\frac{\bar{u}(x_1)}{u_s}\le\frac{\bar{u}(\mathfrak{a})}{u_s},
\end{equation}
\end{itemize}
then we have the inequality \eqref{estimate-I main} for some positive constants $\mu_0$ and $\mu_1$ fixed depending only on the data and $\mfrak{a}$.
\end{lemma}

\begin{proof}
The Mach number function $M(x_1)$ of a background solution $(\bar{\rho}, \bar u, S_0, \bar{\Phi}(=\int \bar E\,dx_1))$(in the sense of Definition \ref{definition:background sol}) is given by
\begin{equation*}
  M(x_1):=\frac{\bar u(x_1)}{\sqrt{\gam S_0 \bar{\rho}^{\gam-1}(x_1)}}.
\end{equation*}
For two constants $\mfrak a\in(0, L^*)$ and $\eta>0$ to be determined, let us set a constant $K$ as
\begin{equation}
\label{definition of K}
  K:=M(\mfrak a),
\end{equation}
and
define a function $\mathfrak{M}:[0,L]\to\mathbb{R}$ by
\begin{equation}\label{definition-Ma}
\mathfrak{M}(x_1):=
\left(\frac{M(x_1+\mfrak a)}{K}\right)^{\eta}=:(\mcl{R}_{\mathfrak{a}})^{\eta}(x_1).
\end{equation}
\newcommand \Ra{\mcl{R}_{\mfrak a}}
Then $\mathfrak{M}$ satisfies the following properties:
\begin{equation*}
\begin{split}
\mathfrak{M}'(x_1)(=\eta \Ra^{\eta-1}\Ra')<0\quad\tx{and}\quad
0<\mfrak{M}(L)<\mathfrak{M}(x_1)<1\quad\tx{for $0<x_1<L$}.
\end{split}
\end{equation*}
We rewrite the term $\mathcal{I}_{\rm main}$ given in \eqref{I-main} as follows:
\begin{equation*}
\mathcal{I}_{\rm main}=\mathcal{I}_{\rm bd}+\mathcal{I}_{\rm coer}+\mathcal{I}_{\rm mix}
\end{equation*}
for $\mathcal{I}_{\rm bd}$, $\mathcal{I}_{\rm coer}$, and $\mathcal{I}_{\rm mix}$ defined by
\begin{eqnarray*}
\mathcal{I}_{\rm bd}&:=&\int_{\Gamma_L}((\partial_1V)^2-\sum_{i,j=2}^3{a}_{ij}\partial_i V\partial_j V)\frac{\mathfrak{M}}{2}d{\bf x'}-\frac{1}{2}\int_{\Gamma_0}g_1^2\mathfrak{M}d{\bf x}',\\
\mathcal{I}_{\rm coer}&:=&\int_{\Om_L}(\bar{a}_1\mathfrak{M}-\frac{1}{2}\mathfrak{M}')(\partial_1V)^2+\frac{1}{2}\sum_{i,j=2}^3\partial_1(\bar{a}_{ij}\mathfrak{M})\partial_i V\partial_j V d{\bf x}\\
	&\nonumber&+\int_{\Om_L}|\nabla W|^2+\bar{h}_1W^2 d{\bf x},\\
\mathcal{I}_{\rm mix}&:=&\int_{\Om_L}(\bar{b}_1\partial_1W+\bar{b}_2W)\mathfrak{M}\partial_1V+\bar{h}_2W\partial_1V d{\bf x}.
\end{eqnarray*}

Since $\mathfrak{M}>0$ and $[a_{ij}]_{i,j=2}^3$ is negative definite(see Lemma \ref{lemma-properties of coeffs}(c)), there exist positive constants $\nu_0$ and $\nu_1$ such that
\begin{equation*}
\mathcal{I}_{\rm bd}\ge \nu_0\int_{\Gamma_L}|\nabla V|^2 d{\bf x}'-\nu_1\int_{\Gamma_0}g_1^2 d{\bf x}' .
\end{equation*}
By using the Cauchy-Schwartz inequality, one can check that
\begin{equation*}
\begin{split}
\mathcal{I}_{\rm coer}+\mathcal{I}_{\rm mix}
\ge& \int_{\Om_L}\alpha|\partial_1V|^2 +\frac{1}{2}\sum_{i,j=2}^3\partial_1(\bar{a}_{ij}\mathfrak{M})\partial_i V\partial_j V d{\bf x}
+\int_{\Om_L}\frac{3}{4}|\nabla W|^2+\frac{3}{4}\bar{h}_1W^2 d{\bf x}
\end{split}
\end{equation*}
for $\alpha$ defined by
\begin{equation*}
\alpha:=(\bar{a}_1\mathfrak{M}-\frac{1}{2}\mathfrak{M}')-2(\bar{b}_1\mathfrak{M})^2-\frac{2}{\bar{h}_1}(\bar{b}_2\mathfrak{M}+\bar{h}_2)^2.
\end{equation*}
The inequality \eqref{estimate-I main} is achieved if we find a constant $\mu>0$ satisfying that
\begin{equation}\label{condition-end}
\alpha\ge \mu\quad\mbox{and}\quad [\partial_1(\bar{a}_{ij}\mathfrak{M})]_{i,j=2}^3\ge \mu \mathbb{I}_2\quad\tx{in $\ol{\Om_L}$}.
\end{equation}
Since $\bar{a}_{23}=\bar{a}_{32}=0$ and $\bar{a}_{22}=\bar{a}_{33}$, \eqref{condition-end} is equivalent to
\begin{equation}\label{condition-end-2}
\alpha(x_1)\ge \mu\quad\mbox{and}\quad\partial_1(\bar{a}_{22}\mathfrak{M})(x_1)\ge \mu\quad\tx{in $\ol{\Om_L}$}.
\end{equation}

One can easily check $\partial_1(\bar{a}_{22}\mathfrak{M})\ge M_{\mathfrak{a}}^{\eta}$ in $\ol{\Om_L}$ if and only if $\eta$ satisfies
\begin{equation}\label{eta-1-con}
\eta\ge\frac{(\Ra)^{\eta}(x_1)-\partial_1\bar{a}_{22}{(\Ra)}^{\eta}(x_1)}
{\bar{a}_{22} {(\Ra)}^{\eta-1}(x_1)\Ra'(x_1)}\quad\tx{for $\le x_1\le L$}.
\end{equation}
So we shall choose $\mu\in(0,1]$ and $\eta$ satisfying
\begin{equation}\label{mathfrakl0}
0<\mu\le\min_{x_1\in[0,L]}({\Ra})^{\eta}(x_1)=({\Ra})^{\eta}(L)\quad\mbox{and}\quad
\eta\ge\mathfrak{l}_0:=\max_{x_1\in[0,L]}\left(\frac{(1-\partial_1\bar{a}_{22}(x_1))\Ra(x_1)}{\bar{a}_{22} (x_1) \Ra'(x_1)}\right)
\end{equation}
so that $\partial_1(\bar{a}_{22}\mathfrak{M})(x_1)\ge\mu$ holds in $\ol{\Om_L}$.

By a direct computation, $\alpha$ can be expressed as
\begin{equation*}
\begin{split}
\alpha(x_1)
&=\sum_{l=1}^2P_l(x_1)+\sum_{k=1}^4N_k(x_1)
\end{split}
\end{equation*}
for $N_k$ ($k=1,2,3,4)$ and $P_l$ ($l=1,2$) defined by
\begin{equation}\label{NandP}
\begin{split}
&N_1:=-\frac{(-\bar{E})(\gamma \bar{u}^2+\gamma S_0\bar{\rho}^{\gamma-1})}{(\gamma S_0\bar{\rho}^{\gamma-1}-\bar{u}^2)^2}\Ra^{\eta}\\
&N_2:=-\frac{2\bar{u}^2}{(\bar{u}^2-\gamma S_0\bar{\rho}^{\gamma-1})^2}\Ra^{2\eta}\\
&N_3:=-2\frac{\gamma S_0\bar{\rho}^{\gamma-1}}{\bar{\rho}}\left(\frac{-(\gamma-1)\bar{u}}{\gamma S_0\bar{\rho}^{\gamma-1}-\bar{u}^2}\frac{\bar{\rho}'}{\bar{\rho}}\right)^2\Ra^{2\eta}\\
&N_4:=-2\frac{\bar{u} J_0}{\gamma S_0\bar{\rho}^{\gamma-1}}\\
&P_1:=-\frac{1}{2}\eta \Ra^{\eta-1}\Ra'\\
&P_2:=4\Ra^{\eta}\frac{(\gamma-1)\bar{u}^2}{\bar{u}^2-\gamma S_0\bar{\rho}^{\gamma-1}}\frac{\bar{\rho}'}{\bar{\rho}}.
\end{split}
\end{equation}
Since $\mfrak{a}>0$, it directly follows from the conditions stated in Lemma \ref{App-lemma} that
\begin{equation*}
  N_k(x_1)<0 \quad\tx{and} \quad P_l(x_1)>0\quad\tx{for $0< x_1\le L$}.
\end{equation*}
Next, we shall  rewrite \eqref{NandP} in terms of ($J_0$, $h_0$, $\kappa$) for $J_0$ given in \eqref{one-re},
\begin{equation*}
h_0:=(\gamma S_0)^{\frac{1}{\gamma+1}}\mbox{ and } \kappa(x_1):=\frac{\bar{u}(x_1)}{u_s}\mbox{  for }{u}_s:=h_0 J_0^{\frac{\gamma-1}{\gamma+1}}.
\end{equation*}
For that purpose, we first rewrite essential functions by a direct computation as follows:
\begin{equation}\label{re-functions}
\left.\begin{split}
&\bar{u}=u_s\kappa=h_0J^{\frac{\gamma-1}{\gamma+1}}_0\kappa,\quad\bar{\rho}=\frac{J_0}{\bar{u}}=h_0^{-1}J_0^{\frac{2}{\gamma+1}}\kappa^{-1},\\
&\gamma S_0\bar{\rho}^{\gamma-1}=h_0^{2}J_0^{\frac{2(\gamma-1)}{\gamma+1}}\kappa^{-\gamma+1},\quad
\Ra
=\frac{1}{K}\kappa_{\mathfrak{a}}^{\frac{\gamma+1}{2}}\quad\mbox{for }\kappa_{\mathfrak{a}}(x_1):=\kappa(x_1+\mathfrak{a}).
\end{split}\right.
\end{equation}
It is well known that the solution $(\bar{\rho},\bar{E})$ to \eqref{one-re-uE} satisfies $\frac{1}{2}\bar{E}^2-H(\bar{\rho})=\frac 12 E_0^2-H(\rho_0)(=:k_0)$, where $H(\bar{\rho})$ is defined by
$$H(\bar{\rho}):=\int_{\rho_s}^{\bar{\rho}}\frac{(t-b_0)}{t}\left(\gamma S_0 t^{\gamma-1}-\frac{J_0^2}{t^2}\right)dt\quad\tx{for $\rho_s:=\frac{J_0}{u_s}$}.$$
Owing to the assumption of $E_0=0$ as given in \eqref{special condition} and the condition $\bar E(x_1)<0$ for $x_1\in(0,L^*]$, we can express $-\bar E$ as
\begin{equation}\label{-E-rewritten}
-\bar{E}=\sqrt{2h_0J_0^{\frac{2\gamma}{\gamma+1}}\mathcal{F}(\kappa)+2k_0}
\end{equation}
for $\mathcal{F}$ defined by
\begin{equation*}
\mathcal{F}(\kappa):=\int_1^{\kappa}\left(1-\frac{t}{\zeta_0}\right)\left(1-\frac{1}{t^{\gamma+1}}\right)dt,\quad \zeta_0:=\frac{J_0}{b_0u_s}.
\end{equation*}
By using \eqref{re-functions} and \eqref{-E-rewritten}, we have
\begin{equation}\label{rho-rho}
\frac{\bar{\rho}'}{\bar{\rho}}=\frac{\sqrt{2h_0J_0^{\frac{2\gamma}{\gamma+1}}\mathcal{F}(\kappa)+2k_0}}{h_0^2J_0^{\frac{2(\gamma-1)}{\gamma+1}}\left(\kappa^{2}-\kappa^{-\gamma+1}\right)}.
\end{equation}
And, for $\bar{\rho}_{\mathfrak{a}}(x_1):=\bar{\rho}(x_1+\mathfrak{a})$, we have
\begin{equation}\label{diff-M}
\begin{split}
\Ra'
=-\frac{\gamma+1}{2}\frac{\sqrt{2h_0J_0^{\frac{2\gamma}{\gamma+1}}\mathcal{F}(\kappa_{\mathfrak{a}})+2k_0}}{h_0^2J_0^{\frac{2(\gamma-1)}{\gamma+1}}\left(\kappa_{\mathfrak{a}}^{2}-\kappa_{\mathfrak{a}}^{-\gamma+1}\right)}\frac{1}{K}\kappa_{\mathfrak{a}}^{\frac{\gamma+1}{2}}.
\end{split}
\end{equation}
Then, by using \eqref{re-functions}-\eqref{diff-M}, we get
\begingroup
\allowdisplaybreaks
\begin{align*}
&N_1=-\frac{\sqrt{2h_0J_0^{\frac{2\gamma}{\gamma+1}}\mathcal{F}(\kappa)+2k_0}\left(\gamma \kappa^2+\kappa^{-\gamma+1}\right)}{h_0^2J_0^{\frac{2(\gamma-1)}{\gamma+1}}\left(\kappa^{2}-\kappa^{-\gamma+1}\right)^2}\frac{1}{K^{\eta}}\kappa_{\mathfrak{a}}^{\frac{(\gamma+1)\eta}{2}}\\
&N_2=-\frac{2\kappa^2}{h_0^2J_0^{\frac{2(\gamma-1)}{\gamma+1}}\left(\kappa^{2}-\kappa^{-\gamma+1}\right)^2}\frac{1}{K^{2\eta}}\kappa_{\mathfrak{a}}^{(\gamma+1)\eta}\\
&N_3=-2{\kappa^{-\gamma+2}}\frac{\left[(\gamma-1)\kappa\sqrt{2h_0J_0^{\frac{2\gamma}{\gamma+1}}\mathcal{F}(\kappa)+2k_0}\right]^2}{h_0^3J_0^{\frac{4\gamma-2}{\gamma+1}}\left(\kappa^{2}-\kappa^{-\gamma+1}\right)^4}\frac{1}{K^{2\eta}}\kappa_{\mathfrak{a}}^{(\gamma+1)\eta}\\
&N_4=-2h_0^{-1}J_0^{\frac{2}{\gamma+1}}\kappa^{\gamma}\\
&P_1=\eta\frac{\gamma+1}{4} \frac{\sqrt{2h_0J_0^{\frac{2\gamma}{\gamma+1}}\mathcal{F}(\kappa_{\mathfrak{a}})+2k_0}}{h_0^2J_0^{\frac{2(\gamma-1)}{\gamma+1}}\left(\kappa_{\mathfrak{a}}^{2}-\kappa_{\mathfrak{a}}^{-\gamma+1}\right)}\frac{1}{K^{\eta}}\kappa_{\mathfrak{a}}^{\frac{(\gamma+1)\eta}{2}}\\
&P_2=\frac{4(\gamma-1)\kappa^2\sqrt{2h_0J_0^{\frac{2\gamma}{\gamma+1}}\mathcal{F}(\kappa)+2k_0}}{h_0^2J_0^{\frac{2(\gamma-1)}{\gamma+1}}\left(\kappa^{2}-\kappa^{-\gamma+1}\right)^2}\frac{1}{K^{\eta}}\kappa_{\mathfrak{a}}^{\frac{(\gamma+1)\eta}{2}}.\\
\end{align*}
\endgroup

By using the fact that $P_1$ is strictly positive, we shall fix positive constants $\mu$, $\eta$, $d$ and $\mfrak{a}$ to satisfy
\begin{equation*}
\frac{P_1}{4}+N_k\ge\frac{\mu}{4}  \quad\tx{for $k=1,2,3,4$},
\end{equation*}
from which it directly follows that
\begin{equation*}
\alp(x_1)\ge P_1+\sum_{k=1}^4 N_k\ge \mu\quad\tx{for $0\le x_1\le L$}.
\end{equation*}

For each $k=1,2,3,4$, we rewrite the term $\frac{P_1}{4}+N_k$ as
\begin{equation*}\label{po-est}
\left\{\begin{split}
\frac{P_1(x_1)}{4}+N_k(x_1)
=&\mfrak{r}(x_1)(\eta\mfrak{b}(x_1)-\beta_k(x_1))\quad\mbox{for }k=1,2,3,\\
\frac{P_1(x_1)}{4}+N_4(x_1)
=&\eta\mfrak{c}(x_1)-\beta_4(x_1)
\end{split}\right.
\end{equation*}
for
\begingroup
\allowdisplaybreaks
\begin{align*}
\mfrak{r}(x_1):=&\,\frac{1}{h_0^2J_0^{\frac{2(\gamma-1)}{\gamma+1}}}\frac{1}{K^{\eta}}\kappa_{\mathfrak{a}}^{\frac{(\gamma+1)\eta}{2}}(x_1),\\
\mfrak{b}(x_1):=&\,\frac{\gamma+1}{16} \frac{\sqrt{2h_0J_0^{\frac{2\gamma}{\gamma+1}}\mathcal{F}(\kappa_{\mathfrak{a}}(x_1))+2k_0}}{\left(\kappa_{\mathfrak{a}}^{2}(x_1)-\kappa_{\mathfrak{a}}^{-\gamma+1}(x_1)\right)}\\
\mfrak{c}(x_1):=&\,\frac{\gamma+1}{16} \frac{\sqrt{2h_0J_0^{\frac{2\gamma}{\gamma+1}}\mathcal{F}(\kappa_{\mathfrak{a}}(x_1))+2k_0}}
{h_0^2J_0^{\frac{2(\gamma-1)}{\gamma+1}}\left(\kappa_{\mathfrak{a}}^{2}(x_1)-\kappa_{\mathfrak{a}}^{-\gamma+1}(x_1)\right)}\frac{1}{K^{\eta}}\kappa_{\mathfrak{a}}^{\frac{(\gamma+1)\eta}{2}}(x_1),\\
\beta_1(x_1):=&\,\frac{\sqrt{2h_0J_0^{\frac{2\gamma}{\gamma+1}}\mathcal{F}(\kappa(x_1))+2k_0}\left(\gamma \kappa^2(x_1)+\kappa^{-\gamma+1}(x_1)\right)}{\left(\kappa^{2}(x_1)-\kappa^{-\gamma+1}(x_1)\right)^2}\\
\beta_2(x_1):=&\,\frac{2\kappa^2(x_1)}{\left(\kappa^{2}(x_1)-\kappa^{-\gamma+1}(x_1)\right)^2}\frac{1}{K^{\eta}}\kappa_{\mathfrak{a}}^{\frac{(\gamma+1)\eta}{2}}(x_1)\\
\beta_3(x_1):=&\,2{\kappa^{-\gamma+2}}\frac{\left[(\gamma-1)\kappa(x_1)\sqrt{2h_0J_0^{\frac{2\gamma}{\gamma+1}}\mathcal{F}(\kappa(x_1))+2k_0}\right]^2}{h_0J_0^{\frac{2\gamma}{\gamma+1}}\left(\kappa^{2}(x_1)-\kappa^{-\gamma+1}(x_1)\right)^4}\frac{1}{K^{\eta}}\kappa_{\mathfrak{a}}^{\frac{(\gamma+1)\eta}{2}}(x_1)\\
\beta_4(x_1):=&\,2h_0^{-1}J_0^{\frac{2}{\gamma+1}}\kappa^{\gamma}(x_1).
\end{align*}
\endgroup

It follows from the monotonicity of $\bar{u}(x_1)$ and the fact of $L^*\le \bar{L}$ that there exists a constant $\delta_0>0$ satisfying that
\begin{equation*}
1+\delta_0\le \kappa_{\mathfrak{a}}(x_1)<\kappa(x_1)\le \frac{u_0}{u_s}\quad\tx{for $x_1\in[0,L]$}.
\end{equation*}
For the constant $K$ given by \eqref{definition of K}, it holds that $K=\kappa^{\frac{\gamma+1}{2}}(\mathfrak{a})=\kappa^{\frac{\gamma+1}{2}}_\mathfrak{a}(0)$ so we have
\begin{equation}
\label{monotonicity of kappa_a}
0<\frac{\kappa_{\mathfrak{a}}^{\frac{\gamma+1}{2}}(L)}{K}\le \frac{\kappa_{\mathfrak{a}}^{\frac{\gamma+1}{2}}(x_1)}{K}\le 1\mbox{ for }x_1\in[0,L].
\end{equation}

For simplicity, let us set $$\lambda_0:=h_0^2J_0^{\frac{2(\gamma-1)}{\gamma+1}}.$$
Obviously, for each $k=1,2,3$, we have
\begin{equation}
\label{mt-1}
\mfrak{r}(\eta\mfrak{b}-\beta_k)(x_1)
\ge \frac 14\Ra^{\eta}(x_1)\quad\tx{for $0\le x_1\le L$}
\end{equation}
if and only if
\begin{equation*}
\eta\ge \max_{x_1\in[0,1]}\frac{1}{\mfrak{b}}\left(\frac{\lambda_0}{4}+\beta_k\right)\quad\tx{for each $k=1,2,3$}.
\end{equation*}

Due to \eqref{monotonicity of kappa_a}, we can achieve \eqref{mt-1} if $\eta$ and $L$ are fixed to satisfy
\begin{equation*}
  \eta\ge
  \max_{x_1\in[0,L]}\left\{\frac{1}{\mfrak{b}}\left(\frac{\lambda_0}{4}+\beta_1\right),
  \frac{1}{\mfrak{b}}\left(\frac{\lambda_0}{4}+\beta_2^*\right),
  \frac{1}{\mfrak{b}}\left(\frac{\lambda_0}{4}+\beta_3^*\right)\right\}
\end{equation*}  for $\beta_2^*$ and $\beta_3^*$ given by
\begin{equation*}
\begin{split}
&{\beta_2}^{\ast}:=\frac{2\kappa^2}{\left(\kappa^{2}-\kappa^{-\gamma+1}\right)^2}\\
&{\beta_3}^{\ast}:=\frac{2(\gamma-1)^2\kappa^{-\gam+4}\left({2h_0J_0^{\frac{2\gamma}{\gamma+1}}\mathcal{F}(\kappa)+2k_0}\right)}{h_0J_0^{\frac{2\gamma}{\gamma+1}}\left(\kappa^{2}-\kappa^{-\gamma+1}\right)^4}.
\end{split}
\end{equation*}

Next, let us set
\begin{equation*}
{\beta}_4^*:=2h_0^{-1}J_0^{\frac{2}{\gamma+1}}\kappa^{\gamma}(0).
\end{equation*}
It follows from the monotonicity of $\bar u$ for $0<x_1<L$ that
\begin{equation}\label{pre-mt2}
\eta\mfrak{c}(x_1)-\beta_4(x_1)\ge \eta\mfrak{c}(x_1)-{\beta_4}^*\quad\tx{for $0\le x_1\le L$}.
\end{equation}
Note that the value of the function $\mfrak{c}(x_1)$ solely depends on the value of $\kappa_{\mfrak a}(x_1)$. Furthermore, $\mfrak{c}$ has a continuous dependence on the value of $\kappa_{\mfrak{a}}(x_1)$ for all $x_1\in[0,L]$. At $x_1=0$ where $\kappa_{\mfrak{a}}(0)=\frac{\bar u(\mfrak{a})}{u_s}$, we have
\begin{equation*}
\mfrak{c}(0)=\frac{\gamma+1}{16} \frac{\sqrt{2h_0J_0^{\frac{2\gamma}{\gamma+1}}\mathcal{F}(\kappa_{\mathfrak{a}}(0))+2k_0}}{h_0^2J_0^{\frac{2(\gamma-1)}{\gamma+1}}\left(\kappa_{\mathfrak{a}}^{2}(0)-\kappa_{\mathfrak{a}}^{-\gamma+1}(0)\right)}=:\mfrak{c}_0>0.
\end{equation*}
Now we fix a constant $\eta$ to satisfy
\begin{equation*}
\eta\ge
  \max_{x_1\in[0,L]}\left\{\mathfrak{l}_0,\frac{1}{\mfrak{b}}\left(\frac{\lambda_0}{4}+\beta_1\right),
  \frac{1}{\mfrak{b}}\left(\frac{\lambda_0}{4}+\beta_2^*\right),
  \frac{1}{\mfrak{b}}\left(\frac{\lambda_0}{4}+\beta_3^*\right), \frac{1}{\mfrak{c}_0}\left(\frac{1}{4}+\beta_4^{\ast}\right)\right\}
\end{equation*}
for $\mathfrak{l}_0$ defined by \eqref{mathfrakl0}. Note that the choice of $\eta$ depends only on the data and $\mfrak{a}$. By choosing the constant $\eta$ as in the above, we have achieved the estimate \eqref{mt-1},
\begin{equation*}
\der_1(\bar a_{22}\mfrak{M})\ge \Ra^{\eta}\quad\tx{and}\quad \eta \mfrak{c}(0)-\beta_4^*\ge \frac 14.
\end{equation*}
Finally, we can fix a constant $d\in(0,1)$ sufficiently small depending only on the data, $\mfrak{a}$ and $\eta$ so that if the condition \eqref{range of u} holds, then we derive from \eqref{pre-mt2} and the continuous dependence of $\mfrak{c}$ on $\kappa_{\mfrak{a}}$ that
\begin{equation*}
  \eta\mfrak{c}(x_1)-\beta_4(x_1)\ge \frac 18\quad\tx{for $0\le x_1\le L$}.
\end{equation*}
Note that the choice of $\eta$ depends only on the data and $\mfrak{a}$, therefore we conclude that the choice of the small constant $d$ depends only on the data and $\mfrak{a}$ eventually.

We have chosen the constant $d$ so that if all the conditions stated in Lemma \ref{App-lemma} are satisfied, then we establish the estimate \eqref{condition-end-2} for the constant $\mu$ given by
\begin{equation*}
  \mu=\min_{x_1\in[0, L]}\left\{\Ra^{\eta}(x_1), \frac 34 \Ra^{\eta}(x_1)+\frac 18\right\}=\min\left\{\Ra^{\eta}(L), \frac 34 \Ra^{\eta}(L)+\frac 18\right\}>0.
\end{equation*}

\end{proof}


\vspace{.25in}
\noindent
{\bf Acknowledgements:}
The research of Myoungjean Bae was supported in part by  Samsung Science and Technology Foundation under Project Number SSTF-BA1502-51.
The research of Hyangdong Park was supported in part by the POSCO Science Fellowship of POSCO TJ Park Foundation and a KIAS Individual Grant (MG086701) at Korea Institute for Advanced Study.

\bigskip
\bibliographystyle{siam}
\bibliography{EP-super-references}

\end{document}